\documentclass[11pt,a4paper]{amsart}
\usepackage[all]{xy}
\SelectTips{eu}{}
\usepackage{ifthen}
\usepackage{amssymb}
\calclayout 

\title{
Localization theory for triangulated categories} 
\author[Henning Krause]{Henning Krause}
\address{Henning Krause\\ Institut f\"ur Mathematik\\
Universit\"at Paderborn\\ 33095 Paderborn\\ Germany.}
\email{hkrause@math.upb.de}

\theoremstyle{plain}
\newtheorem{lem}{Lemma}[subsection]
\newtheorem{Lem}{Lemma}[section] 
\newtheorem{prop}[lem]{Proposition}
\newtheorem{cor}[lem]{Corollary}
\newtheorem{thm}[lem]{Theorem}
\newtheorem*{Thm}{Theorem}

\theoremstyle{remark}

\theoremstyle{definition}
\newtheorem{rem}[lem]{Remark}
\newtheorem{exm}[lem]{Example}

\numberwithin{equation}{subsection}

\newcommand{\smatrix}[1]{\left[\begin{smallmatrix}#1\end{smallmatrix}\right]}

\renewcommand{\mod}{\operatorname{mod}\nolimits}
\newcommand{\Ob}{\operatorname{Ob}\nolimits}
\newcommand{\Mor}{\operatorname{Mor}\nolimits}

\newcommand{\colim}[1]{\mathop{\underset{#1}
  {\underrightarrow{\mathrm{colim}}}}}

\newcommand{\Add}{\operatorname{Add}\nolimits}

\newcommand{\Mod}{\operatorname{Mod}\nolimits}

\newcommand{\Hom}{\operatorname{Hom}\nolimits}
\newcommand{\RHom}{\operatorname{{\bfR}Hom}\nolimits}
\newcommand{\Ho}{\operatorname{Ho}\nolimits}

\renewcommand{\Im}{\operatorname{Im}\nolimits}
\newcommand{\Ker}{\operatorname{Ker}\nolimits}
\newcommand{\Coker}{\operatorname{Coker}\nolimits}

\newcommand{\Ext}{\operatorname{Ext}\nolimits}
\newcommand{\Tor}{\operatorname{Tor}\nolimits}

\newcommand{\Lex}{\operatorname{Lex}\nolimits}

\newcommand{\Loc}{\operatorname{Loc}\nolimits}
\newcommand{\Coloc}{\operatorname{Coloc}\nolimits}

\newcommand{\Id}{\operatorname{Id}\nolimits} 
\newcommand{\id}{\operatorname{id}\nolimits} 
\newcommand{\card}{\operatorname{card}\nolimits} 

\newcommand{\Ab}{\mathrm{Ab}} 

\newcommand{\op}{\mathrm{op}}

\newcommand{\inc}{\mathrm{inc}}
\newcommand{\can}{\mathrm{can}}
\newcommand{\qis}{\mathrm{qis}}
\newcommand{\ac}{\mathrm{ac}}

\newcommand{\comp}{\mathop{\circ}}
\newcommand{\lto}[1][{}]{\stackrel{#1}{\longrightarrow}} 
\renewcommand{\to}[1][{}]{\stackrel{#1}{\rightarrow}} 
\newcommand{\xto}{\xrightarrow}

\def\a{\alpha}
\def\b{\beta}

\def\d{\delta}
\def\g{\gamma}
\def\p{\phi}
\def\r{\rho}
\def\s{\sigma}
\def\t{\tau}

\def\m{\mu}
\def\n{\nu}
\def\k{\kappa}
\def\la{\lambda}

\def\Ga{\varGamma}
\def\La{\Lambda}
\def\Si{\Sigma}

\def\A{{\mathcal A}}

\def\C{{\mathcal C}}
\def\D{{\mathcal D}}

\def\I{{\mathcal I}}

\def\P{{\mathcal P}}

\def\Sc{{\mathcal S}}
\def\X{{\mathcal X}}

\def\T{{\mathcal T}}
\def\U{{\mathcal U}}
\def\V{{\mathcal V}}

\def\bbN{\mathbb N}

\def\bbQ{\mathbb Q}

\def\bbZ{\mathbb Z}

\def\bfD{\mathbf D}

\def\bfK{\mathbf K}
\def\bfL{\mathbf L}

\def\bfR{\mathbf R}

\begin{document}

\maketitle 

\setcounter{tocdepth}{1}
\tableofcontents

\section{Introduction}

These notes provide an introduction to the theory of localization for
triangulated categories. Localization is a machinery to formally
invert morphisms in a category.  We explain this formalism in some
detail and we show how it is applied to triangulated categories.

There are basically two ways to approach the localization theory for
triangulated categories and both are closely related to each other. To
explain this, let us fix a triangulated category $\T$. The first
approach is \emph{Verdier localization}. For this one chooses a full
triangulated subcategory $\Sc$ of $\T$ and constructs a universal exact
functor $\T\to\T/\Sc$ which annihilates the objects belonging to $\Sc$.
In fact, the quotient category $\T/\Sc$ is obtained by formally
inverting all morphisms $\s$ in $\T$ such that the cone of $\s$
belongs to $\Sc$. 

On the other hand, there is \emph{Bousfield
localization}.  In this case one considers an exact functor
$L\colon\T\to\T$ together with a natural morphism $\eta X\colon X\to
LX$ for all $X$ in $\T$ such that $L(\eta X)=\eta (LX)$ is
invertible. There are two full triangulated subcategories arising from
such a localization functor $L$. We have the subcategory $\Ker L$
formed by all $L$-acyclic objects, and we have the essential image
$\Im L$ which coincides with the subcategory formed by all $L$-local
objects. Note that $L$, $\Ker L$, and $\Im L$ determine each
other. Moreover, $L$ induces an equivalence $\T/\Ker L\xto{\sim}\Im
L$. Thus a Bousfield localization functor $\T\to\T$ is nothing but the
composite of a Verdier quotient functor $\T\to\T/\Sc$ with a fully
faithful right adjoint $\T/\Sc\to\T$.

Having introduced these basic objects, there are a number of immediate
questions. For example, given a triangulated subcategory $\Sc$ of $\T$,
can we find a localization functor $L\colon\T\to\T$ satisfying $\Ker
L=\Sc$ or $\Im L=\Sc$? On the other hand, if we start with $L$, which
properties of $\Ker L$ and $\Im L$ are inherited from $\T$? It turns
out that well generated triangulated categories in the sense of Neeman
\cite{Nee2001} provide an excellent setting for studying these
questions.

Let us discuss briefly the relevance of well generated categories. The
concept generalizes that of a compactly generated triangulated
category. 
For example, the derived category of unbounded chain complexes of
modules over some fixed ring is compactly generated. Also, the stable
homotopy category of CW-spectra is compactly generated. Given any
localization functor $L$ on a compactly generated triangulated
category, it is rare that $\Ker L$ or $\Im L$ are compactly
generated. However, in all known examples $\Ker L$ and $\Im L$ are
well generated. The following theorem provides a conceptual
explanation; it combines several results from Section~\ref{se:loc-wellgen}.

\begin{Thm}
Let $\T$ be a well generated triangulated category and $\Sc$ a full
triangulated subcategory which is closed under small coproducts. Then
the following are equivalent.
\begin{enumerate}
\item The triangulated category $\Sc$ is well generated.
\item The  triangulated category $\T/\Sc$ is well  generated.
\item There exists a cohomological functor $H\colon\T\to\A$ into a
locally presentable abelian category such that $H$ preserves small
coproducts and $\Sc=\Ker H$.
\item There exists a small set $\Sc_0$ of objects in $\Sc$ such that
$\Sc$ admits no proper full triangulated subcategory closed under small
coproducts and containing $\Sc_0$.
\end{enumerate}
Moreover, in this case there exists a localization functor
$L\colon\T\to\T$ such that $\Ker L=\Sc$.
\end{Thm}
Note that every abelian Grothendieck category is locally presentable;
in particular every module category is locally presentable.

Our approach for studying localization functors on well generated
triangulated categories is based on the interplay between triangulated
and abelian structure. A well known construction due to Freyd provides
for any triangulated category $\T$ an abelian category $A(\T)$
together with a universal cohomological functor $\T\to
A(\T)$. However, the category $A(\T)$ is usually far too big and
therefore not manageable. If $\T$ is well generated, then we have a
canonical filtration $$A(\T)=\bigcup_\a A_\a(\T)$$ indexed by all
regular cardinals, such that for each $\a$ the category $A_\a(\T)$ is
abelian and locally $\a$-presentable in the sense of Gabriel and Ulmer
\cite{GU}. Moreover, each inclusion $A_\a(\T)\to A(\T)$ admits an
exact right adjoint and the composite
$$H_\a\colon\T\lto A(\T)\lto A_\a(\T)$$ is the universal cohomological
functor into a locally $\a$-presentable abelian category. Thus we may
think of the functors $\T\to A_\a(\T)$ as successive approximations of
$\T$ by locally presentable abelian categories. For instance, there
exists for each object $X$ in $\T$ some cardinal $\a(X)$ such that the
induced map $\T(X,Y)\to A_\b(\T)(H_\b X,H_\b Y)$ is bijective for all
$Y$ in $\T$ and all $\b\geq\a(X)$.

These notes are organized as follows. We start off with an
introduction to categories of fractions and localization functors for
arbitrary categories. Then we apply this to triangulated
categories. First we treat arbitrary triangulated categories and
explain the localization in the sense of Verdier and Bousfield. Then
we pass to compactly and well generated triangulated categories where
Brown representability provides an indispensable tool for constructing
localization functors. Module categories and their derived categories
are used to illustrate most of the concepts; see \cite{Dw} for
complementary material from topology. The results on well generated
categories are based on facts from the theory of locally presentable
categories; we have collected these in a separate appendix.

\subsection*{Acknowledgement}
The plan to write an introduction to the theory of triangulated
localization took shape during the ``Workshop on Triangulated
Categories'' in Leeds 2006. I wish to thank the organizers Thorsten
Holm, Peter J{\o}rgensen, and Rapha\"el Rouquier for their skill and
diligence in organizing this meeting.  Most of these notes were then
written during a three months stay in 2007 at the Centre de Recerca
Matem\`atica in Barcelona as a participant of the special program
``Homotopy Theory and Higher Categories''.  I am grateful to the
organizers Carles Casacuberta, Joachim Kock, and Amnon Neeman for
creating a stimulating atmosphere and for several helpful discussions.
Finally, I would like to thank Xiao-Wu Chen, Daniel Murfet, and Jan
\v{S}\v{t}ov\'i\v{c}ek for their helpful comments on a preliminary
version of these notes.

\section{Categories of fractions and localization functors}

\subsection{Categories}
Throughout we fix a universe of sets in the sense of Grothendieck
\cite{GV}.  The members of this universe will be called \emph{small
sets}.

Let $\C$ be a category. We denote by $\Ob\C$ the set of objects and by
$\Mor\C$ the set of morphisms in $\C$. Given objects $X,Y$ in $\C$,
the set of morphisms $X\to Y$ will be denoted by $\C(X,Y)$.  The
identity morphism of an object $X$ is denoted by $\id_\C X$ or just
$\id X$. If not stated otherwise, we always assume that the morphisms
between two fixed objects of a category form a small set.

A category $\C$ is called \emph{small} if the isomorphism classes of
objects in $\C$ form a small set. In that case we define the
\emph{cardinality} of $\C$ as $\card\C=\sum_{X,Y\in\C_0}\card\C(X,Y)$
where $\C_0$ denotes a representative set of objects of $\C$, meeting
each isomorphism class exactly once.

Let $F\colon\I\to \C$ be a functor from a  small (indexing)
category $\I$ to a category $\C$.  Then we write $\colim{i\in\I} Fi$
for the colimit of $F$, provided it exists. Given a cardinal $\a$, the
colimit of $F$ is called \emph{$\a$-colimit} if $\card\I<\a$.  An
example of a colimit is the coproduct $\coprod_{i\in I}X_i$ of a family
$(X_i)_{i\in I}$ of objects in $\C$ where the indexing set $I$ is
always assumed to be small.  We say that a category $\C$ \emph{admits
small coproducts} if for every family $(X_i)_{i\in I}$ of objects in
$\C$ which is indexed by a small set $I$ the coproduct $\coprod_{i\in
I}X_i$ exists in $\C$. Analogous terminology is used for limits and products.

\subsection{Categories of fractions}

Let $F\colon\C\to\D$ be a functor.  We say that $F$ makes a morphism
$\s$ of $\C$ invertible if $F\s$ is invertible.  The set of all
those morphisms which $F$ inverts is denoted by $\Si(F)$.

Given a category $\C$ and any set $\Si$ of morphisms of $\C$, we
consider the {\em category of fractions} $\C[\Si^{-1}]$ together with
a canonical \emph{quotient functor}
$$Q_\Si\colon\C\lto\C[\Si^{-1}]$$ having the following properties.
\begin{enumerate}
\item[(Q1)] $Q_\Si$ makes the morphisms in $\Si$ invertible.
\item[(Q2)] If a functor $F\colon\C\to\D$ makes the morphisms in $\Si$
invertible, then there is a unique functor $\bar F\colon \C[\Si^{-1}]\to\D$
such that $F=\bar F\comp Q_\Si$.
\end{enumerate}

Note that $\C[\Si^{-1}]$ and $Q_\Si$ are essentially unique if they
exists.  Now let us sketch the construction of $\C[\Si^{-1}]$ and
$Q_\Si$. At this stage, we ignore set-theoretic issues, that is, the
morphisms between two objects of $\C[\Si^{-1}]$ need not to form a
small set.  We put $\Ob\C[\Si^{-1}]=\Ob\C$. To define the morphisms of
$\C[\Si^{-1}]$, consider the quiver (i.e.\ oriented graph) with set of
vertices $\Ob\C$ and with set of arrows the disjoint union
$(\Mor\C)\amalg\Si^{-1}$, where $\Si^{-1}=\{\s^{-1}\colon Y\to
X\mid\Si\ni\s\colon X\to Y\}$.  Let $\P$ be the set of paths in this
quiver (i.e.\ finite sequences of composable arrows), together with
the obvious composition which is the concatenation operation and
denoted by $\comp_\P$. We define $\Mor\C[\Si^{-1}]$ as the quotient of
$\P$ modulo the following relations:
\begin{enumerate}
\item $\b\comp_\P\a=\b\comp\a$ for all composable morphisms $\a,\b\in\Mor\C$.
\item $\id_\P X=\id_\C X$ for all $X\in\Ob\C$.
\item $\s^{-1}\comp_\P\s=\id_\P X$ and $\s\comp_\P\s^{-1}=\id_\P Y$ for all
$\s\colon X\to Y$ in $\Si$.
\end{enumerate}
The composition in $\P$ induces the composition of morphisms in
$\C[\Si^{-1}]$. The functor $Q_\Si$ is the identity on objects and on
$\Mor\C$ the composite
$$\Mor\C\xto{\inc}
(\Mor\C)\amalg\Si^{-1}\xto{\inc}\P\xto{\can}\Mor\C[\Si^{-1}].$$

Having completed the construction of the category of fractions
$\C[\Si^{-1}]$, let us mention that it is also called \emph{quotient
category} or \emph{localization} of $\C$ with respect to $\Si$.

\subsection{Adjoint functors}

Let $F\colon\C\to\D$ and $G\colon\D\to\C$ be a pair of
functors and assume that $F$ is left adjoint to $G$. We denote by 
$$\theta\colon F\comp G\to \Id\D\quad \text{and}\quad\eta\colon\Id\C\to
G\comp F$$ the corresponding adjunction morphisms. Let $\Si=\Si(F)$ denote
the set of morphisms $\s$  of $\C$ such that $F\s$ is invertible.
Recall that a morphism $\mu\colon F\to F'$ between two functors is
invertible if for each object $X$ the morphism $\mu X\colon FX\to F'X$
is invertible.

\begin{prop}\label{pr:quot}
The following statements are equivalent.
\begin{enumerate}
\item The functor $G$ is fully faithful.
\item The morphism $\theta\colon F\comp G\to \Id\D$ is invertible.
\item The functor $\bar F\colon\C[\Si^{-1}]\to\D$ satisfying $F=\bar
F\comp Q_\Si$ is an equivalence.
\end{enumerate}
\end{prop}
\begin{proof}
See \cite[I.1.3]{GZ}.
\end{proof}

\subsection{Localization functors}

A functor $L\colon \C\to \C$ is called a \emph{localization functor} if
there exists a morphism $\eta\colon\Id\C\to L$ such that $L\eta\colon
L\to L^2$ is invertible and $L\eta=\eta L$. Note that we only require
the existence of $\eta$; the actual morphism is not part of the
definition of $L$. However, we will see that $\eta$ is determined by
$L$, up to a unique isomorphism $L\to L$.

\begin{prop}\label{pr:loc}
Let $L\colon\C\to\C$ be a functor and $\eta\colon\Id\C\to L$ be a morphism.
Then the following are equivalent.
\begin{enumerate}
\item $L\eta\colon L\to L^2$ is invertible and $L\eta=\eta L$.
\item There exists a functor $F\colon\C\to\D$ and a fully faithful
right adjoint $G\colon\D\to\C$ such that $L=G\comp F$ and
$\eta\colon\Id\C\to G\comp F$ is the adjunction morphism.
\end{enumerate}
\end{prop}
\begin{proof}
(1) $\Rightarrow$ (2): Let $\D$ denote the full subcategory of $\C$
    formed by all objects $X$ such that $\eta X$ is invertible. For each
    $X\in\D$, let $\theta X\colon LX\to X$ be the inverse of $\eta
    X$. Define $F\colon \C\to\D$ by $FX=LX$ and let $G\colon \D\to\C$
    be the inclusion. We claim that $F$ and $G$ form an adjoint pair. In
    fact, it is straightforward to check that the maps
$$\D(FX,Y)\lto \C(X,GY),\quad \a\mapsto G\a\comp\eta X,$$
and $$\C(X,GY)\lto\D(FX,Y),\quad \b\mapsto \theta Y\comp F\b,$$
are mutually inverse bijections.

(2) $\Rightarrow$ (1): Let $\theta\colon FG\to\Id\D$ denote the second
adjunction morphism. Then the composites $$F\xto{F\eta}FGF\xto{\theta
  F} F\quad\text{and}\quad G\xto{\eta G}GFG\xto{G\theta}G$$ are
identity morphisms; see \cite[IV.1]{MacL}. We know from
Proposition~\ref{pr:quot} that $\theta$ is invertible because $G$ is
fully faithful.  Therefore $L\eta=GF\eta$ is invertible. Moreover, we
have
\[L\eta=GF\eta=(G\theta F)^{-1}=\eta GF=\eta L.\qedhere\]
\end{proof}

\begin{cor}\label{co:loc}
A functor $L\colon \C\to\C$ is a localization functor if and only if
there exists a functor $F\colon\C\to\D$ and a fully faithful right
adjoint $G\colon\D\to\C$ such that $L=G\comp F$.  In that case there
exist a unique equivalence $\C[\Si^{-1}]\to \D$ making the following
diagram commutative
$$
\xymatrixrowsep{.5pc}
\xymatrix{&&\C[\Si^{-1}]\ar[dd]^-\sim\ar[rrd]^-{\bar L}\\
\C\ar[rru]^-{Q_{\Si}}\ar[rrd]_-{F}&&&&\C\\ &&\D\ar[rru]_-{G}
}$$
where $\Si$ denotes the set of morphisms $\s$ in $\C$ such that $L\s$
is invertible. 
\end{cor}
\begin{proof}
The characterization of a localization functor follows from
Proposition~\ref{pr:loc}.  Now observe that $\Si$ equals the set of
morphisms $\s$ in $\C$ such that $F\s$ is invertible since $G$ is
fully faithful. Thus  we can apply Proposition~\ref{pr:quot} to obtain the
equivalence $\C[\Si^{-1}]\to \D$ making the diagram commutative.
\end{proof}

\subsection{Local objects}

Given a localization functor $L\colon\C\to\C$, we wish to describe
those objects $X$ in $\C$ such that $X\xto{\sim} LX$. To this end, it is
convenient to make the following definition. An object $X$ in a
category $\C$ is called {\em local} with respect to a set $\Si$ of
morphisms if for every morphism $W\to W'$ in $\Si$ the induced map
$\C(W',X)\to\C(W,X)$ is bijective. Now let $F\colon\C\to\D$ be a
functor and let $\Si(F)$ denote the set of morphisms $\s$ of $\C$ such
that $F\s$ is invertible. An object $X$ in $\C$ is called $F$-local if
it is local with respect to $\Si(F)$.

\begin{lem}\label{le:local}
Let $F\colon\C\to\D$ be a functor and $X$ an object of $\C$. Suppose
there are two morphisms $\eta_1\colon X\to Y_1$ and $\eta_2\colon X\to
Y_2$ such that $F\eta_i$ is invertible and $Y_i$ is $F$-local for
$i=1,2$. Then there exists a unique isomorphism $\p\colon Y_1\to Y_2$
such that $\eta_2=\p\comp\eta_1$.
\end{lem}
\begin{proof}
The morphism $\eta_1$ induces a bijection $\C(Y_1,Y_2)\to\C(X,Y_2)$
and we take for $\p$ the unique morphism which is sent to $\eta_2$.
Exchanging the roles of $\eta_1$ and $\eta_2$, we obtain the inverse for $\p$.
\end{proof}

\begin{prop}\label{pr:local-obj}
Let $L\colon\C\to\C$ be a localization functor and $\eta\colon\Id\C\to
L$ a morphism such that $L\eta$ is invertible. Then the following are
equivalent for an object $X$ in $\C$.
\begin{enumerate}
\item The object $X$ is $L$-local.
\item The map $\C(LW,X)\to \C(W,X)$ induced by $\eta W$ is bijective for all $W$ in $\C$.
\item The morphism $\eta X\colon X\to LX$ is invertible.
\item The map $\C(W,X)\to \C(LW,LX)$ induced by $L$ is bijective for all $W$ in $\C$.
\item The object $X$ is isomorphic to $LX'$ for some object $X'$ in $\C$.
\end{enumerate}
\end{prop}
\begin{proof} 
(1) $\Rightarrow$ (2): The morphism $\eta W$ belongs to $\Si(L)$ and
therefore $\C(\eta W,X)$ is bijective if $X$ is $L$-local.

(2) $\Rightarrow$ (3): Put $W=X$. We obtain a morphism $\p\colon LX\to
X$ which is an inverse for $\eta X$. More precisely, we have
$\p\comp\eta X=\id X$. On the other hand,
$$\eta X\comp \p=L\p\comp\eta LX=L\p\comp L\eta X=L(\p\comp\eta X)=\id LX.$$
Thus $\eta X$ is invertible.

(3) $\Leftrightarrow$ (4): We use the factorization
$\C\xto{F}\D\xto{G}\C$ of $L$ from Proposition~\ref{pr:loc}. Then we
obtain for each $W$ in $\C$ a factorization
$$\C(W,X)\lto\C(W,LX)\lto[\sim]\C(FW,FX)\lto[\sim]\C(LW,LX)$$ of the
map $f_W\colon\C(W,X)\to \C(LW,LX)$ induced by $L$. Here, the first map
is induced by $\eta X$, the second follows from the adjunction, and
the third is induced by $G$. Thus $f_W$ is bijective for all $W$ iff the first
map is bijective for all $W$ iff $\eta X$ is invertible.

(3) $\Rightarrow$ (5): Take $X'=X$.

(5) $\Rightarrow$ (1): We use again the factorization
$\C\xto{F}\D\xto{G}\C$ of $L$ from Proposition~\ref{pr:loc}.  Fix $\s$
in $\Si(L)$ and observe that $F\s$ is invertible. Then we have
$\C(\s,X)\cong\C(\s,G(FX'))\cong\D(F\s,FX')$ and this implies that
$\C(\s,X)$ is bijective since $F\s$ is invertible.
\end{proof}

Given a functor $F\colon\C\to\D$, we denote by $\Im F$ the
\emph{essential image} of $F$, that is, the full subcategory of $\D$
which is formed by all objects isomorphic to $FX$ for some object $X$
in $\C$.

\begin{cor}
\label{co:local-objects}
Let $L\colon\C\to\C$ be a localization functor. Then $L$ induces an
equivalence $\C[\Si(L)^{-1}]\xto{\sim}\Im L$ and $\Im L$ is the full
subcategory of $\C$ consisting of all $L$-local subobjects.
\end{cor}
\begin{proof}
Write $L$ as composite $\C\xto{F}\Im L\xto{G}\C$ of two functors,
where $FX=LX$ for all $X$ in $\C$ and $G$ is the inclusion
functor. Then it follows from Corollary~\ref{co:loc} that $F$ induces
an equivalence $\C[\Si(L)^{-1}]\xto{\sim}\Im L$. The second assertion is an
immediate consequence of Proposition~\ref{pr:local-obj}.
\end{proof}

Given a localization functor $L\colon\C\to\C$ and an object $X$ in
$\C$, the morphism $X\to LX$ is initial among all morphisms to an
object in $\Im\D$ and terminal among all morphisms in $\Si(L)$. The
following statement makes this precise.

\begin{cor}
Let $L\colon\C\to\C$ be a localization functor and $\eta\colon\Id\C\to
L$ a morphism such that $L\eta$ is invertible. Then for each morphism
$\eta X\colon X\to LX$ the following holds.
\begin{enumerate}
\item The object $LX$ belongs to $\Im L$ and every morphism $X\to Y$
with $Y$ in $\Im L$ factors uniquely through $\eta X$.
\item The morphism $\eta X$ belongs to $\Si(L)$ and factors uniquely
through every morphism $X\to Y$ in $\Si(L)$.
\end{enumerate}
\end{cor}
\begin{proof}
Apply Proposition~\ref{pr:local-obj}.
\end{proof}

\begin{rem}
(1) Let $L\colon\C\to\C$ be a localization functor and suppose there
    are two morphisms $\eta_i\colon\Id\C\to L$ such that $L\eta_i$ is
    invertible for $i=1,2$.  Then there exists a unique isomorphism
    $\p\colon L\xto{\sim} L$ such that $\eta_2=\p\comp\eta_1$. This
    follows from Lemma~\ref{le:local}.

(2) Given any functor $F\colon\C\to\D$, the full subcategory of $F$-local
objects is closed under taking all limits which exist in $\C$.
\end{rem}

\subsection{Existence of localization functors}

We provide a criterion for the existence of a localization functor
$L$; it explains how $L$ is determined by the category of $L$-local
objects.

\begin{prop}
Let $\C$ be a category and $\D$ a full subcategory. Suppose that every
object in $\C$ isomorphic to one in $\D$ belongs to $\D$. Then the
following are equivalent.
\begin{enumerate}
\item There exists a localization functor $L\colon\C\to\C$ with $\Im L=\D$.
\item For every object $X$ in $\C$ there exists a morphism $\eta
X\colon X\to X'$ with $X'$ in $\D$ such that every morphism $X\to Y$
with $Y$ in $\D$ factors uniquely through $\eta X$.
\item The inclusion functor $\D\to\C$ admits a left adjoint.
\end{enumerate}
\end{prop}
\begin{proof}
(1) $\Rightarrow$ (2): Suppose there exists a localization functor
    $L\colon\C\to\C$ with $\Im L=\D$ and let $\eta\colon \Id\C\to L$ be a
    morphism such that $L\eta$ is invertible. Then
    Proposition~\ref{pr:local-obj} shows that $\C(\eta X,Y)$ is
    bijective for all $Y$ in $\D$.

(2) $\Rightarrow$ (3): The morphisms $\eta X$ provide a functor
$F\colon\C\to\D$ by sending each $X$ in $\C$ to $X'$.
It is straightforward to check that $F$ is a left adjoint for the inclusion $\D\to\C$.

(3) $\Rightarrow$ (1): Let $G\colon\D\to\C$ denote the inclusion and
    $F$ its right adjoint. Then $L=G\comp F$ is a localization functor
    with $\Im L=\D$ by Proposition~\ref{pr:loc}.
\end{proof}

\subsection{Localization functors preserving coproducts}

We characterize the fact that a localization functor preserves small coproducts.

\begin{prop}
\label{pr:loc-coprod}
Let $L\colon\C\to\C$ be a localization functor and suppose the
category $\C$ admits small coproducts. Then the following are equivalent.
\begin{enumerate}
\item The functor $L$ preserves small coproducts.
\item The $L$-local objects are closed under taking small coproducts
in $\C$.
\item The right adjoint of the quotient functor $\C\to \C[\Si(L)^{-1}]$
preserves small coproducts.
\end{enumerate}
\end{prop}
\begin{proof}
(1) $\Rightarrow$ (2): Let $(X_i)_{i\in I}$ be a family of $L$-local
    objects. Thus the natural morphisms $X_i\to LX_i$ are invertible
    by Proposition~\ref{pr:local-obj} and they induce an isomorphism
$$\coprod_i X_i\xto{\sim}\coprod_i LX_i\xto{\sim}L(\coprod_i X_i).$$
It follows that $\coprod_i X_i$ is $L$-local.

(2) $\Leftrightarrow$ (3): We can identify $\C[\Si(L)^{-1}]=\Im L$ by
Corollary~\ref{co:local-objects} and then the right adjoint of the
quotient functor identifies with the inclusion $\Im L\to\C$. Thus the
right adjoint preserves small coproducts if and only if the inclusion
$\Im L\to\C$ preserves small coproducts.

(3) $\Rightarrow$ (1): Write $L$ as composite
$\C\xto{}\C[\Si(L)^{-1}]\xto{}\C$ of the quotient functor $Q$ with its
right adjoint $\bar L$. Then $Q$ preserves small coproducts since it
is a left adjoint. It follows that $L$ preserves small coproducts if
$\bar L$ preserves small coproducts.
\end{proof}

\subsection{Colocalization functors}
A functor $\Ga\colon \C\to\C$ is called \emph{colocalization functor}
if its opposite functor $\Ga^\op\colon\C^\op\to\C^\op$ is a
localization functor. We call an object $X$ in $\C$
\emph{$\Ga$-colocal} if it is $\Ga^\op$-local when viewed as an object
of $\C^\op$. Note that a colocalization functor $\Ga\colon\C\to\C$
induces an equivalence
\[\C[\Si(\Ga)^{-1}]\lto[\sim]\Im\Ga\] 
and the essential image $\Im\Ga$ equals the full subcategory of $\C$
consisting of all $\Ga$-colocal objects.

\begin{rem}
We think of $\Ga$ as $L$ turned upside down; this explains our
notation. Another reason for the use of $\Ga$ is the interpretation of
local cohomology as colocalization.
\end{rem}

\subsection{Example: Localization of modules}

Let $A$ be an associative ring and denote by $\Mod A$ the category of
(right) $A$-modules. Suppose that $A$ is commutative and let $S\subseteq
A$ be a multiplicatively closed subset, that is, $1\in S$ and $st\in
S$ for all $s,t\in S$. We denote by
$$S^{-1}A=\{x/s\mid x\in A\text{ and }s\in S\}$$ the ring of
fractions. 
For each $A$-module $M$, let
$$S^{-1}M=\{x/s\mid x\in M\text{ and }s\in S\}$$ be the localized
module. An $S^{-1}A$-module $N$ becomes an $A$-module via restriction
of scalars along the canonical ring homomorpism $A\to S^{-1}A$.  We
obtain a pair of functors
$$F\colon\Mod A\lto\Mod S^{-1}A,\quad M\mapsto S^{-1}M\cong
M\otimes_A S^{-1}A,$$ 
$$G\colon\Mod S^{-1}A\lto\Mod A,\quad N\mapsto
N\cong\Hom_{S^{-1}A}(S^{-1}A,N).$$ Moreover, for each pair of modules
$M$ over $A$ and $N$ over $S^{-1}A$, we have natural morphisms
$$\eta M\colon M\lto (G\comp F)M=S^{-1}M,\quad x\mapsto x/1,$$
$$\theta N\colon S^{-1}N=(F\comp G)N\lto N, \quad x/s\mapsto xs^{-1}.$$
These natural morphisms induce mutually inverse bijections as follows:
$$\Hom_A(M,GN)\lto[\sim]\Hom_{S^{-1}A}(FM,N),\quad \a\mapsto \theta N\comp
F\a,$$
$$\Hom_{S^{-1}A}(FM,N)\lto[\sim]\Hom_A(M,GN),\quad\b\mapsto
G\b\comp\eta M.$$ 

It is clear that the functors $F$ and $G$ form an adjoint pair, that
is, $F$ is a left adjoint of $G$ and $G$ is a right adjoint of
$F$. Moreover, the adjunction morphism $\theta\colon F\comp G\to\Id$ is
invertible. Therefore the composite $L=G\comp F$ is a localization
functor.

Let us formulate this slightly more generally. Fix a ring homomorphism
$f\colon A\to B$. Then it is well known that the restriction functor
$\Mod B\to\Mod A$ is fully faithful if and only if $f$ is an
epimorphism; see \cite[Proposition~XI.1.2]{St}. Thus the functor $\Mod
A\to\Mod A$ taking a module $M$ to $M\otimes_AB$ is a localization
functor provided that $f$ is an epimorphism.

\subsection{Example: Localization of spectra}

A \emph{spectrum} $E$ is a sequence of based topological spaces $E_n$
and based homeomorphisms $E_n\to\Omega E_{n+1}$. A morphism of spectra
$E\to F$ is a sequence of based continuous maps $E_n\to F_n$ strictly
compatible with the given structural homeomorphisms. The homotopy
groups of a spectrum $E$ are the groups $\pi_n E=\pi_{n+i}(E_i)$ for
$i\geq 0$ and $n+i\geq 0$. A morphism between spectra is a \emph{weak
equivalence} if it induces an isomorphism on homotopy groups. The
\emph{stable homotopy category} $\Ho\Sc$ is obtained from the category
$\Sc$ of spectra by formally inverting the weak equivalences. Thus
$\Ho\Sc=\Sc[\Si^{-1}]$ where $\Si$ denotes the set of weak
equivalences. We refer to \cite{Ad,Q} for details.

\subsection{Notes} 
The category of fractions is introduced by Gabriel and Zisman in
\cite{GZ}, but the idea of formally inverting elements can be traced
back much further; see for instance \cite{Ore}. The appropriate
context for localization functors is the theory of monads; see \cite{MacL}.

\section{Calculus of fractions}

\subsection{Calculus of fractions}

Let $\C$ be a category and $\Si$ a set of morphisms in $\C$.  The
category of fractions $\C[\Si^{-1}]$ admits an elementary description
if some extra assumptions on $\Si$ are satisfied. We say that $\Si$
{\em admits a calculus of left fractions} if the following holds.
\begin{enumerate}
\item[(LF1)] If $\s,\t$ are composable morphisms in $\Si$, then $\t\comp \s$
is in $\Si$. The identity morphism $\id X$ is in $\Si$ for all $X$ in $\C$.
\item[(LF2)] Each pair of morphisms  $X'\xleftarrow{\s} X\xto{\a}Y$
with $\s$ in $\Si$ can be completed to a commutative square
$$\xymatrix{X\ar[r]^\a\ar[d]^{\s}&Y\ar[d]^{\s'}\\
X'\ar[r]^{\a'}&Y'}$$
such that $\s'$ is in $\Si$.
\item[(LF3)] Let $\a,\b\colon X\to Y$ be morphisms in $\C$. If there is a morphism
$\s\colon X'\to X$ in $\Si$ with $\a\comp\s=\b\comp\s$, then there
exists a morphism $\t\colon Y\to Y'$ in $\Si$ with $\t\comp\a=\t\comp\b$.
\end{enumerate}

Now assume that $\Si$ admits a calculus of left fractions. Then one obtains a new
category $\Si^{-1}\C$ as follows. The objects are those of $\C$. Given
objects $X$ and $Y$, we call a pair $(\a,\s)$ of morphisms 
$$\xymatrix{X\ar[r]^-\a&Y'&Y\ar[l]_-\s}$$ in $\C$ with $\s$ in $\Si$ a
\emph{left fraction}.  The morphisms $X\to Y$ in $\Si^{-1}\C$ are
equivalence classes $[\a,\s]$ of such left fractions, where two
diagrams $(\a_1,\s_1)$ and $(\a_2,\s_2)$ are equivalent if there
exists a commutative diagram
$$\xymatrix{&Y_1\ar[d]\\
X\ar[r]^{\a_3}\ar[ru]^{\a_1}\ar[rd]_{\a_2}&Y_3&
Y\ar[lu]_{\s_1}\ar[ld]^{\s_2}\ar[l]_{\s_3}\\ &Y_2\ar[u]}$$ with $\s_3$
in $\Si$. The composition of two equivalence classes $[\a,\s]$ and
$[\b,\t]$ is by definition the equivalene class
$[\b'\comp\a,\s'\comp\t]$ where $\s'$ and $\b'$ are obtained from
condition (LF2) as in the following commutative diagram.
$$\xymatrix{&&Z''\\ &Y'\ar[ru]^{\b'}&& Z'\ar[lu]_{\s'}\\
X\ar[ru]^\a&&Y\ar[lu]_{\s}\ar[ru]^\b&&Z\ar[lu]_{\t}}$$ We obtain a
canonical functor $$P_\Si\colon\C\lto\Si^{-1}\C$$ by taking the
identity map on objects and by sending a morphism $\a\colon X\to Y$ to the
equivalence class $[\a,\id Y]$. Let us compare $P_\Si$ with the
quotient functor $Q_\Si\colon\C\to\C[\Si^{-1}]$.

\begin{prop}
The functor $F\colon\Si^{-1}\C\to\C[\Si^{-1}]$ which is the identity
map on objects and which takes a morphism $[\a,\s]$ to $(Q_\Si\s)^{-1}\comp
Q_\Si\a$ is an isomorphism.
\end{prop}
\begin{proof}
The functor $P_\Si$ inverts all morphisms in $\Si$ and factors therefore
through $Q_\Si$ via a functor $G\colon \C[\Si^{-1}]\to\Si^{-1}\C$.
It is straightforward to check that $F\comp G=\Id$ and $G\comp F=\Id$.
\end{proof}
From now on, we will identify $\Si^{-1}\C$ with $\C[\Si^{-1}]$
whenever $\Si$ admits a calculus of left fractions. A set of morphisms
$\Si$ in $\C$ {\em admits a calculus of right fractions} if the dual
conditions of (LF1) -- (LF3) are satisfied. Moreover, $\Si$ is called
a {\em multiplicative system} if it admits both, a calculus of left
fractions and a calculus of right fractions. Note that all results
about sets of morphisms admitting a calculus of left fractions have a
dual version for sets of morphisms admitting a calculus of right
fractions.

\subsection{Calculus of fractions and adjoint functors}
Given a category $\C$ and a set of morphisms $\Si$, it is an
interesting question to ask when the quotient functor
$\C\to\C[\Si^{-1}]$ admits a right adjoint. It turns out that this
problem is closely related to the property of $\Si$ to admit a
calculus of left fractions.

\begin{lem}
Let $F\colon\C\to\D$ and $G\colon\D\to\C$ be a pair of adjoint
functors. Assume that the right adjoint $G$ is fully faithful and let
$\Si$ be the set of morphisms $\s$ in $\C$ such that $F\s$ is invertible.
Then $\Si$ admits a calculus of left fractions.
\end{lem}
\begin{proof} 
We need to check the conditions (LF1) -- (LF3). Observe first that
$L=G\comp F$ is a localization functor so that we can apply
Proposition~\ref{pr:local-obj}.

(LF1): This condition is clear because $F$ is a functor.

(LF2): Let $X'\xleftarrow{\s} X\xto{\a}Y$ be a pair of morphisms with
 $\s$ in $\Si$. This can be completed to a commutative square
$$\xymatrix{X\ar[r]^\a\ar[d]^{\s}&Y\ar[d]^{\s'}\\ X'\ar[r]^{\a'}&Y'}$$
if we take for $\s'$ the morphism $\eta Y\colon Y\to LY$ in $\Si$,
because the map $\C(\s,LY)$ is surjective by Proposition~\ref{pr:local-obj}.

(LF3): Let $\a,\b\colon X\to Y$ be morphisms in $\C$ and suppose there
is a morphism $\s\colon X'\to X$ in $\Si$ with
$\a\comp\s=\b\comp\s$. Then we take $\t=\eta Y$ in $\Si$ and have
$\t\comp\a=\t\comp\b$, because the map $\C(\s,LY)$ is
injective  by Proposition~\ref{pr:local-obj}.
\end{proof}

\begin{lem}\label{le:calc-adj}
Let $\C$ be a category and $\Si$ a set of morphisms admitting a
calculus of left fractions. Then the following are equivalent for an
object $X$ in $\C$.
\begin{enumerate}
\item  $X$ is local with respect to $\Si$.
\item The quotient functor induces a bijection
$\C(W,X)\to\C[\Si^{-1}](W,X)$ for all $W$.
\end{enumerate}
\end{lem}
\begin{proof}
(1) $\Rightarrow$ (2): To show that
$f_W\colon\C(W,X)\to\C[\Si^{-1}](W,X)$ is surjective, choose a left
fraction $W\xto{\a}X'\xleftarrow{\s}X$ with $\s$ in $\Si$. Then there
exists $\t\colon X'\to X$ with $\t\comp\s=\id X$ since $X$ is
local. Thus $f_W(\t\comp\a)=[\a,\s]$. To show that $f_W$ is injective,
suppose that $f_W(\a)=f_W(\b)$. Then we have $\s\comp\a=\s\comp\b$ for
some $\s\colon X\to X'$ in $\Si$. The morphism $\s$ is a section
because $X$ is local, and therefore $\a=\b$.

(2) $\Rightarrow$ (1): Let $\s\colon W\to W'$ be a morphism in $\Si$.
Then we have $\C(\s,X)\cong\C[\Si^{-1}]([\s,\id W'],X)$. Thus
$\C(\s,X)$ is bijective since $[\s,\id W']$ is invertible.
\end{proof}

\begin{prop}\label{pr:calc-adj}
Let $\C$ be a category, $\Si$ a set of morphisms admitting a calculus
of left fractions, and $Q\colon\C\to\C[\Si^{-1}]$ the quotient
functor. Then the following are equivalent.
\begin{enumerate}
\item The functor $Q$ has a right adjoint
(which is then fully faithful).
\item For each object $X$ in $\C$, there exist a morphism $\eta X\colon X\to
X'$ such that $X'$ is local with respect to $\Si$ and $Q(\eta X)$ is
invertible.
\end{enumerate}
\end{prop}
\begin{proof} 
(1) $\Rightarrow$ (2): Denote by $Q_\r$ the right adjoint of $Q$ and
    by $\eta\colon \Id\C\to Q_\r Q$ the adjunction morphism. We take
    for each object $X$ in $\C$ the morphism $\eta X\colon X\to Q_\r Q
    X$. Note that $Q_\r Q X$ is local by
    Proposition~\ref{pr:local-obj}.

(2) $\Rightarrow$ (1): We fix objects $X$ and $Y$. Then we have two
natural bijections
$$\C[\Si^{-1}](X,Y)\xto{\sim}\C[\Si^{-1}](X,Y')\xleftarrow{\sim}\C(X,Y').$$
The first is induced by $\eta Y\colon Y\to Y'$ and is bijective since
$Q(\eta Y)$ is invertible.  The second map is bijective by
Lemma~\ref{le:calc-adj}, since $Y'$ is local with respect to $\Si$.
Thus we obtain a right adjoint for $Q$ by sending each object $Y$
of $\C[\Si^{-1}]$ to $Y'$.
\end{proof}

\subsection{A criterion for the fractions to form a small set}

Let $\C$ be a category and $\Si$ a set of morphisms in $\C$. Suppose
that $\Si$ admits a calculus of left fractions. From the construction
of $\C[\Si^{-1}]$ we cannot expect that for any given pair of objects
$X$ and $Y$ the equivalence classes of fractions in
$\C[\Si^{-1}](X,Y)$ form a small set. The situation is different if
the category $\C$ is  small.  Then it is clear that
$\C[\Si^{-1}](X,Y)$ is a small set for all objects $X,Y$. The following
criterion generalizes this simple observation.

\begin{lem}\label{le:frac-set}
Let $\C$ be a category and $\Si$ a set of morphisms in $\C$ which
admits a calculus of left fractions. Let $Y$ be an object in $\C$ and
suppose that there exists a small set $S=S(Y,\Si)$ of objects in $\C$
such that for every morphism $\s\colon Y\to Y'$ in $\Si$ there is a
morphism $\t\colon Y'\to Y''$ with $\t\comp\s$ in $\Si$ and $Y''$ in
$S$. Then $\C[\Si^{-1}](X,Y)$ is a small set for every object $X$ in
$\C$.
\end{lem}
\begin{proof}
The condition on $Y$ implies that every fraction
$X\stackrel{\a}\rightarrow Y'\stackrel{\s}\leftarrow Y$ is equivalent
to one of the form $X\stackrel{\a'}\rightarrow
Y''\stackrel{\s'}\leftarrow Y$ with $Y''$ in $S$.  Clearly, the
fractions of the form $(\a',\s')$ with $\s'\in\C(Y,Y'')$ and $Y''\in
S$ form a small set.
\end{proof}

\subsection{Calculus of fractions for subcategories}
We provide a criterion such that the calculus of fractions for a set
of morphisms in a category $\C$ is compatible with the passage to a
subcategory of $\C$.

\begin{lem}\label{le:sub}
Let $\C$ be a category and $\Si$ a set of morphisms admitting a
calculus of left fractions. Suppose $\D$ is a full subcategory of $\C$
such that for every morphism $\s\colon Y\to Y'$ in $\Si$ with $Y$ in
$\D$ there is a morphism $\t\colon Y'\to Y''$ with $\t\comp\s$ in
$\Si\cap\D$.  Then $\Si\cap \D$ admits a calculus of left fractions
and the induced functor $\D[(\Si\cap\D)^{-1}]\to\C[\Si^{-1}]$ is fully
faithful.
\end{lem}
\begin{proof}
It is straightforward to check (LF1) -- (LF3) for $\Si\cap\D$.  Now let
$X,Y$ be objects in $\D$. Then we need to show that the induced map
$$f\colon \D[(\Si\cap\D)^{-1}](X,Y)\lto\C[\Si^{-1}](X,Y)$$ is
bijective.  The map sends the equivalence class of a fraction to the
equivalence class of the same fraction.  If $[\a,\s]$ belongs to
$\C[\Si^{-1}](X,Y)$ and $\t$ is a morphism with $\t\comp\s$ in
$\Si\cap\D$, then $[\t\comp\a,\t\comp\s]$ belongs to
$\D[(\Si\cap\D)^{-1}](X,Y)$ and $f$ sends it to $[\a,\s]$.  Thus $f$
is surjective. A similar argument shows that $f$ is injective.
\end{proof}

\begin{exm}
Let $A$ be a commutative noetherian ring and $S\subseteq A$ a
multiplicatively closed subset. Denote by $\Si$ the set of morphisms
$\s$ in $\Mod A$ such that $S^{-1}\s$ is invertible. Then $\Si$ is a
multiplicative system and one can show directly that for the
subcategory $\mod A$ of finitely generated $A$-modules and
$T=\Si\cap\mod A$ the dual of the condition in Lemma~\ref{le:sub}
holds. Thus the induced functor
$$(\mod A)[T^{-1}]\lto(\Mod A)[\Si^{-1}]$$
is fully faithful.
\end{exm}

\subsection{Calculus of fractions and coproducts}

We provide a criterion for the quotient functor $\C\to\C[\Si^{-1}]$
to preserve small coproducts.

\begin{prop}\label{pr:coprod}
Let $\C$ be a category which admits small coproducts. Suppose that
$\Si$ is a set of morphisms in $\C$ which admits a calculus of left
fractions. If $\coprod_i\s_i$ belongs to $\Si$ for every family
$(\s_i)_{i\in I}$ in $\Si$, then the category $\C[\Si^{-1}]$ admits
small coproducts and the quotient functor $\C\to\C[\Si^{-1}]$
preserves small coproducts.
\end{prop}
\begin{proof} 
Let $(X_i)_{i\in I}$ be a family of objects in $\C[\Si^{-1}]$ which is
indexed by a small set $I$.  We claim that the coproduct
$\coprod_iX_i$ in $\C$ is also a coproduct in $\C[\Si^{-1}]$.  Thus we
need to show that for every object $Y$, the canonical map
\begin{equation}\label{eq:coprod}
\C[\Si^{-1}](\coprod_i X_i,Y)\lto\prod_i\C[\Si^{-1}](X_i,Y)
\end{equation}
is bijective.

To check surjectivity of \eqref{eq:coprod}, let
$(X_i\stackrel{\a_i}\rightarrow Z_i\stackrel{\s_i}\leftarrow Y)_{i\in
I}$ be a family of left fractions.  Using (LF2), we obtain a
commutative diagram
$$\xymatrix{
\coprod_iX_i\ar[r]^{\coprod_i\a_i}&\coprod_iZ_i\ar[d]
&\coprod_iY\ar[d]^{\pi_Y}\ar[l]_{\coprod_i\s_i}\\
&Z&Y\ar[l]_\s
}$$
where $\pi_Y\colon\coprod_iY\to Y$ is the summation morphism and
$\s\in\Si$. It is easily checked that
$$(X_i\rightarrow Z\stackrel{\s}\leftarrow
Y)\,\sim\,(X_i\stackrel{\a_i}\rightarrow Z_i\stackrel{\s_i}\leftarrow
Y)$$ for all $i\in I$, and therefore \eqref{eq:coprod} sends
$\coprod_iX_i\rightarrow Z\stackrel{\s}\leftarrow Y$ to the family
$(X_i\stackrel{\a_i}\rightarrow Z_i\stackrel{\s_i}\leftarrow
Y_i)_{i\in I}$.

To check injectivity  of \eqref{eq:coprod}, let $\coprod_iX_i\stackrel{\a'}\rightarrow
Z'\stackrel{\s'}\leftarrow Y$ and
$\coprod_iX_i\stackrel{\a''}\rightarrow Z''\stackrel{\s''}\leftarrow
Y$ be left fraction such that
$$(X_i\stackrel{\a'_i}\rightarrow Z'\stackrel{\s'}\leftarrow
Y)\,\sim\, (X_i\stackrel{\a''_i}\rightarrow
Z''\stackrel{\s''}\leftarrow Y)$$ for all $i$. We may assume that
$Z'=Z=Z''$ and $\s'=\s=\s''$ since we can choose morphisms $\t'\colon
Z'\to Z$ and $\t''\colon Z''\to Z$ with
$\t'\comp\s'=\t''\comp\s''\in\Si$.  Thus there are morphisms
$\b_i\colon Z\to Z_i$ with $\b_i\comp\a'_i=\b_i\comp\a''_i$ and
$\b_i\comp\s\in\Si$ for all $i$.  Each $\b_i$ belongs to the {\em
saturation} ${\bar\Si}$ of $\Si$ which is the set of all
morphisms in $\C$ which become invertible in $\C[\Si^{-1}]$. Note that
a morphism $\p$ in $\C$ belongs to ${\bar\Si}$ if and only if
there are morphisms $\p'$ and $\p''$ such that $\p\comp\p'$ and
$\p''\comp\p$ belong to $\Si$. Therefore ${\bar\Si}$ is also
closed under taking coproducts. Moreover, ${\bar\Si}$ admits a
calculus of left fractions, and we obtain therefore a commutative
diagram
$$\xymatrix{
\coprod_iX_i\ar[r]&\coprod_iZ\ar[d]^{\coprod_i\b_i}\ar[r]^{\pi_Z}&Z\ar[d]^\t\\
&\coprod_iZ_i\ar[r]&Z^*
}$$
with $\tau\in {\bar \Si}$. Thus $\tau\comp\s\in {\bar \Si}$, and we have
$$(\coprod_iX_i\stackrel{\a'}\rightarrow Z\stackrel{\s}\leftarrow
Y)\,\sim\, (\coprod_iX_i\stackrel{\a''}\rightarrow
Z\stackrel{\s}\leftarrow Y)$$ since $\pi_Z\comp\coprod_i\a'_i=\a'$ and
$\pi_Z\comp\coprod_i\a''_i=\a'$. Therefore the map \eqref{eq:coprod}
is also injective, and this completes the proof.
\end{proof}

\begin{exm}
Let $\C$ be a category which admits small coproducts and $L\colon
\C\to\C$ be a localization functor. Then a morphism $\s$ in $\C$
belongs to $\Si=\Si(L)$ if and only if the induced map $\C(\s,LX)$ is
invertible for every object $X$ in $\C$. Thus $\Si$ is closed under
taking small coproducts and therefore the quotient functor
$\C\to\C[\Si^{-1}]$ preserves small coproducts.
\end{exm}

\subsection{Notes}
The calculus of fractions for categories has been developed by Gabriel
and Zisman in \cite{GZ} as a tool for homotopy theory.

\section{Localization for triangulated categories}

\subsection{Triangulated categories}

Let $\T$ be an additive category with an equivalence
$S\colon\T\to\T$. A {\em triangle} in $\T$ is a sequence
$(\a,\b,\g)$ of morphisms
$$X\lto[\a] Y\lto[\b] Z\lto[\g]S X,$$ and a morphism between two
triangles $(\a,\b,\g)$ and $(\a',\b',\g')$ is a triple $(\p_1,\p_2,\p_3)$
of morphisms in $\T$ making the following diagram commutative.
$$\xymatrix{
X\ar[r]^\a\ar[d]^{\p_1}&Y\ar[r]^\b\ar[d]^{\p_2}&Z\ar[r]^\g\ar[d]^{\p_3}&S
X\ar[d]^{S\p_1}\\ X'\ar[r]^{\a'}&Y'\ar[r]^{\b'}&Z'\ar[r]^{\g'}&S X'
}$$ The category $\T$ is called {\em triangulated} if it is equipped
with a set of distinguished triangles (called {\em exact triangles})
satisfying the following conditions.
\begin{enumerate}
\item[(TR1)] A triangle isomorphic to an exact triangle is exact. For
each object $X$, the triangle $0\to X\xto{\id} X\to 0$ is exact.
Each morphism $\a$ fits into an exact triangle $(\a,\b,\g)$.
\item[(TR2)] A triangle $(\a,\b,\g)$ is exact if and only if
$(\b,\g,-S\a)$ is exact.
\item[(TR3)] Given two exact triangles $(\a,\b,\g)$ and
$(\a',\b',\g')$, each pair of morphisms $\p_1$ and $\p_2$ satisfying
$\p_2\comp\a=\a'\comp\p_1$ can be completed to a morphism
$$\xymatrix{
X\ar[r]^\a\ar[d]^{\p_1}&Y\ar[r]^\b\ar[d]^{\p_2}&Z\ar[r]^\g\ar[d]^{\p_3}&S
X\ar[d]^{S\p_1}\\ X'\ar[r]^{\a'}&Y'\ar[r]^{\b'}&Z'\ar[r]^{\g'}&S X'
}$$ of triangles.
\item[(TR4)] Given exact triangles $(\a_1,\a_2,\a_3)$,
$(\b_1,\b_2,\b_3)$, and $(\g_1,\g_2,\g_3)$ with $\g_1=\b_1\comp\a_1$,
there exists an exact triangle $(\d_1,\d_2,\d_3)$ making
the following diagram commutative.
$$\xymatrix{X\ar[r]^{\a_1}\ar@{=}[d]&Y\ar[r]^{\a_2}\ar[d]^{\b_1}&
U\ar[r]^{\a_3}\ar[d]^{\d_1}& S X\ar@{=}[d]\\
X\ar[r]^{\g_1}&Z\ar[r]^{\g_2}\ar[d]^{\b_2}&
V\ar[r]^{\g_3}\ar[d]^{\d_2}&S X\ar[d]^{S\a_1}\\
&W\ar@{=}[r]\ar[d]^{\b_3}& W\ar[d]^{\d_3}\ar[r]^{\b_3}&S Y\\ &S
Y\ar[r]^{S\a_2}&S U }$$
\end{enumerate}

Recall that an idempotent endomorphism $\p=\p^2$ of an object $X$ in
an additive category \emph{splits} if there exists a factorization
$X\xto{\pi}Y\xto{\iota}X$ of $\p$ with $\pi\comp\iota=\id Y$.

\begin{rem}
Suppose a triangulated category $\T$ admits countable coproducts. Then
every idempotent endomorphism splits.  More precisely, let $\p\colon
X\to X$ be an idempotent morphism in $\T$, and denote by $Y$ a homotopy
colimit of the sequence
$$X\lto[\p]X\lto[\p]X\lto[\p]\cdots.$$ The morphism $\p$ factors
through the canonical morphism $\pi\colon X\to Y$ via a morphism
$\iota\colon Y\to X$, and we have $\pi\comp\iota=\id Y$. Thus $\p$
splits; see \cite[Proposition~1.6.8]{Nee2001} for details.
\end{rem}

\subsection{Exact functors}

An {\em exact} functor $\T\to\U$ between triangulated categories is a
pair $(F,\mu)$ consisting of a functor $F\colon\T\to\U$ and
an isomorphism $\mu\colon F\comp S_\T\to S_\U\comp F$ such that for
every exact triangle $X\to[\a]Y\to[\b]Z\to[\g]S_\T X$ in $\T$ the
triangle
$$FX\lto[F\a]FY\lto[F\b]FZ\xto{\mu X\comp F\g}S_\U (FX)$$ is exact in $\U$.

We have the following useful lemma.
\begin{lem}\label{le:exact-adj}
Let $F\colon\T\to\U$ and $G\colon\U\to\T$ be an adjoint pair of functors between triangulated
categories. If one of both functors is exact, then also the other is exact.
\end{lem}
\begin{proof}
See \cite[Lemma~5.3.6]{Nee2001}.
\end{proof}
\subsection{Multiplicative systems}

Let $\T$ be a triangulated category and $\Si$ a set of morphisms
which is a multiplicative system. Recall this means that $\Si$ admits
a calculus of left and right fractions. Then we say that $\Si$ is {\em
compatible with the triangulation} if
\begin{enumerate}
\item given $\s$ in $\Si$, the morphism $S^n\s$ belongs to $\Si$ for
all $n\in\bbZ$, and
\item given a morphism $(\p_1,\p_2,\p_3)$
between exact triangles with $\p_1$ and $\p_2$ in $\Si$, there is also a
morphism $(\p_1,\p_2,\p'_3)$ with $\p'_3$ in $\Si$.
\end{enumerate}

\begin{lem}\label{le:qis}
Let $\T$ be a triangulated category and $\Si$ a multiplicative system
of morphisms which is compatible with the triangulation. Then the
quotient category $\T[\Si^{-1}]$ carries a unique triangulated structure such
that the quotient functor $\T\to\T[\Si^{-1}]$  is exact.
\end{lem}
\begin{proof}
The equivalence $S\colon\T\to\T$ induces a unique equivalence
$\T[\Si^{-1}]\to\T[\Si^{-1}]$ which commutes with the quotient functor
$Q\colon\T\to\T[\Si^{-1}]$. This follows from the fact that
$S\Si=\Si$. Now take as exact triangles in $\T[\Si^{-1}]$ all those
isomorphic to images of exact triangles in $\T$. It is straightforward
to verify the axioms (TR1) -- (TR4); see \cite[II.2.2.6]{V}. The
functor $Q$ is exact by construction. In particular, we have $Q\comp
S_\T=S_{\T[\Si^{-1}]}\comp Q$.
\end{proof}

\subsection{Cohomological functors}
A functor $H\colon\T\to\A$ from a triangulated category $\T$ to an
abelian category $\A$ is \emph{cohomological} if $H$ sends every exact
triangle in $\T$ to an exact sequence in $\A$.

\begin{exm}
For each object $X$ in $\T$, the representable functors
$\T(X,-)\colon\T\to\Ab$ and $\T(-,X)\colon\T^\op\to\Ab$ into the
category $\Ab$ of abelian groups are cohomological functors.
\end{exm}

\begin{lem}\label{le:coh}
Let $H\colon\T\to\A$ be a cohomological functor. Then the set $\Si$
of morphisms $\s$ in $\T$ such that $H(S^n\s)$ is invertible for all
$n\in\bbZ$ forms a multiplicative system which is compatible with the
triangulation of $\T$.
\end{lem}
\begin{proof}
We need to verify that $\Si$ admits a calculus of left and right
fractions.  In fact, it is sufficient to check conditions (LF1) --
(LF3), because then the dual conditions are satisfied as well since
the definition of $\Si$ is self-dual.

(LF1): This condition is clear because $H$ is a functor.

(LF2): Let $\a\colon X\to Y$ and $\s\colon X\to X'$ be morphisms with
$\s$ in $\Si$. We complete $\a$ to an exact triangle and apply (TR3)
to obtain the following morphism between exact triangles.
$$\xymatrix{
W\ar[r]\ar@{=}[d]&X\ar[r]^\a\ar[d]^{\s}&Y\ar[r]\ar[d]^{\s'}&S
W\ar@{=}[d]\\ W\ar[r]&X'\ar[r]^{\a'}&Y'\ar[r]&S W }$$ Then the
5-lemma shows that $\s'$ belongs to $\Si$.

(LF3): Let $\a,\b\colon X\to Y$ be morphisms in $\T$ and $\s\colon X'\to X$
in $\Si$ such that $\a\comp\s=\b\comp \s$. Complete $\s$ to an exact
triangle $X'\to[\s] X\to[\p] X''\to S X'$. Then $\a-\b$ factors through
$\p$ via some morphism $\psi\colon X''\to Y$. Now complete $\psi$ to an
exact triangle $X''\to[\psi] Y\to[\t] Y'\to S X''$. Then $\t$ belongs
to $\Si$ and $\t\comp\a=\t\comp\b$.

It remains to check that $\Si$ is compatible with the triangulation.
Condition (1) is clear from the definition of $\Si$. For condition (2), observe
that given any morphism $(\p_1,\p_2,\p_3)$ between exact triangles
with $\p_1$ and $\p_2$ in $\Si$, we have that $\p_3$ belongs to
$\Si$. This is an immediate consequence of the 5-lemma.
\end{proof}

\subsection{Triangulated and thick subcategories}
Let $\T$ be a triangulated category. A non-empty full subcategory $\Sc$ is a
{\em triangulated subcategory} if the following conditions hold.
\begin{enumerate}
\item[(TS1)] $S^n X\in\Sc$ for all $X\in\Sc$ and $n\in\bbZ$.
\item[(TS2)] Let $X\to Y\to Z\to S X$ be an exact triangle in $\T$.
  If two objects from $\{X,Y,Z\}$ belong to $\Sc$, then also the third.
\end{enumerate}
A triangulated subcategory $\Sc$ is {\em thick} if in addition the following
condition holds.
\begin{enumerate}
\item[(TS3)] Let $X\xto{\pi} Y\xto{\iota} X$ be morphisms in $\T$ such
that $\id Y=\pi\comp\iota$. If $X$ belongs to $\Sc$, then also $Y$.
\end{enumerate}
Note that a triangulated subcategory $\Sc$ of $\T$ inherits a canonical
triangulated structure from $\T$. 

Next observe that a triangulated subcategory $\Sc$ of $\T$ is
thick provided that $\Sc$ admits countable coproducts. This follows
from the fact that in a triangulated category with countable
coproducts all idempotent endomorphisms split.

Let $\T$ be a triangulated category and let $F\colon\T\to\U$ be an
additive functor. The {\em kernel} $\Ker F$ of $F$ is by definition
the full subcategory of $\T$ which is formed by all objects $X$ such
that $FX=0$. If $F$ is an exact functor into a triangulated category,
then $\Ker F$ is a thick subcategory of $\T$. Also, if $F$ is a
cohomological functor into an abelian category, then
$\bigcap_{n\in\bbZ}S^n(\Ker F)$ is a thick subcategory of $\T$.

\subsection{Verdier localization}
Let $\T$ be a triangulated category. Given a triangulated subcategory
$\Sc$, we denote by $\Si(\Sc)$ the set of morphisms $X\to Y$ in $\T$
which fit into an exact triangle $X\to Y\to Z\to SX$ with $Z$ in $\Sc$.
 
\begin{lem}
Let $\T$ be a triangulated category and $\Sc$  a triangulated
subcategory. Then $\Si(\Sc)$ is a multiplicative system
which is compatible with the triangulation of $\T$.
\end{lem}
\begin{proof} 
The proof is similar to that of Lemma~\ref{le:coh}; see
\cite[II.2.1.8]{V} for details.
\end{proof}

The {\em localization} of $\T$ with respect to a triangulated
subcategory $\Sc$ is by definition the quotient category
$$\T/\Sc:=\T[\Si(\Sc)^{-1}]$$ together with the quotient functor
$\T\to\T/\Sc$.

\begin{prop}
Let $\T$ be a triangulated category and $\Sc$ a full triangulated
subcategory. Then the category $\T/\Sc$ and the quotient functor
$Q\colon\T\to\T/\Sc$ have the following properties.
\begin{enumerate}
\item The category $\T/\Sc$ carries a unique triangulated structure
such that $Q$ is exact.
\item A morphism in $\T$ is annihilated by $Q$ if and only if it
factors through an object in $\Sc$.
\item The kernel $\Ker Q$ is the smallest thick subcategory containing
$\Sc$.
\item Every exact functor $\T\to\U$ annihilating $\Sc$ factors
uniquely through $Q$ via an exact functor $\T/\Sc\to\U$.
\item Every cohomological functor $\T\to\A$ annihilating $\Sc$ factors
uniquely through $Q$ via a cohomological functor $\T/\Sc\to\A$.
\end{enumerate}
\end{prop}
\begin{proof}
(1) follows from Lemma~\ref{le:qis}.

(2) Let $\p$ be a morphism in $\T$. We have $Q\p=0$ iff $\s\comp\p=0$
for some $\s\in\Si(\Sc)$ iff $\p$ factors through some object in $\Sc$.

(3) Let $X$ be an object in $\T$. Then $QX=0$ if and only if $Q(\id
    X)=0$. Thus part (2) implies that the kernel of $Q$ conists of all
    direct factors of objects in $\Sc$.

(4) An exact functor $F\colon\T\to\U$ annihilating $\Sc$ inverts every
    morphism in $\Si(\Sc)$. Thus there exists a unique functor $\bar
    F\colon\T/\Sc\to\U$ such that $F=\bar F\comp Q$. The functor $\bar
    F$ is exact because an exact triangle $\Delta$ in $\T/\Sc$ is up to
    isomorphism of the form $Q\Delta'$ for some exact triangle $\Delta'$
    in $\T$. Thus $\bar F\Delta\cong F\Delta'$ is exact.

(5) Analogous to (4).
\end{proof}

\subsection{Localization of subcategories}

Let $\T$ be a triangulated category with two full triangulated
subcategories $\T'$ and $\Sc$.  Then we put $\Sc'=\Sc\cap\T'$ and have
$\Si_{\T'}(\Sc')=\Si_\T(\Sc)\cap\T'$. Thus we can form the following
commutative diagram of exact functors
$$\xymatrix{\Sc'\ar[d]^-\inc\ar[rr]^-\inc&&\T'\ar[d]^-\inc\ar[rr]^-\can&&
\T'/\Sc'\ar[d]^-J\\
\Sc\ar[rr]^-\inc&&\T\ar[rr]^-\can&&\T/\Sc}$$
and ask when the functor $J$ is fully faithful. We have the following criterion.

\begin{lem}\label{le:tria-sub}
Let $\T$, $\T'$, $\Sc$, $\Sc'$ be as above. Suppose that either
\begin{enumerate}
\item every morphism from an object in $\Sc$ to an object in $\T'$
factors through some object in $\Sc'$, or
\item every morphism from an object in $\T'$ to an object in $\Sc$
factors through some object in $\Sc'$.
\end{enumerate}  
Then the induced functor $J\colon\T'/\Sc'\to\T/\Sc$
is fully faithful.
\end{lem}
\begin{proof}
Suppose that condition (1) holds. We apply the criterion from
Lemma~\ref{le:sub}. Thus we take a morphism $\s\colon Y\to Y'$ from
$\Si(\Sc)$ with $Y$ in $\T'$ and need to find $\t\colon Y'\to Y''$
such that $\t\comp\s$ belongs to $\Si(\Sc)\cap\T'$. To this end
complete $\s$ to an exact triangle $X\xto{\p} Y\xto{\s} Y'\to
SX$. Then $X$ belongs to $\Sc$ and by our assumption we have a
factorization $X\xto{\p'} Z\xto{\p''} Y$ of $\p$ with $Z$ in
$\Sc'$. Complete $\p''$ to an exact triangle $Z\xto{\p''} Y\xto{\psi}
Y''\to SZ$. Then (TR3) yields a morphism $\t\colon Y'\to Y''$
satisfying $\psi=\t\comp\s$.  In particular, $\t\comp\s$ lies in
$\Si(\Sc)\cap\T'$ since $Z$ belongs to $\Sc'$.  The proof using
condition (2) is dual.
\end{proof}

\subsection{Orthogonal subcategories}

Let $\T$ be a triangulated category and $\Sc$ a triangulated
subcategory. Then we define two full subcategories
\begin{align*}
\Sc^\perp&=\{Y\in\T\mid\T(X,Y)=0\text{ for all }X\in\Sc\}\\
^\perp\Sc&=\{X\in\T\mid\T(X,Y)=0\text{ for all }Y\in\Sc\}
\end{align*}
and call them \emph{orthogonal subcategories} with respect to $\Sc$.
Note that $\Sc^\perp$ and $^\perp\Sc$ are thick subcategories of $\T$.

\begin{lem} 
\label{le:local-acyclic} 
Let $\T$ be a triangulated category and $\Sc$ a triangulated
subcategory. Then the following are equivalent for an object $Y$ in $\T$.
\begin{enumerate}
\item $Y$ belongs to $\Sc^\perp$.
\item $Y$ is $\Si(\Sc)$-local, that is, $\T(\s,Y)$ is bijective for all $\s$ in $\Si(\Sc)$.
\item The quotient functor induces a bijection
$\T(X,Y)\to\T/\Sc(X,Y)$ for all $X$ in $\T$.
\end{enumerate}
\end{lem}
\begin{proof}
(1) $\Rightarrow$ (2): Suppose $\T(X,Y)=0$ for all $X$ in $\Sc$. Then
every $\s$ in $\Si(\Sc)$ induces a bijection $\T(\s,Y)$ because
$\T(-,Y)$ is cohomological. Thus $Y$ is $\Si(\Sc)$-local.

(2) $\Rightarrow$ (1): Suppose that $Y$ is $\Si(\Sc)$-local. If $X$
    belongs to $\Sc$, then the morphism $\s\colon X\to 0$ belongs to
    $\Si(\Sc)$ and induces therefore a bijection $\C(\s,Y)$. Thus $Y$
    belongs to $\Sc^\perp$.

(2) $\Leftrightarrow$ (3): Apply Lemma~\ref{le:calc-adj}.
\end{proof}

\subsection{Bousfield localization}

Let $\T$ be a triangulated category. We wish to study exact
localization functors $L\colon\T\to\T$.  To be more precise, we assume
that $L$ is an exact functor and that $L$ is a localization functor in
the sense that there exists a morphism $\eta\colon\Id\C\to L$ with
$L\eta\colon L\to L^2$ being invertible and $L\eta=\eta L$. Note that
there is an isomorphism $\mu\colon L\comp S\xto{\sim} S\comp L$ since
$L$ is exact, and there exists a unique choice such that $\mu
X\comp\eta SX=S\eta X$ for all $X$ in $\T$. This follows from
Lemma~\ref{le:local}.

We observe that the kernel of an exact localization functor is a thick
subcategory of $\T$. The following fundamental result characterizes
the thick subcategories of $\T$ which are of this form.

\begin{prop}\label{pr:tria-loc}
Let $\T$ be a triangulated category and $\Sc$ a thick subcategory. Then
the following are equivalent.
\begin{enumerate}
\item There exists an exact localization functor $L\colon\T\to\T$ with $\Ker L=\Sc$.
\item The inclusion functor $\Sc\to\T$ admits a right adjoint.
\item For each $X$ in $\T$ there exists an exact triangle $X'\to X\to
X''\to SX''$ with $X'$ in $\Sc$ and $X''$ in $\Sc^\perp$.
\item The quotient functor $\T\to\T/\Sc$ admits a right adjoint.
\item The composite $\Sc^\perp\xto{\inc}\T\xto{\can}\T/\Sc$ is an equivalence.
\item The inclusion functor $\Sc^\perp\to\T$ admits a left adjoint and  $^\perp(\Sc^\perp)=\Sc$.
\end{enumerate}
\end{prop}
\begin{proof}
Let  $I\colon\Sc\to\T$ and $J\colon\Sc^\perp\to\T$ denote the inclusions and
$Q\colon \T\to\T/\Sc$ the quotient functor.

(1) $\Rightarrow$ (2): Suppose that $L\colon\T\to\T$ is an exact
localization functor with $\Ker L=\Sc$ and let $\eta\colon \Id\T\to L$ be a
morphism such that $L\eta$ is invertible.  We obtain a right adjoint
$I_\r\colon\T\to\Sc$ for the inclusion $I$ by completing for each $X$
in $\T$ the morphism $\eta X$ to an exact triangle $I_\r X\xto{\theta
  X} X\xto{\eta X} LX\to S(I_\r X)$. Note that $I_\r X$ belongs to
$\Sc$ since $L\eta X$ is invertible. Moreover, $\T(W,\theta X)$ is
bijective for all $W$ in $\Sc$ since $\T(W,LX)=0$ by
Lemma~\ref{le:local-acyclic}.  Here we use that $LX$ is $\Si(L)$-local
by Proposition~\ref{pr:local-obj} and that $\Si(L)=\Si(\Sc)$.  Thus
$I_\r$ provides a right adjoint for $I$ since $\T(W,I_\r
X)\cong\T(IW,X)$ for all $W$ in $\Sc$ and $X$ in $\T$.  In particular,
we see that the exact triangle defining $I_\r X$ is, up to a unique
isomorphism, uniquely determined by $X$.  Therefore $I_\r$ is well
defined.

(2) $\Rightarrow$ (3): Suppose that $I_\r\colon\T\to\Sc$ is a right
adjoint of the inclusion $I$.  We fix an object $X$ in $\T$ and
complete the adjunction morphism $\theta X\colon I_\r X\to X$ to an
exact triangle $I_\r X\xto{\theta X} X\to X''\to S(I_\r X)$. Clearly,
$I_\r X$ belongs to $\Sc$. We have $\T(W,X'')=0$ for all $W$ in $\Sc$
since $\T(W,\theta X)$ is bijective. Thus $X''$ belongs to $\Sc^\perp$.

(3) $\Rightarrow$ (4): We apply Proposition~\ref{pr:calc-adj} to
obtain a right adjoint for the quotient functor $Q$.  To this end fix
an object $X$ in $\T$ and an exact triangle $X'\to X\xto{\eta} X''\to
SX''$ with $X'$ in $\Sc$ and $X''$ in $\Sc^\perp$. The morphism $\eta$
belongs to $\Si(\Sc)$ by definition, and the object $X''$ is
$\Si(\Sc)$-local by Lemma~\ref{le:local-acyclic}. Now it follows from
Proposition~\ref{pr:calc-adj} that $Q$ admits a
right adjoint.

(4) $\Rightarrow$ (1): Let $Q_\r\colon\T/\Sc\to\T$ denote a right
adjoint of $Q$. This functor is fully faithful by
Proposition~\ref{pr:quot} and exact by Lemma~\ref{le:exact-adj}.  Thus
$L=Q_\r\comp Q$ is an exact functor with $\Ker L=\Ker Q=\Sc$.  Moreover, $L$
is a localization functor by Corollary~\ref{co:loc}.

(4) $\Rightarrow$ (5): Let $Q_\r\colon\T/\Sc\to\T$ denote a right
adjoint of $Q$. The composite
$Q\comp J\colon\Sc^\perp\to\T/\Sc$ is fully faithful by
Lemma~\ref{le:local-acyclic}. Given an object $X$ in $\T/\Sc$, we have
$Q(Q_\r X)\cong X$ by Proposition~\ref{pr:quot}, and $Q_\r X$ belongs to
$\Sc^\perp$, since $\T(W,Q_\r X)\cong \T/\Sc(QW,X)=0$ for all $W$ in
$\Sc$. Thus $Q\comp J$ is dense and therefore an equivalence.

(5) $\Rightarrow$ (6): Suppose $Q\comp J\colon\Sc^\perp\to\T/\Sc$ is an
equivalence and let $F\colon \T/\Sc\to\Sc^\perp$ be a quasi-inverse.  We
have for all $X$ in $\T$ and $Y$ in $\Sc^\perp$
$$\T(X,JY)\xto{\sim}\T/\Sc(QX,QJY)\xto{\sim}\Sc^\perp(FQX,FQJY)\xto{\sim}\Sc^\perp(FQX,Y),$$
where the first bijection follows from Lemma~\ref{le:local-acyclic} and
the others are clear from the choice
of $F$.  Thus $F\comp Q$ is a left adjoint for the inclusion $J$. 

It remains to show that $^\perp(\Sc^\perp)=\Sc$.  The inclusion
${^\perp(\Sc^\perp)}\supseteq \Sc$ is clear.  Now let $X$ be an object
of $^\perp(\Sc^\perp)$. Then we have
$$\T/\Sc(QX,QX)\cong\Sc^\perp(FQX,FQX)\cong\T(X,J(FQX))=0.$$ Thus $QX=0$
and therefore $X$ belongs to $\Sc$.

(6) $\Rightarrow$ (3): Suppose that $J_\la\colon\T\to\Sc^\perp$ is a
    left adjoint of the inclusion $J$. 
We fix an object $X$ in $\T$ and
complete the adjunction morphism $\mu X\colon X\to J_\la X$ to an
exact triangle $X'\to X\xto{\mu X}J_\la X\to SX'$. Clearly,
$J_\la X$ belongs to $\Sc^\perp$. We have $\T(X',Y)=0$ for all $Y$ in $\Sc^\perp$
since $\T(\mu X,Y)$ is bijective. Thus $X'$ belongs to $^\perp(\Sc^\perp)=\Sc$.
\end{proof}

The following diagram displays the functors which arise from a
localization functor $L\colon\T\to\T$. We use the convention that $F_\rho$ denotes
a right adjoint of a functor $F$. 
$$\xymatrix{\Sc\,\ar@<-.7ex>[rr]_-{I=\inc}&&\,\T\,
\ar@<-.7ex>[rr]_-{Q=\can}\ar@<-.7ex>[ll]_-{I_\rho}&&
\,\T/\Sc\ar@<-.7ex>[ll]_-{Q_\rho}}\qquad (L=Q_\rho\comp Q\quad\text{and}\quad\Ga=I\comp I_\rho)$$

\subsection{Acyclic and local objects}

Let $\T$ be a triangulated category and $L\colon\T\to\T$ an exact
localization functor. An object $X$ in $\T$ is by  definition
\emph{$L$-acyclic} if $LX=0$. Recall that an object in $\T$ is
$L$-local if and only if it belongs to the essential image $\Im L$ of
$L$; see Proposition~\ref{pr:local-obj}.  The exactness of $L$ implies
that $\Sc:=\Ker L$ is a thick subcategory and that $\Si(L)=\Si(\Sc)$.
Therefore  $L$-local and $\Si(\Sc)$-local objects coincide.

The following result says that acyclic and local objects form an
orthogonal pair. 

\begin{prop}
\label{pr:local-acyclic} 
Let $L\colon \T\to\T$ be an exact localization functor. Then we have
$$\Ker L={^\perp(\Im L)}\quad\text{and}\quad(\Ker L)^\perp=\Im L.$$
More explictly, the following holds.
\begin{enumerate}
\item $X\in\T$ is $L$-acyclic if and only if $\T(X,Y)=0$
for every  $L$-local object $Y$.
\item $Y\in\T$ is $L$-local if and only if $\T(X,Y)=0$
for every $L$-acyclic object $X$. 
\end{enumerate}
\end{prop}
\begin{proof}
(1) We write $L=G\comp F$ where $F$ is a functor and $G$ a fully
faithful right adjoint; see Corollary~\ref{co:loc}. Suppose first we
have given objects $X,Y$ such that $X$ is $L$-acyclic and $Y$ is
$L$-local. Observe that $FX=0$ since $G$ is faithful. Thus
$$\T(X,Y)\cong\T(X,GFY)\cong\T(FX,FY)=0.$$

Now suppose that $X$ is an object with $\T(X,Y)=0$ for all $L$-local $Y$. Then
$$\T(FX,FX)\cong\T(X,GFX)=0$$ and therefore $FX=0$. Thus $X$ is $L$-acyclic.

(2) This is a reformulation of Lemma~\ref{le:local-acyclic}.
\end{proof}

\subsection{A functorial triangle}
Let $\T$ be a triangulated category and $L\colon\T\to\T$ an exact
localization functor. We denote by $\eta\colon \Id\T\to L$ a morphism
such that $L\eta$ is invertible. It follows from
Proposition~\ref{pr:tria-loc} and its proof that we obtain an exact
functor $\Ga\colon \T\to\T$ by completing for each $X$ in $\T$ the
morphism $\eta X$ to an exact triangle
\begin{equation}\label{eq:extria}
\Ga X\lto[\theta X] X\lto[\eta X] LX\lto S(\Ga X).
\end{equation}
The exactness of $\Ga$ follows from Lemma~\ref{le:exact-adj}. Observe
that $\Ga X$ is $L$-acyclic and that $LX$ is $L$-local.  In fact, the
exact triangle \eqref{eq:extria} is essentially determined by these
properties. This is a consequence of the following basic properties of
$L$ and $\Ga$.

\begin{prop}\label{pr:ker_im}
The functors $L,\Ga\colon \T\to\T$ have the following properties.
\begin{enumerate}
\item $L$ induces an equivalence $\T/\Ker L\xto{\sim}\Im L$.
\item $L$ induces a left adjoint for the inclusion $\Im L\to\T$.
\item $\Ga$ induces a right adjoint for the inclusion $\Ker L\to \T$.
\end{enumerate}
\end{prop}
\begin{proof}
(1) is a reformulation of Corollary~\ref{co:local-objects}, and (2)
follows from Corollary~\ref{co:loc}.  (3) is an immediate consequence
of the construction of $\Ga$ via Proposition~\ref{pr:tria-loc}.
\end{proof}

\begin{prop}
Let $L\colon\T\to\T$ be an exact localization functor and $X$ an
object in $\C$. Given any exact triangle $X'\to X\to X''\to SX'$ with
$X'$ $L$-acyclic and $X''$ $L$-local, there are unique isomorphisms
$\a$ and $\b$ making the following diagram commutative.
\begin{equation}\label{eq:loc-tria}
\xymatrix{
X'\ar[r]\ar[d]^{\a}&X\ar[r]\ar@{=}[d]&X''\ar[r]\ar[d]^{\b}&S
X'\ar[d]^{S\a}\\ \Ga X\ar[r]^{\theta X}&X\ar[r]^{\eta X}&L X\ar[r]&S(\Ga X)}
\end{equation}
\end{prop}
\begin{proof}
The morphism $\theta X$ induces a bijection $\T(X',\theta X)$ since $X'$
is acyclic.  Thus $X'\to X$ factors uniquely through $\theta X$ via a
morphism $\a\colon X'\to \Ga X$. An application of (TR3) gives a
morphism $\b\colon X''\to LX$ making the diagram \eqref{eq:loc-tria}
commutative. Now apply $L$ to this diagram. Then $L\b$ is an
isomorphism since $LX'=0=L\Ga X$, and $L\b$ is isomorphic to $\b$
since $X''$ and $LX$ are $L$-local. Thus $\b$ is an isomorphism, and therefore
$\a$ is an isomorphism.
\end{proof}

\subsection{Localization versus colocalization}
For exact functors on triangulated categories, we have the following
symmetry principle relating localization and colocalization.

\begin{prop}\label{pr:sym}
Let $\T$ be a triangulated category.
\begin{enumerate}
\item Suppose $L\colon\T\to\T$ is an exact localization functor and
$\Ga\colon\T\to\T$ the functor which is defined in terms of the exact
triangle \eqref{eq:extria}. Then $\Ga$ is an exact colocalization
functor with $\Ker \Ga=\Im L$ and $\Im\Ga=\Ker L$.
\item Suppose $\Ga\colon\T\to\T$ is an exact colocalization functor and
$L\colon\T\to\T$ the functor which is defined in terms of the exact
triangle \eqref{eq:extria}. Then $L$ is an exact localization functor
with $\Ker L=\Im\Ga$ and $\Im L=\Ker \Ga$.
\end{enumerate}
\end{prop}
\begin{proof}
It suffices to prove (1) because (2) is the dual statement.  So let
$L\colon\T\to\T$ be an exact localization functor. It follows from the
construction of $\Ga$ that it is of the form $\Ga=I\comp I_\r$ where
$I_\r$ denotes a right adjoint of the fully faithful inclusion
$I\colon\Ker L\to\T$. Thus $\Ga$ is a colocalization functor by
Corollary~\ref{co:loc}.  The exactness of $\Ga$ follows from
Lemma~\ref{le:exact-adj}, and the identities $\Ker \Ga=\Im L$ and
$\Im\Ga=\Ker L$ are easily derived from the exact triangle
\eqref{eq:extria}.
\end{proof}

\subsection{Recollements}
A \emph{recollement} is by definition a diagram of exact functors
\begin{equation}\label{eq:recol}
\xymatrix{\T'\,\ar[rr]|-I&&\,\T\,
\ar[rr]|-Q\ar@<1.25ex>[ll]^-{I_\la}\ar@<-1.25ex>[ll]_-{I_\r}&&
\,\T''\ar@<1.25ex>[ll]^-{Q_\la}\ar@<-1.25ex>[ll]_-{Q_\r} }
\end{equation}
satisfying the following conditions.
\begin{enumerate}
\item $I_\lambda$ is a left adjoint and  $I_\rho$ a right adjoint of $I$.
\item $Q_\lambda$ is a left adjoint and  $Q_\rho$ a right adjoint of $Q$.
\item $I_\la I\cong\Id{\T'}\cong I_\rho I$ and $QQ_\rho
\cong\Id{\T''}\cong QQ_\la$.
\item $\Im I=\Ker Q$.
\end{enumerate}
Note that the isomorphisms in (3) are supposed to be the adjunction
morphisms resulting from (1) and (2).

A recollement gives rise to various localization and colocalization
functors for $\T$. First observe that the functors $I$, $Q_\la$, and
$Q_\rho$ are fully faithful; see Proposition~\ref{pr:quot}. Therefore
$Q_\rho Q$ and $II_\la$ are localization functors and $Q_\la Q$ and
$II_\rho$ are colocalization functors. This follows from
Corollary~\ref{co:loc}. Note that the localization functor $L=Q_\rho Q$
has the additional property that the inclusion $\Ker L\to \T$ admits a
left adjoint. Moreover, $L$ determines the recollement up to an
equivalence.

\begin{prop}
Let $L\colon\T\to\T$ be an exact localization functor and suppose that the
inclusion $\Ker L\to \T$ admits a left adjoint. Then $L$ induces a
recollement of the following form.
$$\xymatrix{\Ker L\,\ar[rr]|-\inc&&\,\T\,
\ar[rr]\ar@<1.25ex>[ll]\ar@<-1.25ex>[ll]&& \,\Im
L\ar@<1.25ex>[ll]\ar@<-1.25ex>[ll]}$$ Moreover, any recollement for
$\T$ is, up to equivalences, of this form for some exact localization
functor $L\colon\T\to\T$.
\end{prop}
\begin{proof}
We apply Proposition~\ref{pr:tria-loc} and its dual assertion.
Observe first that any localization functor $L\colon\T\to\T$ induces
the following diagram.
$$\xymatrix{\Ker L\,\ar@<-.7ex>[rr]_-{I=\inc}&&\,\T\,
\ar@<-.7ex>[rr]_-{Q=L}\ar@<-.7ex>[ll]_-{I_\rho=\Ga}&& \,\Im
L\ar@<-.7ex>[ll]_-{Q_\rho=\inc}}$$ The functor $I$ admits a left
adjoint if and only if $Q$ admits a left adjoint. Thus the diagram can
be completed to a recollement if and only if the inclusion $I$ admits
a left adjoint.

Suppose now there is given a recollement of the form \eqref{eq:recol}.
Then $L=Q_\rho Q$ is a localization functor and the inclusion $\Ker
L\to \T$ admits a left adjoint. The functor $I$ induces an equivalence
$\T'\xto{\sim}\Ker L$ and $Q_\rho$ induces an equivalence
$\T''\xto{\sim}\Im L$. It is straightforward to formulate and check
the various compatibilities of these equivalences.
\end{proof}

As a final remark, let us mention that for any recollement of the form
\eqref{eq:recol}, the functors $Q_\la$ and $Q_\rho$ provide two (in
general different) embeddings of $\T''$ into $\T$. If we identify
$\T'=\Im I$, then $Q_\r$ identifies $\T''$ with $(\T')^\perp$ and
$Q_\la$ identifies $\T''$ with $^\perp(\T')$; see
Proposition~\ref{pr:local-acyclic}.

\subsection{Example: The derived category of a module category}

Let $A$ be an associative ring. We denote by $\bfK(\Mod A)$ the
category of chain complexes of $A$-modules whose morphisms are the
homotopy classes of chain maps. The functor $H^n\colon\bfK(\Mod
A)\to\Mod A$ taking the cohomology of a complex in degree $n$ is
cohomological. A morphism $\p$ is called \emph{quasi-isomorphism} if
$H^n\p$ is an isomorphism for all $n\in\bbZ$, and we denote the set
of all quasi-isomorphisms by $\qis$. Then
$$\bfD(A):=\bfD(\Mod A):=\bfK(\Mod A)[\qis^{-1}]$$ is by definition
the \emph{derived category} of $\Mod A$. The kernel of the quotient
functor $Q\colon\bfK(\Mod A)\to\bfD(\Mod)$ is the full subcategory
$\bfK_\ac(\Mod A)$ which is formed by all acyclic complexes. Note that
$Q$ admits a left adjoint $Q_\la$ taking each complex to its
K-projective resolution and a right adjoint $Q_\r$ taking each complex
to its K-injective resolution. Thus we obtain the following recollement.
\begin{equation}\label{eq:DA}
\xymatrix{\bfK_\ac(\Mod A)\,\ar[rr]|-\inc&&\,\bfK(\Mod A)\,
\ar[rr]|-Q\ar@<1.25ex>[ll]^-{}\ar@<-1.25ex>[ll]_-{}&&
\,\bfD(\Mod A)\ar@<1.25ex>[ll]^-{Q_\la}\ar@<-1.25ex>[ll]_-{Q_\r} }
\end{equation}
It follows that for each pair of chain complexes $X,Y$ the set of
morphisms $\bfD(\Mod A)(X,Y)$ is small, since $Q_\la$ induces a
bijection with $\bfK(\Mod A)(Q_\la X,Q_\la Y)$. The adjoints of $Q$
are discussed in more detail in Section~\ref{se:complexes}.

\subsection{Example: A derived category without small morphism sets}

For any abelian category $\A$, the \emph{derived category} $\bfD(\A)$
is by definition $\bfK(\A)[\qis^{-1}]$. Here, $\bfK(\A)$ denotes the
category of chain complexes in $\A$ whose morphisms are the homotopy
classes of chain maps, and $\qis$ denotes the set of
quasi-isomorphisms. Let us identify objects in $\A$ with chain
complexes concentrated in degree zero. 

We give an example of an abelian category $\A$ and an object $X$
in $\A$ such that the set $\Ext^1_\A(X,X)\cong \bfD(\A)(X,SX)$ is not
small.  This example is taken from Freyd \cite[pp.~131]{Fre1964} and
has been pointed out to me by Neeman.

Let $U$ denote the set of all cardinals of small sets. This set is not
small. Consider the free associative $\bbZ$-algebra $\bbZ\langle
U\rangle$ which is generated by the elements of $U$. Now let $\A=\Mod
A$ denote the category of $A$-modules, where it is assumed that the
underlying set of each module is small. Let $\bbZ$ denote the trivial
$A$-module, that is, $zu=0$ for all $z\in\bbZ$ and $u\in U$.  We claim
that the set $\Ext_\A^1(\bbZ,\bbZ)$ is not small. To see this, define
for each $u\in U$  an $A$-module $E_u=\bbZ\oplus\bbZ$ by
$$(z_1,z_2)x=\begin{cases}(z_2,0)&\quad\text{if $x=u$,}\\
(0,0)&\quad\text{if $x\neq u$,}
\end{cases}$$
where $(z_1,z_2)\in E_u$ and $x\in U$.  Then we have short exact
sequences $0\to \bbZ\xto{\smatrix{1\\ 0}}E_u\xto{\smatrix{0&1}}\bbZ\to
0$ which yield pairwise different elements of $\Ext_\A^1(\bbZ,\bbZ)$ as $u$
runs though the elements in $U$.

\subsection{Example: The recollement induced by an idempotent}

Recollements can be defined for abelian categories in the same way as
for triangulated categories. A typical example arises for any module
category from an idempotent element of the underlying ring.

Let $A$ be an associative ring and $e^2=e\in A$ an idempotent. Then
the functor $F\colon\Mod A\to \Mod eAe$ taking a module $M$ to $Me$
and restriction along $p\colon A\to A/AeA$ induce the following
recollement.
$$\xymatrix{\Mod A/AeA\,\ar[rrr]|-{p_*}&&&\,\Mod A\,
\ar[rrr]|-{F}\ar@<1.25ex>[lll]^-{}\ar@<-1.25ex>[lll]_-{}&&& \,\Mod
eAe\ar@<1.25ex>[lll]^-{-\otimes_{eAe}eA}\ar@<-1.25ex>[lll]_-{\Hom_{eAe}(Ae,-)}
}$$ Note that we can describe adjoints of $F$ since
$$F=\Hom_A(eA,-)=-\otimes_AAe.$$ The recollement for $\Mod A$ induces
the following recollement of triangulated categories for $\bfD(A)$. 
$$\xymatrix{\Ker\bfD(F)\,\ar[rrr]|-{\inc}&&&\,\bfD(A)\,
\ar[rrr]|-{\bfD(F)}\ar@<1.25ex>[lll]^-{}\ar@<-1.25ex>[lll]_-{}&&&
\,\bfD(eAe)\ar@<1.25ex>[lll]^-{-\otimes^\bfL_{eAe}eA}\ar@<-1.25ex>[lll]_-{\RHom_{eAe}(Ae,-)}
}$$ The functor $F$ is exact and $\bfD(F)$ takes by definition a
complex $X$ to $FX$.  The functor
$\bfD(p_*)\colon\bfD(A/AeA)\to\bfD(A)$ identifies $\bfD(A/AeA)$ with
$\Ker\bfD(F)$ if and only if $\Tor_i^A(A/AeA,A/AeA)=0$ for all $i>0$.

\subsection{Notes}
Triangulated categories were introduced independently in algebraic
geometry by Verdier in his th\`ese \cite{V}, and in algebraic topology
by Puppe \cite{P}. Grothendieck and his school used the formalism of
triangulated and derived categories for studying homological
properties of abelian categories. Early examples are Grothendieck
duality and local cohomology for categories of sheaves.  The basic
example of a triangulated category from topology is the stable
homotopy category.

Localizations of triangulated categories are discussed in Verdier's
th\`ese \cite{V}. In particular, he introduced the localization (or
\emph{Verdier quotient}) of a triangulated category with respect to a
triangulated subcategory.  In the context of stable homotopy theory,
it is more common to think of localization functors as endofunctors;
see for instance the work of Bousfield \cite{Bou}, which explains the
term \emph{Bousfield localization}.  The standard reference for
recollements is \cite{BBD}. Resolutions of unbounded complexes were
first studied by Spaltenstein in \cite{Sp}; see also \cite{AH}.

\section{Localization via Brown representability}

\subsection{Brown representatbility}

Let $\T$ be a triangulated category and suppose that $\T$ has small
coproducts. A \emph{localizing subcategory} of $\T$ is by definition a
thick subcategory which is closed under taking small coproducts. A
localizing subcategory of $\T$ is \emph{generated} by a fixed set
of objects if it is the smallest localizing subcategory of
$\T$ which contains this set.

We say that $\T$ is \emph{perfectly generated} by some small set $\Sc$
of objects of $\T$ provided the following holds.
\begin{enumerate}
\item[(PG1)] There is no proper localizing subcategory of $\T$ which
contains $\Sc$.
\item[(PG2)] Given a family $(X_i\to Y_i)_{i\in I}$ of morphisms in
$\T$ such that the induced map $\T(C,X_i)\to\T(C,Y_i)$ is surjective for all
$C\in\Sc$ and $i\in I$,  the induced map $$\T(C,\coprod_i
X_i)\lto\T(C,\coprod_i Y_i)$$ is surjective.
\end{enumerate}
We say that a triangulated category $\T$ with small products
is \emph{perfectly cogenerated} if $\T^\op$ is perfectly generated.

\begin{thm}
\label{th:brown}
Let $\T$ be a triangulated category with small coproducts and suppose
$\T$ is perfectly generated.
\begin{enumerate}
\item A functor $F\colon\T^\op\to\Ab$ is cohomological and sends small
coproducts in $\T$ to products if and only if $F\cong\T(-,X)$ for
some object $X$ in $\T$.
\item An exact functor $\T\to\U$ between triangulated categories preserves 
small coproducts if and only if it has a right adjoint.
\end{enumerate}
\end{thm}
\begin{proof} For a proof of (1) see \cite[Theorem~A]{K}.
To prove (2), suppose that $F$ preserves small coproducts. Then one
defines the right adjoint $G\colon\U\to\T$ by sending an object $X$ in
$\U$ to the object in $\T$ representing $\U(F-,X)$.  Thus
$\U(F-,X)\cong\T(-,GX)$.  Conversely, given a right adjoint of
$F$, it is automatic that $F$ preserves small coproducts.
\end{proof}

\begin{rem}
(1) In the presence of (PG2), condition (PG1) is equivalent to the
    following: For an object $X$ in $\T$, we have $X=0$ if $\T(S^nC,X)=0$ for
    all $C\in\Sc$ and $n\in\bbZ$.

(2) A perfectly generated triangulated category $\T$ has small
products. In fact, Brown representability implies that for any family
of objects $X_i$ in $\T$ the functor $\prod_i\T(-,X_i)$ is represented
by an object in $\T$.
\end{rem}

\subsection{Localization functors via Brown representability}
The existence of localization functors is basically equivalent to the
existence of certain right adjoints; see
Proposition~\ref{pr:tria-loc}. We combine this observation with
Brown's representability theorem and obtain the following.

\begin{prop}
\label{pr:loc-brown}
Let $\T$ be a triangulated category which admits small coproducts and
fix a localizing subcategory $\Sc$.
\begin{enumerate}
\item Suppose $\Sc$ is perfectly generated. Then there exists an exact
localization functor $L\colon \T\to\T$ with $\Ker L=\Sc$.
\item Suppose $\T$ is perfectly generated. Then there exists an exact
localization functor $L\colon \T\to\T$ with $\Ker L=\Sc$ if and only if
the morphisms between any two objects in $\T/\Sc$ form a small set.
\end{enumerate}
\end{prop}
\begin{proof}
The existence of a localization functor $L$ with $\Ker L=\Sc$ is
equivalent to the existence of a right adjoint for the inclusion
$\Sc\to\T$, and it is equivalent to the existence of a right adjoint for the
quotient functor $\T\to\T/\Sc$.  Both functors preserve small
coproducts since $\Sc$ is closed under taking small coproducts; see
Proposition~\ref{pr:coprod}. Now apply Theorem~\ref{th:brown} for the
existence of right adjoints.
\end{proof}

\subsection{Compactly generated triangulated categories}

Let $\T$ be a triangulated category with small coproducts.  An object
$X$ in $\T$ is called \emph{compact} (or \emph{small}) if every
morphism $X\to\coprod_{i\in I}Y_i$ in $\T$ factors through
$\coprod_{i\in J}Y_i$ for some finite subset $J\subseteq I$. Note that
$X$ is compact if and only if the representable functor
$\T(X,-)\colon\T\to\Ab$ preserves small coproducts.  The compact
objects in $\T$ form a thick subcategory which we denote by $\T^c$.

The triangulated category $\T$ is called \emph{compactly generated} if
it is perfectly generated by a small set of compact objects. Observe
that condition (PG2) is automatically satisfied if every object in
$\Sc$ is compact.

A compactly generated triangulated category $\T$ is perfectly
cogenerated. To see this, let $\Sc$ be a set of compact generators.
Then the objects representing $\Hom_\bbZ(\T(C,-),\bbQ/\bbZ)$, where
$C$ runs through all objects in $\Sc$, form a set of perfect
cogenerators for $\T$.

The following proposition is a reformulation of Brown representability
for compactly generated triangulated categories.

\begin{prop}
Let $F\colon\T\to\U$ be an exact functor between triangulated
categories. Suppose that $\T$ has small coproducts and that $\T$ is
compactly generated.
\begin{enumerate}
\item The functor $F$ admits a right adjoint if and only if $F$
preserves small coproducts.
\item The functor $F$ admits a left adjoint if and only if $F$
preserves small products.
\end{enumerate}
\end{prop}

\subsection{Right adjoint functors preserving coproducts}

The following lemma provides a characterization of the fact that a
right adjoint functor preserves small coproducts. This will be useful
in the context of compactly generated categories.

\begin{lem}
\label{le:adj-coprod}
Let $F\colon\T\to\U$ be an exact functor between triangulated
categories which has a right adjoint $G$.
\begin{enumerate}
\item If $G$ preserves small coproducts, then $F$ preserves compactness.
\item If $F$ preserves compactness and $\T$ is generated by compact
objects, then $G$ preserves small coproducts.
\end{enumerate}
\end{lem}
\begin{proof}
Let $X$ be an object in $\T$ and $(Y_i)_{i\in I}$ a family of objects
in $\U$.

(1) We have
\begin{equation}
\label{eq:adj-coprod}
\U(FX,\coprod_i Y_i)\cong\T(X,G(\coprod_i Y_i))\cong\T(X,\coprod_i
GY_i).
\end{equation}
If $X$ is compact, then the isomorphism shows that a morphism
$FX\to\coprod_i Y_i$ factors through a finite coproduct. It follows
that $FX$ is compact.

(2) Let $X$ be compact. Then the canonical morphism $\p\colon\coprod_i
    GY_i\to G(\coprod_i Y_i)$ induces an isomorphism
$$\T (X,\coprod_iGY_i)\cong\coprod_i\T(X,GY_i)\cong\coprod_i\U(FX,Y_i)
\cong\U(FX,\coprod_i Y_i)\cong\T(X,G(\coprod_i Y_i)),$$ 
where the last isomorphism uses that $FX$ is compact. It is easily
checked that the objects $X'$ in $\T$ such that $\T(X',\p)$ is an
isomorphism form a localizing subcategory of $\T$.  Thus $\p$ is an
isomorphism because the compact objects generate $\T$.
\end{proof}

\subsection{Localization functors preserving coproducts}

The following result provides a characterization of the fact that an
exact localization functor $L$ preserves small coproducts; in that
case one calls $L$ \emph{smashing}. The example given below explains
this terminology.

\begin{prop}\label{pr:recol}
Let $\T$ be a category with small coproducts and $L\colon\T\to\T$ an
exact localization functor. Then the following are
equivalent.
\begin{enumerate}
\item The functor $L\colon\T\to\T$ preserves small coproducts.
\item The colocalization functor $\Ga\colon\T\to\T$ with
$\Ker\Ga=\Im L$ preserves small coproducts.
\item The right adjoint of the inclusion functor $\Ker L\to\T$
preserves small coproducts.
\item The right adjoint of the quotient functor $\T\to\T/\Ker L$
preserves small coproducts.
\item The subcategory $\Im L$ of all $L$-local objects is closed
under taking small coproducts.
\end{enumerate}
If $\T$ is perfectly generated, in addition the following is equivalent.
\begin{enumerate}
\item[(6)] 
There exists a recollement of the following form.
\begin{equation}\label{eq:rec}
\xymatrix{\Im L\,\ar[rr]|-\inc&&\,\T\, \ar[rr]\ar@<1.25ex>[ll]\ar@<-1.25ex>[ll]&&
\,\Ker L\ar@<1.25ex>[ll]\ar@<-1.25ex>[ll]
}
\end{equation}
\end{enumerate}
\end{prop}
\begin{proof}
(1) $\Leftrightarrow$ (4) $\Leftrightarrow$ (5) follows from
Proposition~\ref{pr:loc-coprod}.

(1) $\Leftrightarrow$ (2) $\Leftrightarrow$ (3) is easily deduced from
the functorial triangle \eqref{eq:extria} relating $L$ and $\Ga$.

(5) $\Leftrightarrow$ (6): Assume that $\T$ is perfectly
generated. Then we can apply Brown's representability theorem and
consider the sequence
$$\Im L\lto[I]\T\lto[Q]\Ker L$$ where $I$ denotes the inclusion and
$Q$ a right adjoint of the inclusion $\Ker L\to\T$.  Note that $Q$
induces an equivalence $\T/\Im L\xto{\sim}\Ker L$; see
Propositions~\ref{pr:ker_im} and \ref{pr:sym}.  The functors $I$ and
$Q$ have left adjoints. Thus the pair $(I,Q)$ gives rise to a
recollement if and only if $I$ and $Q$ both admit right adjoints. It
follows from Proposition~\ref{pr:tria-loc} that this happens if and
only if $Q$ admits a right adjoint. Now Brown's representability
theorem implies that this is equivalent to the fact that $Q$ preserves
small coproducts. And Proposition~\ref{pr:coprod} shows that $Q$
preserves small coproducts if and only if $\Im L$ is closed under
taking small coproducts. This finishes the proof.
\end{proof}

\begin{rem}
(1) The implication (6) $\Rightarrow$ (5) holds without any extra
    assumption on $\T$.

(2) Suppose an exact localization functor $L\colon\T\to\T$ preserves small
coproducts. If $\T$ is compactly generated, then $\Im L$ is compactly
generated.  This follows from Lemma~\ref{le:adj-coprod}, because
the left adjoint of the inclusion $\Im L\to\T$ sends the compact
generators of $\T$ to compact generators for $\Im L$. A similar argument
shows that $\Im L$ is perfectly generated provided that $\T$ is
perfectly generated.
\end{rem}

\begin{exm}
Let $\Sc$ be the stable homotopy category of spectra and $\wedge$ its
smash product.  Then an exact localization functor $L\colon \Sc\to\Sc$
preserves small coproducts if and only if $L$ is of the form
$L=-\wedge E$ for some spectrum $E$. We sketch the argument. Let $S$
denote the sphere spectrum. There exists a natural morphism $\eta
X\colon X\wedge LS\to LX$ for each $X$ in $\Sc$. Suppose that $L$
preserves small coproducts. Then the subcategory of objects $X$ in
$\Sc$ such that $\eta X$ is invertible contains $S$ and is closed under
forming small coproduts and exact triangles. Thus $L=-\wedge E$ for $E=LS$.
\end{exm}

Let $L\colon\T\to\T$ be an exact localization functor which induces a
recollement of the form \eqref{eq:rec}.  Then the sequence
$\Ker L\to\T\to\Im L$ of left adjoint functors induces a sequence
$$(\Ker L)^c\lto\T^c\lto (\Im L)^c$$ of exact functors, by
Lemma~\ref{le:adj-coprod}.  This sequence is of some interest. The
functor $(\Ker L)^c\to\T^c$ is fully faithful and identifies $(\Ker
L)^c$ with a thick subcategory of $\T^c$, whereas the functor $\T^c\to
(\Im L)^c$ shares some formal properties with a quotient functor.  A
typical example arises from finite localization; see
Theorem~\ref{th:finite}. However, there are examples where $\T$ is
compactly generated but $0=(\Ker L)^c\subseteq\Ker L\neq 0$; see
\cite{Kr2005} for details.

\subsection{Finite localization}

A common type of localization for triangulated categories is finite
localization. Here, we explain the basic result and refer to our discussion of well
generated categories for a more general approach and further details.

Let $\T$ be a compactly generated triangulated category and suppose we
have given a subcategory $\Sc'\subseteq\T^c$.  Let $\Sc$ denote the
localizing subcategory generated by $\Sc'$.  Then $\Sc$ is compactly
generated and therefore the inclusion functor $\Sc\to\T$ admits a right
adjoint by Brown's representability theorem. In particular, we have a
localization functor $L\colon\T\to\T$ with $\Ker L=\Sc$ and the
morphisms between any pair of objects in $\T/\Sc$ form a small set; see
Proposition~\ref{pr:loc-brown}.  We observe that the compact objects
of $\Sc$ identify with the smallest thick subcategory of $\T^c$
containing $\Sc'$. This follows from Corollary~\ref{co:loc-gen}.  Thus
we obtain the following commutative diagram of exact functors.
$$\xymatrix{\Sc^c\ar[d]^-\inc\ar[rr]^-\inc&&\T^c\ar[d]^-\inc\ar[rr]^-\can&&
\T^c/\Sc^c\ar[d]^-J\\
\Sc\ar[rr]^-{I=\inc}&&\T\ar[rr]^-{Q=\can}&&\T/\Sc}$$

\begin{thm}\label{th:finite}
Let $\T$ and $\Sc$ be as above. Then the quotient category $\T/\Sc$ is
compactly generated. The induced exact functor
$J\colon\T^c/\Sc^c\to\T/\Sc$ is fully faithful and the category
$(\T/\Sc)^c$ of compact objects equals the full subcategory consisting
of all direct factors of objects in the image of $J$.  Moreover, the
inclusion $\Sc^\perp\to\T$ induces the following recollement.
$$\xymatrix{\Sc^\perp\,\ar[rr]|-{\inc}&&\,\T\,
\ar[rr]\ar@<1.25ex>[ll]\ar@<-1.25ex>[ll]&&
\,\Sc\ar@<1.25ex>[ll]\ar@<-1.25ex>[ll] }$$
\end{thm}
\begin{proof}
The inclusion $I$ preserves compactness and therefore the right
adjoint $I_\r$ preserves small coproducts by
Lemma~\ref{le:adj-coprod}. Thus $Q_\r$ preserve small coproducts by
Proposition~\ref{pr:recol}, and therefore $Q$ preserves compactness,
again by Lemma~\ref{le:adj-coprod}.  It follows that $J$ induces a functor
$\T^c/\Sc^c\to (\T/\Sc)^c$.  In particular, $Q$ sends a set of compact
generators of $\T$ to a set of compact generators for $\T/\Sc$.

Next we apply Lemma~\ref{le:tria-sub} to show that $J$ is fully
faithful. For this, one needs to check that every morphism from a
compact object in $\T$ to an object in $\Sc$ factors through some
object in $\Sc^c$. This follows from Theorem~\ref{th:set-local}. The image
of $J$ is a full triangulated subcategory of $\T^c$ which generates
$\T/\Sc$.  Another application of Corollary~\ref{co:loc-gen} shows that
every compact object of $\T/\Sc$ is a direct factor of some object in
the image of $J$.

Let $L\colon\T\to\T$ denote the localization functor with $\Ker
L=\Sc$. Then $\Sc^\perp$ equals the full subcategory of $L$-local
objects. This subcategory is closed under small coproducts since $\Sc$
is generated by compact objects. Thus the existence of the recollement
follows from Proposition~\ref{pr:recol}.
\end{proof}

\subsection{Cohomological localization via localization of graded modules}

Let $\T$ be a triangulated category which admits small
coproducts. Suppose that $\T$ is generated by a small set of compact
objects. We fix a graded\footnote{All graded rings and modules are graded over
$\bbZ$. Morphisms between graded modules are degree zero maps.} ring
$\La$ and a graded cohomological functor
$$H^*\colon\T\lto\A$$ into the category $\A$ of graded $\La$-modules.
Thus $H^*$ is a functor which sends each exact triangle in $\T$ to an
exact sequence in $\A$, and we have an isomorphism $H^*\comp S\cong
T\comp H^*$ where $T$ denotes the shift functor for $\A$.  In
addition, we assume that $H^*$ preserves small products and
coproducts.

\begin{thm}
\label{th:locloc}
Let $L\colon\A\to\A$ be an exact localization functor for the category
$\A$ of graded $\La$-modules. Then there exists an exact localization
functor $\tilde L\colon\T\to\T$ such that the following square
commutes up to a natural isomorphism.
$$\xymatrix{ \T\ar[rr]^-{\tilde L}\ar[d]^-{H^*}&&\T\ar[d]^-{H^*}\\
\A\ar[rr]^-L&&\A }$$ More precisely, the adjunction morphisms
$\Id\A\to L$ and $\Id\T\to\tilde L$ induce for each $X$ in $\T$ the following
isomorphisms.
$$H^*\tilde LX\stackrel{\sim}\longrightarrow L(H^*\tilde
LX)=LH^*(\tilde LX)\stackrel{\sim}\longleftarrow LH^*(X)$$ An object
$X$ in $\T$ is $\tilde L$-acyclic if and only if $H^*X$ is
$L$-acyclic.  If an object $X$ in $\T$ is $\tilde L$-local, then
$H^*X$ is $L$-local. The converse holds, provided that $H^*$ reflects
isomorphisms.
\end{thm}
\begin{proof}
We recall that $\T$ is perfectly cogenerated because it is compactly
generated.  Thus Brown's representability theorem provides a compact
object $C$ in $\T$ such that
$$H^*X\cong\T(C,X)^*:=\coprod_{i\in\bbZ}\T(C,S^i Y)\quad\text{for
all}\quad X\in\T.$$

Now consider the essential image $\Im L$ of $L$ which equals the full
subcategory formed by all $L$-local objects in $\A$. Because
$L$ is exact, this subcategory is {\em coherent}, that is, for any
exact sequence $X_1\to X_2\to X_3\to X_4\to X_5$ with
$X_1,X_2,X_4,X_5\in\A$, we have $X_3\in\A$.  This is an immediate
consequence of the 5-lemma.  In addition, $\Im L$ is closed under
taking small products. The $L$-local objects form an abelian
Grothendieck category and therefore $\Im L$ admits an injective
cogenerator, say $I$; see \cite{G}.  Using again Brown's
representability theorem, there exists $\tilde I$ in $\T$ such that
\begin{equation}\label{eq:1}\A(H^*-,I)\cong\T(-,\tilde I)\quad\text{and therefore}
\quad\A(H^*-,I)^*\cong\T(-,\tilde I)^*.
\end{equation} 
Now consider the subcategory $\V$ of $\T$ which is formed by all
objects $X$ in $\T$ such that $H^*X$ is $L$-local. This is a triangulated
subcategory which is closed under taking small products.  Observe that
$\tilde I$ belongs to $\V$. To prove this, take a free presentation
$$F_1\lto F_0\lto H^*C\lto 0$$ over $\La$ and apply $\A(-,I)^*$ to
it. Using the isomorphism \eqref{eq:1}, we see that
$H^*\tilde I$ belongs to $\Im L$ because $\Im L$ is coherent and closed
under taking small products.

Now let $\U$ denote the smallest triangulated subcategory of $\T$
containing $\tilde I$ and closed under taking small products. Observe
that $\U\subseteq \V$. We claim that $\U$ is perfectly cogenerated by
$\tilde I$. Thus, given a family of morphisms $\p_i\colon X_i\to Y_i$
in $\U$ such that $\T(Y_i,\tilde I)\to\T(X_i,\tilde I)$ is surjective
for all $i$, we need to show that $\T(\prod_iY_i,\tilde
I)\to\T(\prod_iX_i,\tilde I)$ is surjective. We argue as follows. If
$\T(Y_i,\tilde I)\to\T(X_i,\tilde I)$ is surjective, then the
isomorphism \eqref{eq:1} implies that $H^*\p_i$ is a monomorphism
since $I$ is an injective cogenerator for $\Im L$. Thus the product
$\prod_i\p_i\colon \prod_iX_i\to \prod_iY_i$ induces a monomorphism
$H^*\prod_i\p_i=\prod_iH^*\p_i$ and therefore $\T(\prod_i\p_i,\tilde
I)$ is surjective.  We conclude from Brown's representability theorem
that the inclusion functor $G\colon \U\to\T$ has a left adjoint
$F\colon\T\to\U$. Thus $\tilde L=G\comp F$ is a localization functor
by Corollary~\ref{co:loc}.

Next we show that an object $X\in\T$ is $\tilde L$-acyclic if and only
if $H^*X$ is $L$-acyclic. This follows from
Proposition~\ref{pr:local-acyclic} and the isomorphism \eqref{eq:1},
because we have
$$\tilde LX=0 \iff \T(X,\tilde I)=0\iff\A(H^*X, I)=0\iff LH^*X=0.$$

Now denote by $\eta\colon\Id\A\to L$ and $\tilde\eta\colon\Id\T\to
\tilde L$ the adjunction morphisms and consider the following
commutative square.
\begin{equation}\label{eq:2}
\xymatrix{H^*X\ar[rr]^-{\eta H^*X}\ar[d]^-{H^*\tilde\eta X}
&&LH^*X\ar[d]^-{LH^*\tilde\eta X}\\ H^*\tilde LX\ar[rr]^-{\eta
H^*\tilde LX}&&LH^*\tilde LX}
\end{equation} 
We claim that $LH^*\tilde\eta X$ and $\eta H^*\tilde L X$ are
invertible for each $X$ in $\T$. The morphism $\tilde\eta X$ induces an
exact triangle
$$X'\to X\xto{\tilde\eta X}\tilde LX\to S X'$$ with $\tilde
LX'=0=\tilde LS X'$. Applying the cohomological functor $LH^*$, we
see that $LH^*\tilde\eta X$ is an isomorphism, since $LH^*X'=0=LH^*S
X'$.  Thus $LH^*\tilde\eta$ is invertible.  The morphism $\eta H^*\tilde
LX$ is invertible because $H^*\tilde LX$ is $L$-local. This follows
from the fact that $\tilde LX$ belongs to $\U$.

The commutative square \eqref{eq:2} implies that $H^*\tilde\eta X$ is
invertible if and only if $\eta H^*X$ is invertible.  Thus if $X$ is
$\tilde L$-local, then $H^*X$ is $L$-local. The converse holds if
$H^*$ reflects isomorphisms. 
\end{proof}

\begin{rem}
(1) The localization functor $\tilde L$ is essentially uniquely determined
by $H^*$ and $L$, because $\Ker \tilde L=\Ker LH^*$.

(2) Suppose that $C$ is a generator of $\T$. If $L$ preserves small
coproducts, then it follows that $\tilde L$ preserves small
coproducts.  In fact, the assumption implies that $H^*\tilde L$
preserves small coproducts, since $LH^*\cong H^*\tilde L$. But $H^*$
reflects isomorphisms because $C$ is a generator of $\T$. Thus $\tilde
L$ preserves small coproducts.
\end{rem}

\subsection{Example: Resolutions of chain complexes}
\label{se:complexes}

Let $A$ be an associative ring. Then the derived category $\bfD(A)$ of
unbounded chain complexes of modules over $A$ is compactly
generated. A compact generator is the ring $A$, viewed as a complex
concentrated in degree zero. Let us be more precise, because we want
to give an explicit construction of $\bfD(A)$ which implies that the
morphisms between any two objects in $\bfD(A)$ form a small set.
Moreover, we combine Brown representability with
Proposition~\ref{pr:tria-loc} to provide descriptions of the adjoints
$Q_\la$ and $Q_\r$ of the quotient functor $Q\colon\bfK(\Mod
A)\to\bfD(A)$ which appear in the recollement \eqref{eq:DA}.

Denote by $\Loc A$ the localizing subcategory of $\bfK(\Mod A)$ which
is generated by $A$. Then $\Loc A$ is a compactly generated
triangulated category and $(\Loc A)^\perp=\bfK_\ac(\Mod A)$ since
$$\bfK(\Mod A)(A,S^nX)\cong H^nX.$$ Brown representability provides a
right adjoint for the inclusion $\Loc A\to\bfK(\Mod A)$ and therefore
the composite $F\colon\Loc A\xto{\inc}\bfK(\Mod A)\xto{\can}\bfD(A)$
is an equivalence by Proposition~\ref{pr:tria-loc}. The right adjoint
of the inclusion $\Loc A\to\bfK(\Mod A)$ annihilates the acyclic
complexes and induces therefore a functor $\bfD(A)\to\Loc A$ (which is
a quasi-inverse for $F$). The composite with the inclusion $\Loc
A\to\bfK(\Mod A)$ is the left adjoint $Q_\la$ of $Q$ and takes a complex
to its \emph{K-projective resolution}.

Now fix an injective cogenerator $I$ for the category of $A$-modules,
for instance $I=\Hom_\bbZ(A,\bbQ/\bbZ)$.  We denote by $\Coloc I$ the
smallest thick subcategory of $\bfK(\Mod A)$ closed under small
products and containing $I$. Then $I$ is a perfect cogenerator for
$\Coloc I$ and $^\perp(\Coloc I)=\bfK_\ac(\Mod A)$ since
$$\bfK(\Mod A)(S^nX, I)\cong \Hom_A(H^nX,I).$$ Brown representability
provides a left adjoint for the inclusion $\Coloc I\to\bfK(\Mod A)$
and therefore the composite $G\colon\Coloc I\xto{\inc}\bfK(\Mod
A)\xto{\can}\bfD(A)$ is an equivalence by
Proposition~\ref{pr:tria-loc}. The left adjoint of the inclusion
$\Coloc I\to\bfK(\Mod A)$ annihilates the acyclic complexes and
induces therefore a functor $\bfD(A)\to\Coloc I$ (which is a
quasi-inverse for $G$). The composition with the inclusion $\Coloc
I\to\bfK(\Mod A)$ is the right adjoint $Q_\r$ of $Q$  and takes a complex
to its \emph{K-injective resolution}.

\subsection{Example: Homological epimorphisms}
Let $f\colon A\to B$ be a ring homomorphism and $f_*\colon \bfD
(B)\to\bfD (A)$ the functor given by restriction of scalars. Clearly,
$f_*$ preserves small products and coproducts.  Thus Brown
representability implies the existence of left and right adjoints for
$f_*$ since $\bfD(A)$ is compactly generated.  For instance, the left
adjoint is the derived tensor functor ${-\otimes^\bfL_A
B}\colon\bfD(A)\to\bfD(B)$ which preserves compactness.

The functor $f_*$ is fully faithful if and only if
$(f_*-)\otimes^\bfL_AB\cong\Id\bfD(B)$ iff $B\otimes_AB\cong B$ and
$\Tor^A_i(B,B)=0$ for all $i>0$.  In that case $f$ is called
\emph{homological epimorphism} and the exact functor
$L\colon\bfD(A)\to\bfD(A)$ sending $X$ to $f_*(X\otimes^\bfL_AB)$ is a
localization functor.

Take for instance a commutative ring $A$ and let $f\colon A\to
S^{-1}A=B$ be the localization with respect to a multiplicatively
closed subset $S\subseteq A$. Then the induced exact localization
functor $L\colon\bfD(A)\to\bfD(A)$ takes a chain complex $X$ to
$S^{-1}X$.  Note that $L$ preserves small coproducts. In particular,
$L$ gives rise to the following recollement.
$$\xymatrix{\bfD(B)\,\ar[rr]|-{f_*}&&\,\bfD(A)\,
\ar[rr]\ar@<1.25ex>[ll]^-{-\otimes^\bfL_AB}\ar@<-1.25ex>[ll]_-{\RHom_A(B,-)}&&
\,\U\ar@<1.25ex>[ll]^-{}\ar@<-1.25ex>[ll]_-{} }$$ The triangulated
category $\U$ is equivalent to the kernel of $-\otimes_A^\bfL B$, and
one can show that $\Ker (-\otimes_A^\bfL B)$ is the localizing
subcategory of $\bfD(A)$ generated by the complexes of the form
$$\cdots\to 0\to A\xto{x}A\to 0\to\cdots\quad (x\in S).$$

\subsection{Notes}
The Brown representability theorem in homotopy theory is due to Brown
\cite{Bro1962}.  Generalizations of the Brown representability theorem
for triangulated categories can be found in work of Franke
\cite{Fra2001}, Keller \cite{Kel1994}, and Neeman
\cite{Nee1996,Nee2001}. The version used here is taken from \cite{K}.
The finite localization theorem for compactly generated triangulated
categories is due to Neeman \cite{Nee1992}; it is based on previous
work of Bousfield, Ravenel, Thomason-Trobaugh, Yao, and others.  The
cohomological localization functors commuting with localization
functors of graded modules have been used to set up local cohomology
functors in \cite{BIK}.

\section{Well generated triangulated categories}
\label{se:wellgen}

\subsection{Regular cardinals}
A cardinal $\a$ is called \emph{regular} if
$\a$ is not the sum of fewer than $\a$ cardinals, all smaller than
$\a$.  For example, $\aleph_0$ is regular because the sum of finitely
many finite cardinals is finite. Also, the successor $\k^+$ of every
infinite cardinal $\k$ is regular.  In particular, there are
arbitrarily large regular cardinals. For more details on regular
cardinals, see for instance \cite{Lev1979}.

\subsection{Localizing subcategories}
Let $\T$ be a triangulated category and $\a$ a regular cardinal. A
coproduct in $\T$ is called \emph{$\a$-coproduct} if it has less than
$\a$ factors. A full subcategory of $\T$ is called
\emph{$\a$-localizing} if it is a thick subcategory and closed under
taking $\a$-coproducts. Given a subcategory $\Sc\subseteq\T$, we denote
by $\Loc_\a\Sc$ the smallest $\a$-localizing subcategory of $\T$ which
contains $\Sc$. Note that $\Loc_\a\Sc$ is  small provided that
$\Sc$ is  small.

A full subcategory of $\T$ is called \emph{localizing} if it is a
thick subcategory and closed under taking small coproducts. The
smallest localizing subcategory containing a subcategory
$\Sc\subseteq\T$ is $\Loc\Sc=\bigcup_\a\Loc_\a\Sc$ where $\a$ runs
through all regular cardinals. We call $\Loc\Sc$ the localizing
subcategory generated by $\Sc$.

\subsection{Well generated triangulated categories}

Let $\T$ be a triangulated category which admits small coproducts and
fix a regular cardinal $\a$. An object $X$ in $\T$ is called
\emph{$\a$-small} if every morphism $X\to\coprod_{i\in I}Y_i$ in $\T$
factors through $\coprod_{i\in J}Y_i$ for some subset $J\subseteq I$
with $\card J<\a$.  The triangulated category $\T$ is called
\emph{$\a$-well generated} if it is perfectly generated by a small set of
$\a$-small objects, and $\T$ is called \emph{well generated} if it is
$\b$-well generated for some regular cardinal $\b$.

Suppose $\T$ is $\a$-well generated by a small set $\Sc$ of $\a$-small
objects. Given any regular cardinal $\b\geq\a$, we denote by $\T^\b$
the $\b$-localizing subcategory $\Loc_\b\Sc$ generated by $\Sc$ and call
the objects of $\T^\b$ \emph{$\b$-compact}. Choosing a representative
for each isomorphism class, one can show that the $\b$-compact objects
form a small set of $\b$-small perfect generators for $\T$. Moreover,
$\T^\b$ does not depend on the choice of $\Sc$. For a proof we refer to
\cite[Lemma~5]{Kr2001}; see also Proposition~\ref{pr:wellgen} and
Remark~\ref{rm:wellgen}. Note that $\T=\bigcup_{\b}\T^\b$, where $\b$
runs through all regular cardinals greater or equal than $\a$, because
$\bigcup_{\b}\T^\b$ is a triangulated subcategory containing $\Sc$ and
closed under small coproducts.

\begin{rem}
The $\a$-small objects of $\T$ form an $\a$-localizing subcategory.
\end{rem}

\begin{exm}
A triangulated category $\T$ is $\aleph_0$-well generated if and only
if $\T$ is compactly generated. In that case $\T^{\aleph_0}=\T^c$.
\end{exm}

\begin{exm}
Let $\A$ be the category of sheaves of abelian groups on a
non-compact, connected manifold of dimension at least $1$. Then the
derived category $\bfD(\A)$ of unbounded chain complexes is well
generated, but the only compact object in $\bfD(\A)$ is the zero
object; see \cite{Nee2001a}. For more examples of well generated but
not compactly generated triangulated categories, see \cite{Nee1998}.
\end{exm}

\subsection{Filtered categories}

Let $\a$ be a regular cardinal. A category $\C$ is called
\emph{$\a$-filtered} if the following holds.
\begin{enumerate}
\item[(FIL1)] There exists an object in $\C$.
\item[(FIL2)] For every family $(X_i)_{i\in I}$ of fewer than $\a$
objects there exists an object $X$ with morphisms $X_i\to X$ for all
$i$.
\item[(FIL3)] For every family $(\p_i\colon X\to Y)_{i\in I}$ of fewer
than $\a$ morphisms there exists a morphism $\psi\colon Y\to Z$ with
$\psi\p_i=\psi\p_j$ for all $i$ and $j$.
\end{enumerate}
One drops the cardinal $\a$ and calls $\C$ \emph{filtered} in case it is $\aleph_0$-filtered.

Given a functor $F\colon\C\to\D$, we use the term \emph{$\a$-filtered
colimit} for the colimit $\colim{X\in\C}FX$ provided that $\C$ is a
 small $\a$-filtered category.

\begin{lem}\label{le:cofinal}
Let $i\colon \C'\to \C$ be a fully faithful functor with $\C$ a small
$\a$-filtered category.  Suppose that $i$ is cofinal in the sense that
for any $X\in\C$ there is an object $Y\in\C'$ and a morphism $X\to
iY$. Then $\C'$ is a small $\a$-filtered category, and for any functor
$F\colon \C\to\D$ into a category which admits $\a$-filtered colimits,
the natural morphism
$$\colim{Y\in\C'}F(iY)\lto \colim{X\in\C}FX$$ is an isomorphism.
\end{lem}
\begin{proof}
See \cite[Proposition~8.1.3]{GV}.
\end{proof}
A full subcategory $\C'$ of a  small $\a$-filtered
category $\C$ is called \emph{cofinal} if for any $X\in\C$ there is an object
$Y\in\C'$ and a morphism $X\to Y$.

\subsection{Comma categories}

Let $\T$ be a triangulated category which admits small coproducts and
fix a full subcategory $\Sc$.  Given an object $X$ in $\T$, let $\Sc/X$
denote the category whose objects are pairs $(C,\m)$ with $C\in\Sc$ and
$\m\in\T(C,X)$. The morphisms $(C,\m)\to (C',\m')$ are the morphisms
$\g\colon C\to C'$ in $\T$ making the following diagram commutative.
$$\xymatrixcolsep{1pc}
\xymatrix{C\ar[rr]^\g\ar[rd]_\m&&C'\ar[ld]^{\m'}\\&X }$$
Analogously, one defines for a morphism $\p\colon X\to X'$ in $\T$ the
category $\Sc/\p$ whose objects are 
commuting squares in $\T$ of the form
$$\xymatrix{C\ar[r]^\g\ar[d]_\m&C'\ar[d]^{\m'}\\X\ar[r]^\p&X' }$$
with $C,C'\in\Sc$.

\begin{lem}
Let $\a$ be a regular cardinal and $\Sc$ an $\a$-localizing subcategory
of $\T$. Then the categories $\Sc/X$ and $\Sc/\p$ are $\a$-filtered for
each object $X$ and each morphism $\p$ in $\T$.
\end{lem}
\begin{proof}
Straightforward.
\end{proof}

\subsection{The comma category of an exact triangle}
Let $\T$ be a triangulated category. We consider the category of pairs
$(\p_1,\p_2)$ of composable morphisms $X_1\xto{\p_1}X_2\xto{\p_2}X_3$
in $\T$. A morphism $\m\colon (\p_1,\p_2)\to (\p'_1,\p'_2)$ is a
triple $\m=(\m_1,\m_2,\m_3)$ of morphisms in $\T$ making the following
diagram commutative.
$$\xymatrix{X_1\ar[r]^{\p_1}\ar[d]^{\m_1}&X_2\ar[r]^{\p_2}\ar[d]^{\m_2}&
X_3\ar[d]^{\m_3}\\ X'_1\ar[r]^{\p'_1}&X'_2\ar[r]^{\p'_2}&X'_3}$$ A pair
$(\p_1,\p_2)$ of composable morphisms is called \emph{exact} if it fits into
an exact triangle $X_1\xto{\p_1}X_2\xto{\p_2}X_3\xto{\p_3}SX_1$.

\begin{lem}
Let $\m\colon (\g_1,\g_2)\to (\p_1,\p_2)$ be a morphism between pairs
of composable morphisms and suppose that $(\p_1,\p_2)$ is exact.  Then
$\m$ factors through an exact pair of composable morphisms which
belong to the smallest full triangulated subcategory containing $\g_1$
and $\g_2$.
\end{lem}
\begin{proof}
We proceed in two steps. The first step provides a factorization of
$\m$ through a pair $(\g'_1,\g'_2)$ of composable morphisms such that
$\g'_2\g'_1=0$.  To achieve this, complete $\g_1$ to an exact triangle
$C_1\xto{\g_1}C_2\xto{\bar\g_2}\bar C_3\to SC_1$. Note that $\p_2\m_2$
factors through $\bar\g_2$.  Now complete $\smatrix{\g_2\\ \bar\g_2}$
to an exact triangle $C_2\xto{\smatrix{\g_2\\ \bar\g_2}} C_3\amalg
\bar C_3\xto{\smatrix{\d&\bar\d}} C'_3\to SC_2$ and observe that
$\m_3$ factors through $\d$ via a morphism $\m_3'\colon C'_3\to
X_3$. Thus we obtain the following factorization of $\m$ with
$(\d\g_2)\g_1=-\bar\d\bar\g_2\g_1=0$.
$$\xymatrix{C_1\ar[rr]^{\g_1}\ar[d]^{\id}&&C_2\ar[rr]^{\g_2}\ar[d]^{\id}&&
C_3\ar[d]^{\d}\\
C_1\ar[rr]^{\g_1}\ar[d]^{\m_1}&&C_2\ar[rr]^{\d\g_2}\ar[d]^{\m_2}&&
C'_3\ar[d]^{\m_3'}\\ X_1\ar[rr]^{\p_1}&&X_2\ar[rr]^{\p_2}&&X_3}$$

For the second step we may assume that $\g_2\g_1=0$. We complete
$\g_2$ to an exact triangle $\bar C_1\xto{\bar\g_1}C_2\xto{\g_2}
C_3\to S\bar C_1$. Clearly, $\g_1$ factors through $\bar\g_1$ via a
morphism $\r\colon C_1\to \bar C_1$ and $\m_2\bar\g_1$ factors through
$\p_1$ via a morphism $\s\colon \bar C_1\to X_1$.  Thus we obtain the
following factorization of $\m$
$$\xymatrix{C_1\ar[rr]^{\g_1}\ar[d]^{\smatrix{\r\\
\id}}&&C_2\ar[rr]^{\g_2}\ar[d]^{\id}&& C_3\ar[d]^{\smatrix{\id \\ 0}}\\ \bar
C_1\amalg
C_1\ar[rr]^{\smatrix{\bar\g_1&0}}\ar[d]^{\smatrix{\s&\m_1-\s\r}}&&C_2\ar[rr]^{\smatrix{\g_2\\
0}}\ar[d]^{\m_2}&& C_3\amalg SC_1\ar[d]^{\smatrix{\m_3&0}}\\
X_1\ar[rr]^{\p_1}&&X_2\ar[rr]^{\p_2}&&X_3}$$
where the middle row fits into an exact triangle.
\end{proof}

The following statement is a reformulation of the previous one in
terms of cofinal subcategories.

\begin{prop}
\label{pr:cof-triangle}
Let $\T$ be a triangulated category and $\Sc$ a full triangulated
subcategory. Suppose that
$X_1\xto{\p_1}X_2\xto{\p_2}X_3\xto{\p_3}SX_1$ is an exact triangle in
$\T$ and denote by $\Sc/(\p_1,\p_2)$ the category whose objects are
the commutative diagrams in $\T$ of the following form.
$$\xymatrix{C_1\ar[r]^{\g_1}\ar[d]&C_2\ar[r]^{\g_2}\ar[d]&C_3\ar[d]\\
X_1\ar[r]^{\p_1}&X_2\ar[r]^{\p_2}&X_3
}$$ such each $C_i$ belongs to $\Sc$.
Then the full subcategory formed by the diagrams such that there exists an exact triangle
$C_1\xto{\g_1}C_2\xto{\g_2}C_3\xto{\g_3}SC_1$ is a cofinal subcategory of $\Sc/(\p_1,\p_2)$.
\end{prop}

\subsection{A Kan extension}
Let $\T$ be a triangulated category with small coproducts and $\Sc$ a
 small full subcategory. Suppose that the objects of $\Sc$
are $\a$-small and that $\Sc$ is closed under $\a$-coproducts. We
denote by $\Add_\a(\Sc^\op,\Ab)$ the category of $\a$-product
preserving functors $\Sc^\op\to\Ab$. This is a locally presentable
abelian category in the sense of \cite{GU} and we refer to the
Appendix~\ref{se:locpres} for basic facts on locally presentable
categories. Depending on the choice of $\Sc$, we can think of
$\Add_\a(\Sc^\op,\Ab)$ as a locally presentable approximation of the
triangulated category $\T$. In order to make this precise, we need to
introduce various functors.

Let $h_\T\colon\T\to A(\T)$ denote the abelianization of $\T$; see
Appendix~\ref{se:abel}. Sometimes we write $\widehat\T$ instead of
$A(\T)$.  The inclusion functor $f\colon \Sc\to\T$ induces a functor
$$f_*\colon A(\T)\lto \Add_\a(\Sc^\op,\Ab),\quad X\mapsto
A(\T)((h_\T\comp f)-,X),$$  and we observe that the composite
$$\T\lto[h_\T]A(\T)\lto[f_*]\Add_\a(\Sc^\op,\Ab)$$ is the restricted
Yoneda functor sending each $X\in\T$ to $\T(-,X)|_\Sc$.

The next proposition discusses a left adjoint for $f_*$.  To this end,
we denote for any category $\C$ by $h_\C$ the Yoneda functor sending
$X$ in $\C$ to $\C(-,X)$.

\begin{prop}\label{pr:A(f)}
The functor $ f_*$ admits a left adjoint $ f^*$ which makes
the following diagram commutative.
$$\xymatrix{
\Sc\ar[rr]^-{h_\Sc}\ar[d]^{f=\inc}&&
\Add_\a(\Sc^\op,\Ab)\ar[d]^{f^*}\\
\T\ar[rr]^-{h_\T}&&A(\T) }$$
Moreover, the functor $ f^*$ has the following properties.
\begin{enumerate}
\item $ f^*$ is fully faithful and identifies
$\Add_\a(\Sc^\op,\Ab)$ with the full subcategory formed by all
colimits of objects in $\{\T(-,X)\mid X\in \Sc\}$.
\item $ f_*$ preserves small coproducts if and only if \emph{(PG2)}
holds for $\Sc$.
\item Suppose that $\Sc$ is a triangulated subcategory of $\T$. Then for $X$
in $\T$, the adjunction morphism $ f^* f_*(h_\T X)\to h_\T X$ identifies
with the canonical morphism $$\colim{(C,\m)\in \Sc/X}h_\T C\lto h_\T
X.$$
\end{enumerate}
\end{prop}
\begin{proof} 
The functor $f^*$ is constructed as a left Kan extension. To explain
this, it is convenient to identify $\Add_\a(\Sc^\op,\Ab)$ with the
category $\Lex_\a(\widehat\Sc^\op,\Ab)$ of left exact functors
$\widehat\Sc^\op\to\Ab$ which preserve $\a$-products. To be more
precise, the Yoneda functor $h\colon\Sc\to \widehat\Sc$ induces an
equivalence
$$\Lex_\a(\widehat\Sc^\op,\Ab)\lto[\sim]\Add_\a(\Sc^\op,\Ab),\quad
F\mapsto F\comp h,$$ because  every additive functor
$\Sc^\op\to\Ab$ extends uniquely to a left exact functor
$\widehat\Sc^\op\to\Ab$; see Lemma~\ref{le:coherent}.

Using this identification, the existence of a fully faithful left
adjoint $\Lex_\a(\widehat\Sc^\op,\Ab)\to A(\T)$ for $ f_*$ and its
basic properties follow from Lemma~\ref{le:kan-inc}, because the
inclusion $ f\colon \Sc\to\T$ induces a fully faithful and right exact
functor $\widehat f\colon\widehat\Sc\to\widehat\T=A(\T)$. This functor
preserves $\a$-coproducts and identifies $\widehat \Sc$ with a full
subcategory of $\a$-presentable objects, since the objects from $\Sc$
are $\a$-small in $\T$.

(2) Let $\Si=\Si( f_*)$ denote the set of morphisms of $A(\T)$ which
$ f_*$ makes invertible. It follows from Proposition~\ref{pr:quot}
that $ f_*$ induces an equivalence
$$A(\T)[\Si^{-1}]\lto[\sim]\Lex_\a(\widehat\Sc^\op,\Ab),$$ and therefore
$ f_*$ preserves small coproducts if and only if $\Si$ is closed
under taking small coproducts, by Proposition~\ref{pr:coprod}. It is
not hard to check that $ f_*$ is exact, and therefore a morphism in
$A(\T)$ belongs to $\Si$ if and only if its kernel and cokernel are
annihilated by $ f_*$. Now observe that an object $F$ in $A(\T)$
with presentation $\T(-,X)\to \T(-,Y)\to F\to 0$ is annihilated by
$ f_*$ if and only if $\T(C,X)\to \T(C,Y)$ is surjective for all
$C\in\Sc$. It follows hat $ f_*$ preserves small coproducts if and
only if (PG2) holds for $\Sc$.

(3)  Let $F= f_*(h_\T X)=\T(-,X)|_\Sc$. Then Lemma~\ref{le:flat} implies that
$F=\colim{(C,\m)\in\Sc/X}h_\Sc C$, since $\Sc/F=\Sc/X$.  Thus
$f^*F=\colim{(C,\m)\in\Sc/X}h_\T C$.
\end{proof}

\begin{cor}
\label{co:kan}
Let $\T$ be a triangulated category with small coproducts. Suppose
$\T$ is $\a$-well generated and denote by $\T^\a$ the full subcategory
formed by all $\a$-compact objects. Then the functor $\T\to A(\T)$
taking an object $X$ to $$\colim{(C,\m)\in \T^\a/X}\T(-, C)$$ preserves
small coproducts.
\end{cor}

\subsection{A criterion for well generatedness}

Let $\T$ be a triangulated category which admits small coproducts.
The following result provides a useful criterion for $\T$ to be well
generated in terms of cohomological functors into locally presentable
abelian categories.

\begin{prop}\label{pr:wellgen}
Let $\T$ be a triangulated category with small coproducts and $\a$ a
regular cardinal.  Let $\Sc_0$ be a small set of objects and denote by $\Sc$
the full subcategory formed by all $\a$-coproducts of objects in $\Sc_0$.
Then the following are equivalent.
\begin{enumerate}
\item The objects of $\Sc_0$ are $\a$-small and \emph{(PG2)} holds for
$\Sc_0$.
\item The objects of $\Sc$ are $\a$-small and \emph{(PG2)} holds for
$\Sc$.
\item The functor $H\colon\T\to\Add_\a(\Sc^\op,\Ab)$ taking $X$ to $\T(-,X)|_\Sc$ preserves
small coproducts.
\end{enumerate}
\end{prop}
\begin{proof}
It is clear that (1) and (2) are equivalent, and it follows from
Proposition~\ref{pr:A(f)} that (2) implies (3).  To prove that (3)
implies (2), assume that $H$ preserves small coproducts. Let $\p\colon
X\to\coprod_{i\in I}Y_i$ be a morphism in $\T$ with $X\in\Sc$. Write
$\coprod_{i\in I}Y_i=\colim{J\subseteq I}Y_J$ as $\a$-filtered colimit
of coproducts $Y_J=\coprod_{i\in J}Y_i$ with $\card J<\a$.  Then we
have
\begin{align*}\colim{J\subseteq I}\T(X,Y_J)
&\cong\colim{J\subseteq I}\Hom_\Sc(\Sc(-,X),HY_J)\\
&\cong\Hom_\Sc(\Sc(-,X),\colim{J\subseteq I}HY_J)\\
&\cong\Hom_\Sc(\Sc(-,X),\coprod_{i\in I}HY_i)\\
&\cong\Hom_\Sc(\Sc(-,X),H(\coprod_{i\in I}Y_i))\\
&\cong\T(X,\coprod_{i\in I}Y_i).
\end{align*}
Thus $\p$ factors through some $Y_J$, and it follows that $X$ is
$\a$-small. Now  Proposition~\ref{pr:A(f)} implies that (PG2) holds for $\Sc$.
\end{proof}

\subsection{Cohomological functors via filtered colimits}

The following theorem shows that cohomological functors on well
generated triangulated categories can be computed via filtered
colimits. This generalizes a fact which is well known for compactly
generated triangulated categories. We say that an abelian category has
\emph{exact $\a$-filtered colimits} provided that every $\a$-filtered
colimit of exact sequences is exact.

\begin{thm}
\label{th:filtcolim}
Let $\T$ be a triangulated category with small coproducts. Suppose
$\T$ is $\a$-well generated and denote by $\T^\a$ the full subcategory
formed by all $\a$-compact objects.  Let $\A$ be an abelian category
which has small coproducts and exact $\a$-filtered colimits.  If
$H\colon \T\to\A$ is a cohomological functor which preserves small
coproducts, then we have for $X$ in $\T$ a natural isomorphism
\begin{equation}\label{eq:colim}
\colim{(C,\m)\in \T^\a/X}HC\lto[\sim]HX.
\end{equation}
\end{thm}
\begin{proof}
The left hand term of \eqref{eq:colim} defines a functor $\tilde
H\colon \T\to\A$ and we need to show that the canonical morphism
$\tilde H\to H$ is invertible.

First observe that $\tilde H$ is cohomological. This is a consequence
of Proposition~\ref{pr:cof-triangle} and Lemma~\ref{le:cofinal}, because for any
exact triangle $X_1\to X_2\to X_3\to SX_1$ in $\T$, the sequence $\tilde
HX_1\to \tilde HX_2\to \tilde HX_3$ can be written as $\a$-filtered
colimit of exact sequences in $\A$.

Next we claim that $\tilde H$ preserves small coproducts. To this end
consider the exact functor $\bar H \colon A(\T)\to\A$ which extends
$H$; see Lemma~\ref{le:abel}.  Note that $\bar H$ preserve small
coproducts because $H$ has this property. We have for $X$ in $\T$
$$\tilde HX=\colim{(C,\m)\in \T^\a/X}\bar H\big(\T(-,C)\big)\cong
\bar H\big(\colim{(C,\m)\in \T^\a/X}\T(-,C)\big).$$
Now the assertion follows from Corollary~\ref{co:kan}.

To complete the proof, consider the full subcategory $\T'$ consisting
of those objects $X$ in $\T$ such that the morphism $\tilde HX\to HX$
is an isomorphism. Clearly, $\T'$ is a triangulated subcategory since
both functors are cohomological, it is closed under taking small
coproducts since they are preserved by both functors, and it contains
$\T^\a$. Thus $\T'=\T$.
\end{proof}

\begin{rem}
For an alternative proof of the fact that $\tilde H$ is cohomological,
one uses Lemma~\ref{le:yon-exact}.
\end{rem}

\subsection{A universal property}
Let $\T$ be a triangulated category which admits small coproducts and
is $\a$-well generated.  We denote by $A_\a(\T)$ the full subcategory
of $A(\T)$ which is formed by all colimits of objects $\T(-,X)$ with
$X$ in $\T^\a$. Observe that $A_\a(\T)$ is a locally presentable
abelian category with exact $\a$-filtered colimits. This follows from
Proposition~\ref{pr:A(f)} and the discussion in
Appendix~\ref{se:locpres}, because $A_\a(\T)$ can be identified with a
category of left exact functors.

We have two functors
$$H_\a\colon\T\lto A_\a(\T),\quad
X\mapsto\colim{(C,\m)\in\T^\a/X}\T(-,C),$$
$$h_\a\colon\T\lto\Add_\a((\T^\a)^\op,\Ab),\quad X\mapsto\T(-,X)|_{\T^\a},$$
which are related by an equivalence as follows.
$$\xymatrixrowsep{.5pc}\xymatrix{
&&\Add_\a((\T^\a)^\op,\Ab)\ar[dd]^{f^*}_\sim\\
\T\ar[rru]^-{h_\a}\ar[rrd]_-{H_\a}\\ && A_\a(\T)}$$ The functor $f^*$
is induced by the inclusion $f\colon\T^\a\to\T$ and discussed in
Proposition~\ref{pr:A(f)}. In particular, there it is shown that
$f^*(h_\a X)=f^*f_*(h_\T X)=H_\a X$ for all $X$ in $\T$.

\begin{prop}
\label{pr:universal}
The functor $H_\a\colon\T\to A_\a(\T)$ has the following universal
property.
\begin{enumerate}
\item The functor $H_\a$ is a cohomological functor to an abelian
category with small coproducts and exact $\a$-filtered colimits and
$H_\a$ preserves small coproducts.
\item Given a cohomological functor $H\colon\T\to\A$ to an abelian
category with small coproducts and exact $\a$-filtered colimits such
that $H$ preserves small coproducts, there exists an essentially unique
exact functor $\bar H\colon A_\a(\T)\to\A$ which preserves small
coproducts and satisfies $H=\bar H\comp H_\a$.
\end{enumerate}
\end{prop}
\begin{proof}
(1) It is clear that $h_\a$ is cohomological and it follows from
Proposition~\ref{pr:A(f)} that $h_\a$ preserves small coproducts. 

(2) Given $H\colon\T\to\A$, we denote by $\tilde H\colon A(\T)\to\A$
the exact functor which extends $H$, and we define $\bar H\colon
A_\a(\T)\to\A$ by sending each $X$ to $\tilde HX$. The following
commutative diagram illustrates this construction.
$$\xymatrix{
\T^\a\ar[rr]^-{h_{\T^\a}}\ar[d]^{f=\inc}&&A(\T^\a)\ar[rr]^-{h_{A(\T^\a)}}\ar[d]^{A(f)}&&
A_\a(\T)\ar[d]^{f^*=\inc}\\
\T\ar@{=}[d]\ar[rr]^-{h_\T}&&A(\T)\ar[d]^{\tilde
H}\ar@{=}[rr]&&A(\T)\\\T\ar[rr]^H&&\A }$$

Let us check the properties of $\bar H$. The functor $\bar H$
preserves small coproducts since $\tilde H$ has this property. The
functor $\bar H$ is exact when restricted to $A(\T^\a)$. Thus it
follows from Lemma~\ref{le:yon-exact} that $\bar H$ is exact.  The
equality $H=\bar H\comp H_\a$ is a consequence of
Theorem~\ref{th:filtcolim} since both functors coincide on $\T^\a$.
Suppose now there is a second functor $A_\a(\T)\to\A$ having the
properties of $\bar H$. Then both functors agree on $\P=\{\T(-,X)\mid
X\in\T^\a\}$ and therefore on all of $A_\a(\T)$ since each object in
$A_\a(\T)$ is a colimit of objects in $\P$ and both functors preserve
colimits.
\end{proof}

\begin{rem}\label{rm:wellgen}
The universal property can be used to show that the category $\T^\a$
of $\a$-compact objects does not depend on the choice of a perfectly
generating set for $\T$. More precisely, if $\T$ is $\a$-well
generated, then two $\a$-localizing subcategories coincide if each
contains a small set of $\a$-small perfect generators. This follows
from the fact that the functor $H_\a$ identifies the $\a$-compact
objects with the $\a$-presentable projective objects of $A_\a(\T)$.
\end{rem}

\subsection{Notes}
Well generated triangulated were introduced and studied by Neeman in
his book \cite{Nee2001} as a natural generalization of compactly
generated triangulated categories. For an alternative approach which
simplifies the definition, see \cite{Kr2001}.  More recently, well
generated categories with specific models have been studied; see
\cite{Por,Tab} for work involving algebraic models via differential
graded categories, and \cite{Hei} for topological models. In
\cite{Ros2005}, Rosick\'y used combinatorial models and showed that
there exist universal cohomological functors into locally presentable
categories which are full. Interesting consequences of this fact are
discussed in \cite{Nee2007}.  The description of the universal
cohomological functors in terms of filtered colimits seems to be new.

\section{Localization for well generated categories}
\label{se:loc-wellgen}

\subsection{Cohomological localization}

The following theorem shows that cohomological functors on well
generated triangulated categories induce localization functors. This
generalizes a fact which is well known for compactly generated
triangulated categories.

\begin{thm}
\label{th:cohloc}
Let $\T$ be a triangulated category with small coproducts which is
well generated. Let $H\colon \T\to \A$ be a cohomological functor into
an abelian category which has small coproducts and exact $\a$-filtered
colimits for some regular cardinal $\a$. Suppose also that $H$
preserves small coproducts. Then there exists an exact localization
functor $L\colon\T\to\T$ such that for each object $X$ we have $LX=0$
if and only if $H(S^nX)=0$ for all $n\in\bbZ$.
\end{thm}
\begin{proof}
We may assume that $\T$ is $\a$-well generated. Let $\Si=\Si(H)$
denote the set of morphisms $\s$ in $\T$ such $H\s$ is
invertible. Next we assume that $S\Si=\Si$. Otherwise, we replace $\A$
by a countable product $\A^\bbZ$ of copies of $\A$ and $H$ by
$(HS^n)_{n\in\bbZ}$.  Then $\Si$ admits a calculus of right fractions
by Lemma~\ref{le:coh}, and we apply the criterion of
Lemma~\ref{le:frac-set} to show that the morphisms between any two
objects in $\T[\Si^{-1}]$ form a small set. The existence of a localization
functor $L\colon\T\to\T$ with $\Ker L=\Ker H$ then follows from
Proposition~\ref{pr:loc-brown}.

Thus we need to specify for each object $Y$ of $\T$ a small set of
objects $S(Y,\Si)$ such that for every morphism $X\to Y$ in $\Si$,
there exist a morphism $X'\to X$ in $\Si$ with $X'$ in
$S(Y,\Si)$. Suppose that $Y$ belongs to $\T^\k$. We define by
induction $\k_{-1}=\k+\a$ and
$$\k_{n}=\sup\{\card\T^\a/U\mid
U\in\T^{\k_{n-1}}\}^++\k_{n-1}\quad\text{for}\quad n\geq 0.$$ Then we
put $S(Y,\Si)=\T^{\bar\k}$ with $\bar\k=(\sum_{n\geq 0}\k_n)^+$.

Now fix $\s\colon X\to Y$ in $\Si$.  The morphism $X'\to X$ in $\Si$
with $X'$ in $S(Y,\Si)$ is constructed as follows. The canonical
morphism $\pi\colon \coprod_{(C,\m)\in\T^{\a}/X}C\to X$ induces an
epimorphism $H\pi$ by Theorem~\ref{th:filtcolim}. We can choose
$\C\subseteq \T^{\a}/X$ with $\card\C\leq\card \T^{\a}/Y$ such that
$\pi_0\colon X_0=\coprod_{(C,\m)\in\C}C\to X$ induces an epimorphism
$H\pi_0$ since $H\s$ is invertible. More precisely, we call two
objects $(C,\m)$ and $(C',\m')$ of $\T^\a/X$ equivalent if
$\s\m=\s\m'$, and we choose as objects of $\C$ precisely one
representative for each equivalence class.

Suppose we have already constructed $\pi_i\colon
X_i\to X$ with $X_i$ in $\T^{\k_i}$ for some $i\geq 0$. Then we
form the following commutative diagram with exact rows.
$$\xymatrix{U_i\ar[r]^{\iota_i}\ar[d]^{\s_i}&X_i\ar@{=}[d]\ar[r]^{\pi_i}&
X\ar[d]^\s\ar[r]&SU_i\ar[d]^{S\s_i}\\
V_i\ar[r]&X_i\ar[r]&Y\ar[r]&SV_i}$$ Note that $H\s_i$ is
invertible. Thus we can choose $\C_i\subseteq \T^{\a}/U_i$ with
$$\card\C_i\leq\card \T^{\a}/V_i\leq\card \T^{\a}/X_i +\card\T^{\a}/Y<\k_{i+1}$$
such that $\xi_i\colon \coprod_{(C,\m)\in\C_i}C\to U_i$ induces an
epimorphism $H\xi_i$.  Now complete $\iota_i\comp\xi_i$ to an exact
triangle and define $\pi_{i+1}\colon X_{i+1}\to X$ by the
commutativity of the following diagram.
$$\xymatrixcolsep{1pc}
\xymatrix{\coprod_{(C,\m)\in\C_i}C\ar[rr]^-{\iota_i\comp\xi_i}&&
X_i\ar[rr]^{\p_i}\ar[rd]_{\pi_i}&& X_{i+1}\ar[ld]^{\pi_{i+1}}\ar[rr]
&&S\big(\coprod_{(C,\m)\in\C_i}C\big)\\ &&&X}$$ Observe that $X_{i+1}$
belongs to $\T^{\k_{i+1}}$ and that $\Ker H\pi_i=\Ker
H\p_i$. The $\p_i$ induce an exact triangle
\begin{equation}\label{eq:hclim}
\coprod_{i\in\bbN}X_i\xto{(\id-\p_i)}\coprod_{i\in\bbN}X_i\lto[\psi]
X'\lto S\big(\coprod_{i\in\bbN}X_i\big)
\end{equation} 
such that $X'$ belongs to $S(Y,\Si)$ and the morphism
$(\pi_i)\colon\coprod_{i\in\bbN}X_i\to X$ factors through $\psi$ via a
morphism $\t\colon X'\to X$. We claim that $H\t$ is invertible. In
fact, the lemma below implies that the $\pi_i$ induce the following
exact sequence $$0\lto\coprod_{i\in\bbN}HX_i\xto{(\id-H\p_i)}
\coprod_{i\in\bbN}HX_i\xto{(H\pi_i)}HX\lto 0.$$ On the other hand, the
exact triangle \eqref{eq:hclim} induces the exact sequence
$$H\big(\coprod_{i\in\bbN}X_i\big)\xto{H(\id-\p_i)}H\big(\coprod_{i\in\bbN}X_i\big)\lto[H\psi]
HX'\lto HS\big(\coprod_{i\in\bbN}X_i\big),$$ and a comparison shows
that $H\t$ is invertible. Here, we use again that $H$ preserves small
coproducts, and this completes the proof.
\end{proof}

\begin{lem}
Let $\A$ be an abelian category which admits countable coproducts.
Then a sequence of epimorphisms $(\pi_i)_{i\in\bbN}$
$$\xymatrixcolsep{1pc}
\xymatrix{X_i\ar[rr]^{\p_i}\ar[rd]_{\pi_i}&&X_{i+1}\ar[ld]^{\pi_{i+1}}\\&Y
}$$ satisfying $\pi_i=\pi_{i+1}\comp\p_i$ and $\Ker\pi_i=\Ker\p_i$ for
all $i$ induces an exact sequence
$$0\lto\coprod_{i\in\bbN}X_i\xto{(\id-\p_i)}\coprod_{i\in\bbN}X_i\xto{(\pi_i)}Y\lto 0.$$
\end{lem}
\begin{proof}
The assumption $U_i:=\Ker\pi_i=\Ker\p_i$ implies that there exists a
morphism $\pi'_i\colon Y\to X_i$ with $\pi_i\pi'_i=\id Y$ and
$\p_i\pi'_i=\pi'_{i+1}$ for all $i\geq 1$. Thus we have a sequence of commuting squares
$$\xymatrix{U_i\amalg
Y\ar[d]^{\smatrix{0&0\\0&\id}}\ar[rr]^-{\smatrix{\inc&\pi'_i}}&&
X_i\ar[d]^{\p_i}\\ U_{i+1}\amalg
Y\ar[rr]^-{\smatrix{\inc&\pi'_{i+1}}}&&X_{i+1}}$$ where the horizontal
maps are isomorphisms. Taking colimits on both sides, the assertion
follows.
\end{proof}

\subsection{Localization with respect to a small set of objects}

Let $\T$ be a well generated triangulated category and $\Sc$ a
localizing subcategory which is generated by a small set of objects. The
following result says that $\Sc$ and $\T/\Sc$ are both well generated
and that the filtration $\T=\bigcup_\a\T^\a$ via $\a$-compact objects
induces canonical filtrations
$$\Sc=\bigcup_\a(\Sc\cap\T^\a)\quad\text{and}\quad
\T/\Sc=\bigcup_\a\T^\a/(\Sc\cap\T^\a).$$

\begin{thm}
\label{th:set-local}
Let $\T$ be a well generated triangulated category and $\Sc$ a
localizing subcategory which is generated by a small set of objects.
Fix a regular cardinal $\a$ such that $\T$ is $\a$-well generated and
$\Sc$ is generated by $\a$-compact objects.
\begin{enumerate}
\item An object $X$ in $\T$ belongs to $\Sc$ if and only if every
morphism $C\to X$ from an object $C$ in $\T^\a$ factors through some
object in $\Sc\cap\T^\a$.
\item The localizing subcategory $\Sc$ and the quotient category
$\T/\Sc$ are $\a$-well generated.
\item We have $\Sc^\a=\Sc\cap\T^\a$ and a commutative diagram of exact
functors
$$\xymatrix{\Sc^\a\ar[d]^-\inc\ar[rr]^-\inc&&\T^\a\ar[d]^-\inc\ar[rr]^-\can&&
\T^\a/\Sc^\a\ar[d]^-J\\ \Sc\ar[rr]^-{\inc}&&\T\ar[rr]^-{\can}&&\T/\Sc}$$
such that $J$ is fully faithful. Moreover, $J$ induces a functor
$\T^\a/\Sc^\a\to (\T/\Sc)^\a$ such that every object of $(\T/\Sc)^\a$ is
a direct factor of an object in the image of $J$.  This functor is an
equivalence if $\a>\aleph_0$.
\end{enumerate}
\end{thm}
\begin{proof}
Let $\C=\Sc\cap\T^\a$. Then the inclusion $i\colon\C\to\T^\a$ induces a
fully faithful and exact functor
$i^*\colon\Add_\a(\C^\op,\Ab)\to\Add_\a((\T^\a)^\op,\Ab)$ which is
left adjoint to the functor $i_*$ taking $F$ to $F\comp i$; see
Lemma~\ref{le:kan}.  Note that the image $\Im i^*$ of $i^*$ is closed
under small coproducts. We consider the restricted Yoneda functor
$h_\a\colon\T\to\Add_\a((\T^\a)^\op,\Ab)$ taking $X$ to
$\T(-,X)|_{\T^\a}$ and observe that $h_\a^{-1}(\Im i^*)$ is a
localizing subcategory of $\T$ containing $\C$. Thus we obtain a
functor $H$ making the following diagram commutative.
$$\xymatrix{ \Sc\ar[rr]^\inc\ar[d]^H&&\T\ar[d]^{h_\a}\\
\Add_\a(\C^\op,\Ab)\ar[rr]^{i^*}&&\Add_\a((\T^\a)^\op,\Ab) }$$ 
Let us compare $H$ with the restricted Yoneda functor
$$H'\colon\Sc\lto\Add_\a(\C^\op,\Ab),\quad X\mapsto\Sc(-,X)|_\C.$$
In fact, we have an isomorphism
$$H\xto{\sim}i_*\comp i^*\comp H= i_*\comp h_\a|_\Sc=H'$$ and $H'$
preserves small coproducts since $h_\a$ does. It follows from
Proposition~\ref{pr:wellgen} that $\C$ provides a small set of
$\a$-small perfect generators for $\Sc$. Thus $\Sc$ is $\a$-well
generated and $\Sc^\a=\Loc_\a\C=\Sc\cap\T^\a$.

Next we apply Proposition~\ref{pr:loc-brown} and obtain a localization
functor $L\colon \T\to\T$ with $\Ker L=\Sc$. We use $L$ to show that
$\Sc= h_\a^{-1}(\Im i^*)$. We know already that $\Sc\subseteq
h_\a^{-1}(\Im i^*)$. Now let $X$ be an object in $h_\a^{-1}(\Im i^*)$
and consider the exact triangle $\Ga X\to X\to LX\to S(\Ga X)$. Then
$\T(C,LX)=0$ for all $C\in\C$ and therefore $i_*h_\a LX=0$. On the
other hand, $h_\a LX=i^*F$ for some functor $F$ and therefore
$0=i_*h_\a LX=i_*i^*F\cong F$. Thus $LX=0$ and therefore $X$ belongs
to $\Sc$. This shows $\Sc=h_\a^{-1}(\Im i^*)$.

Now we prove (1) and use the description of the essential image of
$i^*$ from Lemma~\ref{le:kan}. We have for an object $X$ in $\T$ that
$X$ belongs to $\Sc$ iff $h_\a X$ belongs to $\Im i^*$ iff every
morphism $\T^\a(-,C)\to h_\a X$ with $C\in\T^\a$ factors through
$\T^\a(-,C')$ for some $C'\in \C$ iff every morphism $C\to X$ with
$C\in\T^\a$ factors through some $C'\in \C$.

An immediate consequence of (1) is the fact that $J$ is fully
faithful. This follows from Lemma~\ref{le:tria-sub}.

Now consider the quotient functor $q\colon\T^\a\to\T^\a/\Sc^\a$. This
induces an exact functor
$q^*\colon\Add_\a((\T^\a)^\op,\Ab)\to\Add_\a((\T^\a/\Sc^\a)^\op,\Ab)$
which is left adjoint to the fully faithful functor $q_*$ taking $F$
to $F\comp q$; see Lemma~\ref{le:kan}. Clearly, $q^*\comp h_\a$
annihilates $\Sc$ and induces therefore a functor $K$ making the
following diagram commutative.
$$\xymatrix{ \T\ar[rr]^{Q=\can}\ar[d]^{h_\a}&&\T/\Sc\ar[d]^{K}\\
\Add_\a((\T^\a)^\op,\Ab)\ar[rr]^{q^*}&&\Add_\a((\T^\a/\Sc^\a)^\op,\Ab)
}$$ Note that $Q$ admits a right adjoint which we denote by $Q_\r$. We
identify $\T^\a/\Sc^\a$ via $J$ with a full triangulated subcategory of
$\T/\Sc$ and consider the restricted Yoneda functor $$K'\colon\T/\Sc\lto
\Add_\a((\T^\a/\Sc^\a)^\op,\Ab),\quad
X\mapsto\T/\Sc(-,X)|_{\T^\a/\Sc^\a}.$$ Adjointness gives the following
isomorphism
$$\T/\Sc(JqC,X)=\T/\Sc(QC,X)\cong\T(C,Q_\r X)$$ for all $C\in\T^\a$ and
$X\in\T/\Sc$.  Thus we have an isomorphism
$$K\xleftarrow{\sim}K\comp Q\comp Q_\r=q^*\comp h_\a\comp Q_\r\cong
q^*\comp q_*\comp K'\xto{\sim} K'$$ and $K'$ preserves small
coproducts since $h_\a$ does. It follows from
Proposition~\ref{pr:wellgen} that $\T^\a/\Sc^\a$ provides a small set of
$\a$-small perfect generators for $\T/\Sc$. Thus $\T/\Sc$ is $\a$-well
generated and $(\T/\Sc)^\a=\Loc_\a(\T^\a/\Sc^\a)$.
\end{proof}

\begin{cor}\label{co:loc-gen}
Let $\T$ be an $\a$-well generated triangulated category and $\Sc$ a
localizing subcategory generated by a small set $\Sc_0$ of $\a$-compact
objects. Then $\Sc$ is $\a$-well generated and $\Sc^\a$ equals the
$\a$-localizing subcategory generated by $\Sc_0$.
\end{cor}
\begin{proof}
In the preceding proof of Theorem~\ref{th:set-local}, we can choose
for $\C$ instead of $\Sc\cap\T^\a$ the $\a$-localizing subcategory of
$\T$ which is generated by $\Sc_0$. Then the proof shows that $\C$
provides a small set of $\a$-small perfect generators for $\Sc$. Thus we have
$\Sc^\a=\C$ by definition.
\end{proof}

The localization with respect to a localizing subcategory generated by
a small set of objects can be interpreted in various ways. The following
remark provides some indication.

\begin{rem}
(1) Let $\T$ be a well generated triangulated category and $\p$ a
morphism in $\T$.  Then there exists a universal exact localization
functor $L\colon\T\to\T$ inverting $\p$. To see this, complete $\p$ to
an exact triangle $X\xto{\p} Y\to Z\to SX$ and let $L$ be the
localization functor such that $\Ker L$ equals the localizing
subcategory generated by $Z$.  Conversely, any exact localization
functor $L\colon\T\to\T$ is the universal exact localization functor
inverting some morphism $\p$ provided that $\Ker L$ is generated by a
small set $\Sc_0$ of objects. To see this, take $\p\colon
0\to\coprod_{X\in\Sc_0}X$.

(2) Let $\T$ be a triangulated category and $L\colon\T\to\T$ an exact
localization functor such that $\Sc=\Ker L$ is generated by a single
object $W$. Then the first morphism $\Ga X\to X$ from the functorial
triangle $\Ga X\to X\to LX\to S(\Ga X)$ is called \emph{cellularization} and the
second morphism $X\to LX$ is called \emph{nullification} with respect
to $W$. The objects in $\Sc$ are \emph{built from} $W$.
\end{rem}

\subsection{Functors between well generated categories}

We consider functors between well generated triangulated categories
which are exact and preserve small coproducts.  The following result
shows that such functors are controlled by their restriction to the
subcategory of $\a$-compact objects for some regular cardinal $\a$.

\begin{prop}
\label{pr:fun-well}
Let $F\colon\T\to\U$ be an exact functor between $\a$-well generated
triangulated categories. Suppose that $F$ preserves small coproducts
and let $G$ be a right adjoint.
\begin{enumerate}
\item There exists a regular cardinal $\b_0\geq \a$ such that $F$
preserves $\b_0$-compactness. In that case $F$ preserves $\b$-compactness
for all regular $\b\geq\b_0$.
\item Given a regular cardinal $\b\geq\b_0$, the restriction
 $f\colon\T^\b\to\U^\b$ of $F$ induces the following diagram of
 functors which commute up to natural isomorphisms.
$$\xymatrix{ \T^\b\ar[rr]^{f=F^\b}\ar[d]^{\inc}&&\U^\b\ar[d]^{\inc}\\
\T\ar[rr]^F\ar[d]^{h_\b(\T)}&&\U\ar[d]^{h_\b(\U)}
\ar[rr]^G&&\T\ar[d]^{h_\b(\T)}\\
\Add_\b((\T^\b)^\op,\Ab)\ar[rr]^{f^*}&&\Add_\b((\U^\b)^\op,\Ab)
\ar[rr]^{f_*}&&\Add_\b((\T^\b)^\op,\Ab) }$$ 
\end{enumerate}
\end{prop}
\begin{proof}
(1) Choose $\b_0\geq\a$ such that
$F(\T^\a)\subseteq\U^{\b_0}$. Then we get for $\b\geq\b_0$
$$F(\T^\b)= F(\Loc_\b\T^\a)\subseteq \Loc_\b
F(\T^\a)\subseteq\Loc_\b\U^{\b_0}=\U^\b.$$

(2) We apply Theorem~\ref{th:filtcolim} to show that $h_\b(\U)\comp
F\cong f^*\comp h_\b(\T)$. In fact, it follows from
Proposition~\ref{pr:universal} and Lemma~\ref{le:kan} that both
composites are cohomological functors, preserve small coproducts,
and agree on $\T^\b$.

The isomorphism $h_\b(\T)\comp G\cong f_*\comp h_\b(\U)$ follows from
the adjointness of $F$ and $G$, since $\T(C,GX)\cong\U(fC,X)$ for
every $C\in\T^\b$ and $X\in\U$.
\end{proof}

\subsection{The kernel of a functor between well generated categories}

We show that the class of well generated triangulated categories is
closed under taking kernels of exact functors which preserve small
coproducts.

\begin{thm}
\label{th:kernel}
Let $F\colon\T\to\U$ be an exact functor between $\a$-well generated
triangulated categories and suppose that $F$ preserves small
coproducts. Let $\Sc=\Ker F$ and choose
a regular cardinal $\b\geq\a$ such that $F$ preserves $\b$-compactness.
\begin{enumerate}
\item An object $X$ in $\T$ belongs to $\Sc$ if and only if every
morphism $C\to X$ with $C\in\T^\b$ factors through a morphism
$\g\colon C\to C'$ in $\T^\b$ satisfying $F\g=0$.
\item Suppose $\b>\aleph_0$. Then $\Sc$ is $\b$-well generated and
$\Sc^\b=\Sc\cap\T^\b$.
\end{enumerate}
\end{thm}
\begin{proof}
Let $f\colon\T^\b\to\U^\b$ be the restriction of $F$ and denote by
$\mathfrak I$ the set of morphisms in $\T^\b$ which are annihilated
by $F$. 

(1) Let $X$ be an object in $\T$. Then it follows from
Proposition~\ref{pr:fun-well} that $FX=0$ if and only $f^*h_\b X=0$.
Now Lemma~\ref{le:kan} implies that $f^*h_\b X=0$ iff each
morphism $C\to X$ with $C\in\T^\b$ factors through some morphism $C\to
C'$ in $\mathfrak I$.

(2) Let $\Sc'$ denote the localizing subcategory of $\T$ which is
generated by all homotopy colimits of sequences
$$C_0\lto C_1\lto C_2\lto \cdots$$ of morphisms in $\mathfrak I$.  We
claim that $\Sc'=\Ker F$. Clearly, we have $\Sc'\subseteq\Ker F$. Now
fix an object $X\in\Ker F$. We have seen in (1) that each morphism
$\m\colon C\to X$ with $C\in\T^\b$ factors through some morphism $C\to
C'$ in $\mathfrak I$.  We obtain by induction a sequence
$$C=C_0\lto[\g_0]C_1\lto[\g_1]C_2\lto[\g_2]\cdots$$ of morphisms in
$\mathfrak I$ such that $\m$ factors through each finite composite
$\g_i\ldots\g_0$. Thus $\m$ factors through the homotopy colimit of
this sequence and therefore through an object of $\Sc'\cap\T^\b$. Here
one uses that $\b>\aleph_0$. We conclude from
Theorem~\ref{th:set-local} that $X$ belongs to $\Sc'$. Moreover, we
conclude from this theorem that $\Sc'$ is $\b$-well generated.
\end{proof}

\begin{rem}
It is necessary to assume in part (2) of the preceding theorem that
$\b>\aleph_0$. For example, there exists a ring $A$ with Jacobson
radical $\mathfrak r$ such that the functor
$F=-\otimes^\bfL_AA/\mathfrak r\colon \bfD (A)\to\bfD (A/\mathfrak r)$
satisfies $\Sc=\Ker F\neq 0$ but $\Sc\cap\bfD (A)^c=0$; see \cite{Kel1994a}.
\end{rem}

Observe that Theorem~\ref{th:kernel} provides a partial answer to the
telescope conjecture for compactly generated categories. This
conjecture claims that the kernel of a localization functor
$L\colon\T\to\T$ is generated by compact objects provided that $L$
preserves small coproducts. Part (1) implies that $\Sc=\Ker L$ is
generated by morphisms between compact objects, and part (2) says that
$\Sc$ is generated by $\aleph_1$-compact objects. I am grateful to
Amnon Neeman for explaining to me how to deduce (2) from (1). The
following corollary makes the connection with the telescope conjecture
more precise; just put $\a=\aleph_0$.

\begin{cor}
Let $L\colon\T\to\T$ be an exact localization functor which preserves
small coproducts. Suppose that $\T$ is $\a$-well generated and let
$\b\geq\max(\a,\aleph_1)$. Then $\Sc=\Ker L$ is $\b$-well generated and
$\Sc^\b=\Sc\cap\T^\b$.
\end{cor}
\begin{proof}
Let $L\colon\T\to\T$ be a localization functor which preserves small
coproducts. Write $L=G\comp F$ as the composite of a quotient
functor $F\colon \T\to\U$ and a fully faithful right adjoint $G$,
where $\U=\T/\Sc$ and $\Sc=\Ker L$. Then $G$ preserves small coproducts
by Proposition~\ref{pr:recol}. The isomorphism \eqref{eq:adj-coprod}
from the proof of Lemma~\ref{le:adj-coprod} shows that $F$ preserves
$\a$-smallness and sends a set of perfect generators of $\T$ to a set
of perfect generators of $\U$. In particular, $F$ preserves
$\b$-compactness for all regular $\b\geq \a$. Now apply
Theorem~\ref{th:kernel}.
\end{proof}

\subsection{The kernel of a cohomological functor on a well generated category}

The following result says that kernels of cohomological functors from
well generated triangulated categories into locally presentable
abelian categories are well generated. The argument is basically the
same as that for kernels of exact functors between well generated
triangulated categories.

\begin{thm}
\label{th:kernel-coh}
Let $H\colon\T\to\A$ be a cohomological functor from a well generated
triangulated category into a locally presentable abelian category and
suppose that $H$ preserves small coproducts. Let $\Sc$ denote the
localizing subcategory of $\T$ consisting of all objects $X$ such that
$H(S^nX)=0$ for all $n\in\bbZ$. Then $\Sc$ is a well generated
triangulated category.
\end{thm}
\begin{proof}
Replacing $H$ by $(HS^n)_{n\in\bbZ}$, we may assume that $\Sc=\Ker H$.
Choose a regular cardinal $\a$ such that $\T$ is $\a$-well generated
and $\A$ is locally $\a$-presentable. Then we have
$H(\T^\a)\subseteq\A^\b$ for some regular cardinal $\b$ and we assume
$\b\geq\max(\a,\aleph_1)$.  The description of $H$ in
Theorem~\ref{th:filtcolim} shows that $H$ restricts to a functor
$h\colon\T^\b\to\A^\b$, and we denote by $\bar h\colon
A(\T^\b)\to\A^\b$ the induced exact functor.  Then we obtain the
following functor
$$h^*\colon\Add_\b((\T^\b)^\op,\Ab)\xto{\sim}
\Lex_\b(A(\T^\b)^\op,\Ab)\xto{\bar
h^*}\Lex_\b((\A^\b)^\op,\Ab)\xto{\sim}\A$$ where the first equivalence
follows from Lemma~\ref{le:Add} and the second equivalence follows from
Lemma~\ref{le:kan-inc}. The functor $\bar h^*$ is a left Kan
extension; it takes a filtered colimit $$F=\colim{(C,\m)\in
A(\T^\b)/F}A(\T^\b)(-,C)\quad\text{to}\quad\colim{(C,\m)\in
A(\T^\b)/F}\A^\b(-,\bar hC).$$ Note that $h^*$ is exact and preserves
small coproducts. This follows from Lemma~\ref{le:yon-exact} and the
fact that $\bar h^*$ is left adjoint to the restriction functor $\bar
h_*$. 

The composite $h^*\comp h_\b\colon\T\to\A$ coincides with $H$ on $\T^\b$ and
therefore $h^*\comp h_\b\cong H$ by Theorem~\ref{th:filtcolim}. In particular,
we have for each $X$ in $\T$ that $HX=0$ if and only if
$h^*(h_\b X)=0$. Now we use the same argument as in the proof of
Theorem~\ref{th:kernel} and show that $\Ker H$ is generated by all
homotopy colimits of countable sequences of morphisms in
$\T^\b$ which are annihilated by $H$.
\end{proof}

\subsection{Localization of well generated categories versus abelian localization}

We demonstrate the interplay between triangulated and abelian
localization.  To this end recall from Proposition~\ref{pr:universal}
that we have for each well generated category $\T$ a universal
cohomological functor $H_\a\colon\T\to A_\a(\T)$ into a locally
$\a$-presentable abelian category. We show that each exact localization
functor for $\T$ can be extended to an exact localization functor
for $A_\a(\T)$ for some regular cardinal $\a$.

\begin{thm}
Let $\T$ be a well generated triangulated category and
$L\colon\T\to\T$ an exact localization functor.  Suppose that $\Ker L$
is well generated. Then there exists a regular cardinal $\a$ and an
exact localization functor $L'\colon A_\a(\T)\to A_\a(\T)$ such that
the following square commutes up to a natural isomorphism.
$$\xymatrix{ \T\ar[rr]^-{L}\ar[d]^-{H_\a}&&\T\ar[d]^-{H_\a}\\
A_\a(\T)\ar[rr]^-{L'}&&A_\a(\T) }$$ More precisely, the adjunction morphisms
$\Id\T\to L$ and $\Id A_\a(\T)\to L'$ induce for each $X$ in $\T$ the following
isomorphisms.
$$H_\a LX\stackrel{\sim}\longrightarrow L'(H_\a
LX)=L'H_\a(LX)\stackrel{\sim}\longleftarrow L'H_\a(X)$$ An object $X$ in
$\T$ is $L$-acyclic if and only if $H_\a X$ is $L'$-acyclic, and $X$ is
$L$-local if and only if $H_\a X$ is $L'$-local. 
\end{thm}
\begin{proof}
Choose a regular cardinal $\a>\aleph_0$ such that $\T$ is $\a$-well
generated and $\Sc=\Ker L$ is generated by $\a$-compact objects. Let
$\U=\T/\Sc$ and write $L=G\comp F$ as the composite of the quotient
functor $F\colon\T\to\U$ with its right adjoint $G\colon\U\to\T$.

Now identify $A_\a(\T)=\Add_\a((\T^\a)^\op,\Ab)$ and
$A_\a(\U)=\Add_\a((\U^\a)^\op,\Ab)$.  The induced functor
$f\colon\T^\a\to\U^\a$ equals, up to an equivalence, the quotient
functor $\T^\a\to\T^\a/\Sc^\a$, by Theorem~\ref{th:set-local}. From $f$
we obtain a pair of adjoint functors $f^*$ and $f_*$ by
Lemma~\ref{le:kan}. Both functors are exact and the right adjoint
$f_*$ is fully faithful. Thus we obtain an exact localization functor
$L'=f_*\comp f^*$ for $A_\a(\T)$ by Corollary~\ref{co:loc}.  The
commutativity $H_\a\comp L\cong L'\comp H_\a$ and the assertions about
acyclic and local objects then follow from
Proposition~\ref{pr:fun-well}.
\end{proof}

\subsection{Example: The derived category of an abelian Grothendieck category}

Let $\A$ be an abelian Grothendieck category. Then the derived
category $\bfD(\A)$ of unbounded chain complexes is a well generated
triangulated category. Let us sketch an argument. 

The Popescu-Gabriel theorem says that for each generator $G$ of $\A$,
the functor $T=\A(G,-)\colon\A\to\Mod A$ (where $A=\A(G,G)$ denotes
the endomorphism ring of $G$) is fully faithful and admits an exact
left adjoint, say $Q$; see \cite[Theorem~X.4.1]{St}.  Consider the
cohomological functor $H\colon\bfD(A)\to\A$ taking a complex $X$ to
$Q(\coprod_{n\in\bbZ}H^nX)$.  Then an application of
Theorem~\ref{th:kernel-coh} shows that $\Sc=\Ker H$ is well generated,
and therefore $\bfD(A)/\Sc$ is well generated by
Theorem~\ref{th:set-local}.

Next observe that $\bfK(Q)$ induces an equivalence
$$\bfK(\Mod A)/(\Ker \bfK(Q))\lto[\sim]\bfK(\A)$$ since $\bfK(Q)$ has
$\bfK(T)$ as a fully faithful right adjoint.  Moreover, the cohomology
of each object in the kernel of $\bfK(Q)$ lies in the kernel of
$Q$. Thus we obtain the following commutative diagram.
$$\xymatrix{\Ker\bfK(Q)\ar[rr]^-\inc\ar[d] &&\bfK(\Mod
A)\ar[d]^\can\ar[rr]^-{\bfK(Q)}&&\bfK(\A)\ar[d]^F\\
\Sc\ar[d]\ar[rr]^-\inc&&\bfD(A)\ar[d]^{H^*}\ar[rr]^-\can&&\bfD(A)/\Sc\ar[d]^{\bar
H}\\ \Ker Q\ar[rr]^-\inc&& \Mod A\ar[rr]^-Q&&\A }$$ It is easily
checked that the kernel of $F$ consists of all acyclic complexes. Thus
$F$ induces an equivalence $\bfD(\A)\xto{\sim}\bfD(A)/\Sc$.

\subsection{Notes}
Given a triangulated category $\T$, there are two basic questions when
one studies exact localization functors $\T\to\T$. One can ask for the
existence of a localization functor with some prescribed kernel, and
one can ask for a classification, or at least some structural results,
for the set of all localization functors on $\T$. Well generated
categories provide a suitable setting for some partial answers.

The fact that cohomological functors induce localization functors is
well known for compactly generated triangulated categories
\cite{Mar1983}, but the result seems to be new for well generated
categories.  The localization theorem which describes the localization
with respect to a small set of objects is due to Neeman
\cite{Nee2001}. The example of the derived category of an abelian
Grothendieck category is discussed in \cite{Gal,Nee2001a}. The
description of the kernel of an exact functor between well generated
categories seems to be new. A motivation for this is the telescope
conjecture which is due to Bousfield and Ravenel \cite{Bou,Rav1984}
and originally formulated for the stable homotopy category of
CW-spectra.

It is interesting to note that the existence of localization functors
depends to some extent on axioms from set theory; see for instance
\cite{CSS,CGR}.

\section{Epilogue: Beyond well generatedness}

Well generated triangulated categories were introduced by Neeman as a
class of triangulated categories which includes all compactly
generated categories and behaves well with respect to localization. We
have discussed in Sections~\ref{se:wellgen} and \ref{se:loc-wellgen}
most of the basic properties of well generated categories but the
picture is still not complete because some important questions remain
open. For instance, given a well generated triangulated category $\T$,
we do not know when a localizing subcategory arises as the kernel of a
localization functor and when it is generated by a small set of
objects. Also, one might ask when the set of all localizing
subcategories is small. Another aspect is Brown representability. We
do know that every cohomological functor $\T^\op\to\Ab$ preserving
small products is representable, but what about covariant functors
$\T\to\Ab$?  It seems that one obtains more insight by studying the
universal cohomological functors $\T\to A_\a(\T)$; in particular we
need to know when they are full; see \cite{Ros2005,Nee2007} for some
recent work in this direction.

Instead of answering these open questions, let us be adventurous and
move a little bit beyond the class of well generated categories. In
fact, there are natural examples of triangulated categories which are
not well generated. Such examples arise from additive categories by
taking their homotopy category of chain complexes. More precisely, let
$\A$ be an additive category and suppose that $\A$ admits small
coproducts. We denote by $\bfK(\A)$ the category of chain complexes in
$\A$ whose morphisms are the homotopy classes of chain maps.  Take for
instance the category $\A=\Ab$ of abelian groups. Then one can show
that $\bfK(\Ab)$ is not well generated; see \cite{Nee2001}. In fact,
more is true. The category $\bfK(\Ab)$ admits no small set of
generators, that is, any localizing subcategory generated by a small
set of objects is a proper subcategory. However, it is not difficult
to show that any localizing subcategory generated by a small set of
objects is well generated. So we may think of $\bfK(\Ab)$ as
\emph{locally well generated}. In fact, discussions with Jan
\v{S}\v{t}ov\'i\v{c}ek suggest that $\bfK(\A)$ is locally well
generated whenever $\A$ is locally finitely presented; see
\cite{Sto}. Recall that $\A$ is \emph{locally finitely presented} if
$\A$ has filtered colimits and there exists a small set of finitely
presented objects $\A_0$ such that every object can be written as the
filtered colimit of objects in $\A_0$. On the other hand, $\bfK(\A)$
is only generated by a small set of objects if $\A=\Add\A_0$ for some
small set of objects $\A_0$. Here, $\Add\A_0$ denotes the smallest
subcategory of $\A$ which contains $\A_0$ and is closed under taking
small coproducts and direct summands. We refer to \cite{Sto} for further details.

\begin{appendix}

\section{The abelianization of a triangulated category}
\label{se:abel}

Let $\C$ be an additive category.  We consider functors
$F\colon\C^\op\to\Ab$ into the category of abelian groups and call a
sequence $F'\to F\to F''$ of functors {\em exact} if the induced
sequence $F'X\to FX\to F''X$ of abelian groups is exact for all $X$ in
$\C$. A functor $F$ is said to be {\em coherent} if there exists an
exact sequence (called {\em presentation})
$$\C(-,X)\lto \C(-,Y)\lto F\lto 0.$$ The morphisms between two
coherent functors form a small set by Yoneda's lemma, and the coherent
functors $\C^\op\to\Ab$ form an additive category with cokernels. We
denote this category by $\widehat\C$.

A basic tool is the fully faithful {\em Yoneda functor} $h_\C\colon
\C\to\widehat\C$ which sends an object $X$ to $\C(-,X)$.  One might
think of this functor as the completion of $\C$ with
respect to the formation of finite colimits. To formulate some further properties, we
recall that a morphism $X\to Y$ is a {\em weak kernel} for a morphism
$Y\to Z$ if the induced sequence $\C(-,X)\to \C(-,Y)\to \C(-,Z)$ is
exact.

\begin{Lem}\label{le:coherent}
Let $\C$ be an additive category.
\begin{enumerate}
\item Given an additive functor $H\colon\C\to\A$ to an additive
category which admits cokernels, there is (up to a unique isomorphism)
a unique right exact functor $\bar H\colon \widehat\C\to\A$ such that
$H=\bar H\comp h_\C$.
\item If $\C$ has weak kernels, then $\widehat \C$ is an abelian
category.
\item If $\C$ has small coproducts, then $\widehat \C$ has small
coproducts and the Yoneda functor preserves small coproducts.
\end{enumerate}
\end{Lem}
\begin{proof} 
(1) Extend $H$ to $\bar H$ by sending $F$ in $\widehat\C$ with presentation
$$\C(-,X)\lto[(-,\p)] \C(-,Y)\lto F\lto 0$$ to the cokernel of
$H\p$. 

(2) The category $\widehat\C$ has cokernels, and it is therefore
sufficient to show that $\widehat\C$ has kernels. To this end fix a
morphism $F_1\to F_2$ with the following presentation.
$$\xymatrix{\C(-,X_1)\ar[r]\ar[d]&\C(-,Y_1)\ar[r]\ar[d]&F_1\ar[r]\ar[d]&0\\
\C(-,X_2)\ar[r]&\C(-,Y_2)\ar[r]&F_2\ar[r]&0}$$ We construct the kernel
$F_0\to F_1$ by specifying the following presentation.
$$\xymatrix{\C(-,X_0)\ar[r]\ar[d]&\C(-,Y_0)\ar[r]\ar[d]&F_0\ar[r]\ar[d]&0\\
\C(-,X_1)\ar[r]&\C(-,Y_1)\ar[r]&F_1\ar[r]&0}$$
First the morphism $Y_0\to Y_1$ is obtained from the weak
kernel sequence
$$Y_0\lto X_2\amalg Y_1\lto Y_2.$$
Then the morphisms $X_0\to
X_1$ and $X_0\to Y_0$ are obtained from the weak kernel sequence
$$X_0\lto X_1\amalg Y_0\lto Y_1.$$

(3) For every family of functors $F_i$ having a presentation
$$\C(-,X_i)\lto[(-,\p_i)] \C(-,Y_i)\lto F_i\lto 0,$$
the coproduct $F=\coprod_i F_i$ has a presentation
$$\C(-,\coprod_iX_i)\lto[(-,\amalg\p_i)] \C(-,\coprod_iY_i)\lto F\lto 0.$$
Thus coproducts in $\widehat\C$ are not computed pointwise.
\end{proof}

The assigment $\C\mapsto\widehat\C$ is functorial in the following
weak sense. Given a functor $F\colon\C\to\D$, there is (up to a unique
isomorphism) a unique right exact functor $\widehat F\colon
\widehat\C\to\widehat\D$ extending the composite $h_\D\comp F\colon
\C\to\widehat \D$.


Now let $\T$ be a triangulated category.  Then we write
$A(\T)=\widehat\T$ and call this category the {\em abelianization} of
$\T$, because the Yoneda functor $\T\to A(\T)$ is the universal
cohomological functor for $\T$.

\begin{Lem}\label{le:abel}
Let $\T$ be a triangulated category.  Then the category $A(\T)$ is
abelian and the Yoneda functor $h_\T\colon \T\to A(\T)$ is
cohomological.  
\begin{enumerate}
\item Given a cohomological functor $H\colon\T\to\A$ to an abelian
category, there is (up to a unique isomorphism) a unique exact functor
$\bar H\colon A(\T)\to\A$ such that $H=\bar H\comp h_\T$.
\item Given an exact functor $F\colon\T\to\T'$ between triangulated
categories, there is (up to a unique isomorphism) a unique exact
functor $A(F)\colon A(\T)\to A(\T')$ such that $h_{\T'}\comp F =A(F)\comp h_\T$.
\end{enumerate}
\end{Lem}
\begin{proof} 
The category $\T$ has weak kernels and therefore $A(\T)$ is
abelian. Note that the weak kernel of a morphism $Y\to Z$ is obtained
by completing the morphism to an exact triangle $X\to Y\to Z\to SX$.

(1) Let $H\colon\T\to\A$ be a cohomological functor and let $\bar
H\colon A(\T)\to\A$ be the right exact functor extending $H$ which
exists by Lemma~\ref{le:coherent}. Then $\bar H$ is exact because $H$
is cohomological.

(2) Let $F\colon\T\to\T'$ be exact. Then $H=h_{\T'}\comp F$ is a
cohomological functor and we let $A(F)=\bar H$ be the exact functor
which extends $H$.
\end{proof}

The assignment $\T\mapsto A(\T)$ from triangulated categories to
abelian categories preserves various properties of
exact functors between triangulated categories. Let us mention some of them.

\begin{Lem}
Let $F\colon\T\to\T'$ and $G\colon\T'\to\T$ be exact functors between
triangulated categories.
\begin{enumerate}
\item $F$ is fully faithful if and only if $A(F)$ is fully faithful. 
\item If $F$ induces an equivalence $\T/\Ker F\xto{\sim}\T'$, then
$A(F)$ induces an equivalence $A(\T)/(\Ker A(F))\xto{\sim}A(\T')$.
\item $F$ preserves small (co)products if and only if $A(F)$ preserves small
(co)products.
\item $F$ is left adjoint to $G$ if and only if $A(F)$ is left adjoint
to $A(G)$.
\end{enumerate}
\end{Lem}
\begin{proof}
Straightforward.
\end{proof}

\subsection*{Notes}
The abelianization of a triangulated category appears in Verdier's
th\`ese \cite{V} and in Freyd's work on the stable homotopy category
\cite{Fr}. Note that their construction is slightly different from the
one given here, which is based on coherent functors in the sense of
Auslander \cite{A}.

\section{Locally presentable abelian categories}
\label{se:locpres}

Fix a regular cardinal $\a$ and a  small additive category
$\C$ which admits $\a$-coproducts. We denote by $\Add(\C^\op,\Ab)$ the
category of additive functors $\C^\op\to\Ab$ into the category of
abelian groups. This is an abelian category which admits small
(co)products. In fact, (co)kernels and (co)products are computed
pointwise in $\Ab$. Given functors $F$ and $G$ in $\Add(\C^\op,\Ab)$,
we write $\Hom_\C(F,G)$ for the set of morphisms $F\to G$. The most
important objects in $\Add(\C^\op,\Ab)$ are the \emph{representable
functors} $\C(-,X)$ with $X\in\C$. Recall that Yoneda's lemma provides
a bijection
$$\Hom_\C(\C(-,X),F)\lto[\sim]FX$$ for all $F\colon\C^\op\to\Ab$ and
$X\in\C$.

We denote by $\Add_\a(\C^\op,\Ab)$ the full subcategory of
$\Add(\C^\op,\Ab)$ which is formed by all functors preserving
$\a$-products. This is an exact abelian subcategory, because kernels
and cokernels of morphism between $\a$-product preserving functors
preserve $\a$-products. In particular, $\Add_\a(\C^\op,\Ab)$ is an
abelian category.

Now suppose that $\C$ admits cokernels. Then $\Lex_\a(\C^\op,\Ab)$
denotes the full subcategory of $\Add(\C^\op,\Ab)$ which is formed by
all left exact functors preserving $\a$-products. This category is
locally presentable in the sense of Gabriel and Ulmer and we refer to
\cite[\S5]{GU} for an extensive treatment. In this appendix we collect
some basic facts.

First observe that $\a$-filtered colimits in $\Lex_\a(\C^\op,\Ab)$ are
computed pointwise.  This follows from the fact that in $\Ab$ taking
$\a$-filtered colimits commutes with taking $\a$-limits; see
\cite[Satz~5.12]{GU}.  In particular, $\Lex_\a(\C^\op,\Ab)$ has small
coproducts because every small coproduct is the $\a$-filtered colimit
of its subcoproducts with less than $\a$ factors.

Next we show that one can identify $\Add_\a(\C^\op,\Ab)$ with a
category of left exact functors. To this end consider the Yoneda
functor $h_\C\colon\C\to\widehat\C$ taking $X$ to $\C(-,X)$.

\begin{Lem}\label{le:Add}
Let $\C$ be a  small additive category with
$\a$-coproducts. Then the Yoneda functor induces an equivalence
$$\Lex_\a(\widehat\C^\op,\Ab)\lto[\sim]\Add_\a(\C^\op,\Ab)$$ by taking
a functor $F$ to $F\comp h_\C$.
\end{Lem}
\begin{proof}
Use that every additive functor
$\C^\op\to\Ab$ extends uniquely to a left exact functor
$\widehat\C^\op\to\Ab$; see Lemma~\ref{le:coherent}.
\end{proof}

From now on we assume that $\C$ admits $\a$-coproducts and
cokernels. Given any additive functor $F\colon\C^\op\to\Ab$, we consider the
category $\C/F$ whose objects are pairs $(C,\m)$ consisting of an
object $C\in\C$ and an element $\m\in FC$. A morphism
$(C,\m)\to(C',\m')$ is a morphism $\p\colon C\to C'$ such that
$F\p(\m')=\m$.

\begin{Lem}\label{le:lex}
Let $F\colon\C^\op\to\Ab$ be an additive functor.
\begin{enumerate}
\item The canonical morphism
$$\colim{(C,\m)\in\C/F}\C(-,C)\lto F$$ in $\Add(\C^\op,\Ab)$ is an
isomorphism.
\item The functor $F$ belongs to $\Lex_\a(\C^\op,\Ab)$ if
and only if the category $\C/F$ is $\a$-filtered.
\end{enumerate}
\end{Lem}
\begin{proof}
(1) is easy. For (2), see \cite[Satz~5.3]{GU}.
\end{proof}

The representable functors in $\Lex_\a(\C^\op,\Ab)$ share the
following finiteness property. Recall that an object $X$ from an
additive category $\A$ with $\a$-filtered colimits is
\emph{$\a$-presentable} if the representable functor
$\A(X,-)\colon\A\to\Ab$ preserves $\a$-filtered colimits.  Next
observe that the inclusion $\Lex_\a(\C^\op,\Ab)\to\Add(\C^\op,\Ab)$
preserves $\a$-filtered colimits. This follows from the fact that in
$\Ab$ taking $\a$-filtered colimits commutes with taking
$\a$-limits. This has the following consequence.

\begin{Lem}\label{le:presentable}
For each $X$ in $\C$, the representable functor $\C(-,X)$ is an
$\a$-presentable object of $\Lex_\a(\C^\op,\Ab)$.
\end{Lem}
\begin{proof}
Combine Yoneda's lemma with the fact that the inclusion 
$\Lex_\a(\C^\op,\Ab)\to\Add(\C^\op,\Ab)$ preserves $\a$-filtered
colimits.
\end{proof}

There is a general result for the category $\Lex_\a(\C^\op,\Ab)$ which
says that taking $\a$-filtered colimits commutes with taking
$\a$-limits; see \cite[Korollar~7.12]{GU}. Here we need the following
special case.

\begin{Lem}
\label{le:ex-colim}
Suppose the category $\Lex_\a(\C^\op,\Ab)$ is abelian. Then an
$\a$-filtered colimit of exact sequences is again exact.
\end{Lem}
\begin{proof}
We need to show that taking $\a$-filtered colimits commutes with
taking kernels and cokernels.  A cokernel is nothing but a colimit and
therefore taking colimits and cokernels commute.  The statement about
kernels follows from the fact that the inclusion
$\Lex_\a(\C^\op,\Ab)\to \Add(\C^\op,\Ab)$ preserves kernels and
$\a$-filtered colimits. Thus we can compute kernels and $\a$-filtered
colimits in $\Add(\C^\op,\Ab)$ and therefore in the category $\Ab$ of
abelian groups. In $\Ab$ it is well known that taking kernels and
filtered colimits commute.
\end{proof}

\begin{Lem}
\label{le:yon-exact}
Suppose that $\C$ is abelian. Then $\Lex_\a(\C^\op,\Ab)$ is abelian
and the Yoneda functor $h_\C\colon\C\to \Lex_\a(\C^\op,\Ab)$ is
exact. Given an abelian category $\A$ which admits small coproducts
and exact $\a$-filtered colimits, and given a functor
$F\colon\Lex_\a(\C^\op,\Ab)\to\A$ preserving $\a$-filtered colimits,
we have that $F$ is exact if and only if $F\comp h_\C$ is exact.
\end{Lem}
\begin{proof}
We use the analogue of Lemma~\ref{le:lex} for morphisms which says
that each morphism $\p$ in $\Lex_\a(\C^\op,\Ab)$ can be written as
$\a$-filtered colimit $\p=\colim{i\in\C/\p}\p_i$ of morphisms between
representable functors. Thus one computes
$$\Coker\p=\colim{i\in\C/\p}\Coker\p_i\quad\text{and}\quad
\Ker\p=\colim{i\in\C/\p}\Ker\p_i,$$ and we see that
$\Lex_\a(\C^\op,\Ab)$ is abelian; see Lemma~\ref{le:ex-colim}. The
formula for kernels and cokernels shows that each exact sequence can
be written as $\a$-filtered colimit of exact sequences in the image of
the Yoneda embedding. The criterion for the exactness of a functor
$\Lex_\a(\C^\op,\Ab)\to\A$ is an immediate consequence.
\end{proof}

Let $\A$ be a cocomplete additive category. We denote by $\A^\a$ the
full subcategory which is formed by all $\a$-presentable
objects. Following \cite{GU}, the category $\A$ is called
\emph{locally $\a$-presentable} if $\A^\a$ is  small and
each object is an $\a$-filtered colimit of $\a$-presentable objects.
We call $\A$ \emph{locally presentable} if it is locally
$\b$-presentable for some cardinal $\b$.  Note that we have for each
locally presentable category $\A$ a filtration $\A=\bigcup_\b\A^\b$
where $\b$ runs through all regular cardinals. We have already seen
that $\Lex_\a(\C^\op,\Ab)$ is locally $\a$-presentable, and the next
lemma implies that, up to an equivalence, all locally $\a$-presentable
categories are of this form.

Let $f\colon\C\to \A$ be a fully faithful and right exact functor into
a cocomplete additive category. Suppose that $f$ preserves
$\a$-coproducts and that each object in the image of $f$ is
$\a$-presentable. Then $f$ induces the functor
$$f_*\colon\A\lto \Lex_\a(\C^\op,\Ab),\quad X\mapsto\A(f-,X),$$ and the
following lemma discusses its left adjoint.

\begin{Lem}
\label{le:kan-inc}
There is a fully faithful functor $f^*\colon\Lex_\a(\C^\op,\Ab)\to\A$
which sends each representable functor $\C(-,X)$ to $fX$ and
identifies $\Lex_\a(\C^\op,\Ab)$ with the full subcategory of $\A$
formed by all colimits of objects in the image of $f$. The functor
$f^*$ is a left adjoint of $f_*$.
\end{Lem}
\begin{proof}
The functor is the left Kan extension of $f$; it
takes $F=\colim{(C,\m)\in\C/F}\C(-,C)$ in $\Lex_\a(\C^\op,\Ab)$ to
$\colim{(C,\m)\in\C/F}fC$ in $\A$. We refer to \cite[Satz~7.8]{GU} for details.
\end{proof}

Suppose now that $\C$ is a triangulated category.  The following lemma
characterizes the cohomological functors $\C^\op\to\Ab$.

\begin{Lem}\label{le:flat}
Let $\C$ be a  small triangulated category and suppose $\C$
admits $\a$-coproducts. For a functor $F$ in $\Add_\a(\C^\op,\Ab)$ the
following are equivalent.
\begin{enumerate}
\item The category $\C/F$ is $\a$-filtered.
\item $F$ is an $\a$-filtered colimit of representable functors.
\item $F$ is a cohomological functor.
\end{enumerate}
\end{Lem}
\begin{proof}
The implications (1) $\Rightarrow$ (2) $\Rightarrow$ (3) are clear. So
we prove (3) $\Rightarrow$ (1). It is convenient to identify
$\Add_\a(\C^\op,\Ab)$ with $\Lex_\a(\widehat\C^\op,\Ab)$ and this identifies
$F$ with the left exact functor $\bar F\colon\widehat \C\to \Ab$ which
extends $F$.  In fact, $\bar F$ is exact since $F$ is cohomological,
by Lemma~\ref{le:abel}.  Now write $\bar F$ as $\a$-filtered colimit
of representable functors $\bar F=\colim{(M,\n)\in\widehat\C/\bar
F}\widehat\C(-,M)$; see Lemma~\ref{le:lex}.
The exactness of $\bar F$ implies that the representable functors
$\C(-,C)$ with $C\in\C$ form a full subcategory of $\widehat \C/\bar
F$ which is cofinal.  We identify this subcategory with $\C/F$ and
conclude from Lemma~\ref{le:cofinal} that $\C/F$ is $\a$-filtered.
\end{proof}

Next we discuss the functoriality of the assignment
$\C\mapsto\Add_\a(\C^\op,\Ab)$.

\begin{Lem}
\label{le:kan}
Let $f\colon\C\to\D$ be an exact functor between  small
triangulated categories which admit $\a$-coproducts. Suppose that $f$
preserves $\a$-coproducts. Then the restriction functor
$$f_*\colon\Add_\a(\D^\op,\Ab)\lto\Add_\a(\C^\op,\Ab),\quad F\mapsto
F\comp f,$$ has a left adjoint $f^*$ which sends $\C(-,X)$ to
$\D(-,fX)$ for all $X$ in $\C$. Moreover, the following holds.
\begin{enumerate}
\item The functors $f_*$ and $f^*$ are exact.
\item Suppose $f$ induces an equivalence $\C/\Ker f\xto{\sim}\D$.
Then $f_*$ is fully faithful.
\item Suppose $f$ is fully faithful. Then $f^*$ is fully
faithful. Moreover, a cohomological functor $F\colon\D^\op\to\Ab$ is
in the essential image of $f^*$ if and only if every morphism
$\D(-,D)\to F$ factors through $\D(-,fC)$ for some object $C$ in $\C$.
\item A cohomological functor $F\colon\C^\op\to\Ab$ belongs to the
kernel of $f^*$ if and only if every morphism $\C(-,C)\to F$ factors
through a morphism $\C(-,\g)\colon\C(-,C)\to\C(-,C')$ such that
$f\g=0$.
\end{enumerate}
\end{Lem}
\begin{proof}
The left adjoint of $f_*$ is the left Kan extension. We can describe
it explicitly if we identify $\Add_\a(\C^\op,\Ab)$ with
$\Lex_\a(\widehat\C^\op,\Ab)$; see Lemma~\ref{le:Add}. 
Given a functor $F$ in $\Lex_\a(\widehat\C^\op,\Ab)$ written as $\a$-filtered colimit
$F=\colim{(C,\m)\in\widehat\C/F} \widehat\C(-,C)$ of representable
functors, we put
$$f^*F=\colim{(C,\m)\in\widehat\C/F} \widehat\D(-,\widehat fC).$$ Thus
$f^*$ makes the following diagram commutative.
$$\xymatrix{\C\ar[d]^f\ar[r]^-{h_\C}&\widehat\C\ar[d]^{\widehat
f}\ar[r]^-{h_{\widehat\C}}&\Lex_\a(\widehat\C^\op,\Ab)
\ar[r]^-{=}\ar[d]^{f^*}&\Add_\a(\C^\op,\Ab)\ar[d]^{f^*}\\
\D\ar[r]^-{h_\D}&\widehat\D\ar[r]^-{h_{\widehat\D}}&\Lex_\a(\widehat\D^\op,\Ab)
\ar[r]^-{=}&\Add_\a(\D^\op,\Ab)}$$ We check that $f^*$ is a left
adjoint for $f_*$. For a representable functor $F=\widehat\C(-,X)$ we
have
\begin{align*}
\Hom_{\widehat\D}(f^*\widehat\C(-,X),G)&=\Hom_{\widehat\D}(\widehat\D(-,\widehat fX),G)
\cong G(\widehat fX)\\ &=f_*G(X)\cong
\Hom_{\widehat\C}(\widehat\C(-,X),f_*G)
\end{align*}
for all $G$ in $\Lex_\a(\widehat\D^\op,\Ab)$.  Clearly, this
isomorphism extends to every colimit of representable functors.

(1) The exactness of $f_*$ is clear because a sequence $F'\to F\to
F''$ in $\Add_\a(\C^\op,\Ab)$ is exact if and only if $F'X\to FX\to
F''X$ is exact for all $X$ in $\C$. For the exactness of $f^*$ we
identify again $\Add(\C^\op,\Ab)$ with $\Lex_\a(\widehat\C^\op,\Ab)$
and apply Lemma~\ref{le:yon-exact}.  Thus we need to check that the
composition of $f^*$ with the Yoneda functor $h_{\widehat\C}$ is
exact. But we have that $f^*\comp
h_{\widehat\C}=h_{\widehat\D}\comp\widehat f$, and now the exactness
follows from that of $f$. Finally, we use the fact that taking
$\a$-filtered colimits in $\Add_\a(\D^\op,\Ab)$ is exact by
Lemma~\ref{le:ex-colim}.

(2) It is well known that for any epimorphism $f\colon\C\to\D$ of
additive categories inducing a bijection $\Ob\C\to\Ob\D$, the
restriction functor $\Add(\D^\op,\Ab)\to\Add(\C^\op,\Ab)$ is fully
faithful; see \cite[Corollary~5.2]{Mit}.  Given a triangulated
subcategory $\C'\subseteq\C$, the quotient functor $\C\to\C/\C'$ is an
epimorphism. Thus the assertion follows since $\Add_\a(\C^\op,\Ab)$ is
a full subcategory of $\Add(\C^\op,\Ab)$.

(3) We keep our identification
$\Add_\a(\C^\op,\Ab)=\Lex_\a(\widehat\C^\op,\Ab)$ and consider the
adjunction morphism $\eta\colon\Id\to f_*\comp f^*$. We claim that
$\eta$ is an isomorphism. Because $f$ is fully faithful, $\eta F$ is
an isomorphism for each representable functor $F=\widehat\C(-,X)$. It
follows that $\eta F$ is an isomorphism for all $F$ since $f^*$ and
$f_*$ both preserve $\a$-filtered colimits and each $F$ can be
expressed as $\a$-filtered colimit of representable functors. Now
Proposition~\ref{pr:quot} implies that $f^*$ is fully faithful.

Let $F$ be a cohomological functor in $\Add_\a(\D^\op,\Ab)$ and apply
Lemma~\ref{le:flat} to write the functor as $\a$-filtered colimit
$F=\colim{(D,\m)\in\D/F}\D(-,D)$ of representable functors. Suppose
first that every morphism $\D(-,D)\to F$ factors through $\D(-,fC)$
for some $C\in\C$. Then $\Im f/F$ is a cofinal subcategory of $\D/F$
and therefore $F=\colim{(D,\m)\in\Im f/F}\D(-,D)$ by
Lemma~\ref{le:cofinal}. Thus $F$ belongs to the essential image of
$f^*$ since $\D(-,fC)=f^*\C(-,C)$ for all $C\in\C$ and the essential
image is closed under taking colimits. Now suppose that $F$ belongs to
the essential image of $f^*$. Then $F=f^*G\cong f^*f_*f^*G=f^*f_*F$ for some
$G$. The functor $f_*F$ is cohomological and therefore
$f_*F=\colim{(C,\m)\in\C/f_*F}\C(-,C)$, again by
Lemma~\ref{le:flat}. Thus $F\cong \colim{(C,\m)\in\C/f_*F}\D(-,fC)$
and we use Lemma~\ref{le:presentable} to conclude that each morphism
$\D(-,D)\to F$ factors through $\D(-,fC)$ for some $(C,\m)\in\C/f_*F$.

(4) Let $F$ be a cohomological functor in $\Add_\a(\C^\op,\Ab)$ and
apply Lemma~\ref{le:flat} to write the functor as $\a$-filtered
colimit $F=\colim{(C,\m)\in\C/F}\C(-,C)$ of representable
functors. Now $f^*F=\colim{(C,\m)\in\C/F}\D(-,fC)=0$ if and only if
for each $D\in\D$, we have $\colim{(C,\m)\in\C/F}\D(D,fC)=0$. This
happens iff for each $(C,\m)\in\C/F$, we find a morphism $\g\colon
C\to C'$ in $\C/F$ inducing a map $\D(fC,fC)\to\D(fC,fC')$ which
annihilates the identity morphism. But this means that $f\g=0$ and
that $\m\colon\C(-,C)\to F$ factors through $\C(-,\g)$. 
\end{proof}

\subsection*{Notes}
Locally presentable categories were introduced and studied by Gabriel
and Ulmer in \cite{GU}; see \cite{AdRo} for a modern treatment. In
\cite{Nee2001}, Neeman initiated the use of locally presentable
abelian categories for studying triangulated categories.

\end{appendix}

\end{document}